\documentclass[11pt]{amsart}
\usepackage[utf8]{inputenc}
\usepackage{amsmath,amssymb,amsthm,colonequals,mathrsfs,mathtools}
\usepackage{array,enumitem}
\usepackage{color}
\usepackage[all,cmtip]{xy}
\usepackage{booktabs}
\usepackage[colorlinks,anchorcolor=blue,citecolor=blue,linkcolor=blue,urlcolor=blue,bookmarksdepth=2]{hyperref}

\usepackage{comment}
\hypersetup{
  pdftitle={Hassett-Keel-Looijenga program for the moduli spaces of rational elliptic surfaces},
  pdfauthor={Masafumi Hattori and Yota Maeda},
  pdfkeywords={rational elliptic surfaces, ball quotients, Looijenga compactifications, Borcherds products, invariant theory, K-moduli}
}
\usepackage{multirow}
\allowdisplaybreaks
\usepackage{makecell}

\usepackage[margin=1in]{geometry}

\usepackage{tikz}
\usetikzlibrary{cd,arrows}
\tikzset{>=stealth}
\tikzcdset{arrow style=tikz}
\tikzset{link/.style={column sep=1.8cm,row sep=0.16cm}}

\AtBeginDocument{%
	\def\MR#1{}
}
\usepackage{fancyhdr}
\usepackage{here}
\usepackage{amscd} 
\usepackage{scalefnt}
\usepackage{url}
\theoremstyle{plain}
\newtheorem{thm}{Theorem}[section]
\newtheorem{lem}[thm]{Lemma}
\newtheorem{cor}[thm]{Corollary}
\newtheorem{prop}[thm]{Proposition}
\newtheorem{conj}[thm]{Conjecture}
\theoremstyle{definition}

\newtheorem{rem}[thm]{Remark}

\newtheorem{defn}[thm]{Definition}

\newtheorem{ex}[thm]{Example}
\numberwithin{equation}{section}
\def\Q{{\mathbb Q}}
\def\R{{\mathbb R}}
\def\Z{{\mathbb Z}}
\def\C{{\mathbb C}}
\def\B{{\mathbb B}}

\def\id{\mathop{\mathrm{id}}\nolimits}

\def\Pic{\mathop{\mathrm{Pic}}\nolimits}

\def\SL{\mathop{\mathrm{SL}}\nolimits}

\def\Proj{\mathop{\mathrm{Proj}}\nolimits}

\def\det{\mathop{\mathrm{det}}\nolimits}
\def\dim{\mathop{\mathrm{dim}}\nolimits}

\def\Div{\mathop{\mathrm{Div}}\nolimits}

\def\ord{\mathrm{ord}}

\def\M{\mathcal{M}}

\def\PGL{\mathop{\mathrm{PGL}}\nolimits}

\def\U{\mathrm{U}}
\newcommand{\DWei}{\Delta_{\mathrm{Wei}}}

\newcommand{\Disc}{\operatorname{Disc}}

\newcommand{\marg}[1]{\normalsize{{\color{red}\footnote{#1}}}{\marginpar[\vskip -.3cm {\color{red}\hfill\tiny{\thefootnote}\normalsize$\implies$}]{\vskip -.3cm{ \color{red}$\impliedby$\tiny\thefootnote}}}}

\newcommand{\Yota}[1]{\marg{\color{blue}(Yota): #1}}

\allowdisplaybreaks[4]

\newcommand{\defeq}{\vcentcolon=}

\fancypagestyle{titlepage}
{
	\fancyhf{}

\fancyfoot[l]{
\href{https://mathscinet.ams.org/mathscinet/msc/msc2020.html}
	{\emph{2020 Mathematics Subject Classification}}
Primary 14J27; Secondary 14D20, 14E30, 14L24, 11F55, 32N15.
		\\
		\emph{Keywords}: rational elliptic surfaces; Hassett--Keel--Looijenga program;
K-moduli; wall crossing; complex ball quotients;
Looijenga arrangement compactifications; Borcherds products.
	}
    }

\begin{document}
\title[K-Moduli Wall Crossing and Automorphic Forms for the Moduli Space of Rational Elliptic Surfaces]{K-Moduli Wall Crossing and Automorphic Forms for the Moduli Space of Rational Elliptic Surfaces}
\author{Masafumi Hattori$^{1}$}
\email{hattori.masafumi.8n@kyoto-u.ac.jp}
\author{Yota Maeda$^{2,3}$}
\email{y.maeda.math@gmail.com}

\maketitle
\thispagestyle{titlepage}
\vspace{-1em}
\begin{center}
  \begin{minipage}{0.9\textwidth}
    \centering
    {\small
    $^{1}$ Department of Mathematics, Kyoto University, Japan\\
    $^{2}$ Fachbereich Mathematik, Technische Universität Darmstadt, Germany\\
    $^{3}$ Mathematical Institute, Tohoku University, Japan\\
    }
  \end{minipage}
\end{center}
\begin{abstract}
Using K-moduli spaces for log quasimaps
$q_t\colon\left(\mathbb P^1,\frac{1-t}{12}D\right)\to[\mathbb A^2/\mathbb G_m]$ of degree twelve with twelve points and weight $t/12$, we construct a modular interpolation $\{\mathcal M_t\}_{0\le t\le1}$ between the Baily--Borel compactification of the Heckman--Looijenga ball quotient $X_o$ and Miranda's GIT compactification of the moduli space of rational elliptic surfaces. We completely determine the wall-crossing and, for every rational $t\in[0,1]$, identify
\[
\mathcal M_t\cong\operatorname{Proj}R\!\left(X_o,\mathcal L+\frac t2\Delta(6)+\frac t3\Delta(9)\right),
\]
where $\mathcal L$ is the automorphic $\mathbb Q$-line bundle and $\Delta(6),\Delta(9)$ are distinguished Heegner divisors.

On the automorphic side, we construct a new automorphic form on $X_o$ via a Borcherds product, whose divisor gives an independent relation among the Heegner divisors.
This relation provides a key input for determining the birational transformations in the K-moduli wall-crossing. As part of this analysis, we show that the first positive chamber $\mathcal M_t$ for $t\in (0,1/7)$ is Looijenga's semi-toroidal compactification. 
\end{abstract}

% \setcounter{tocdepth}{1}
% \tableofcontents

\section{Introduction}
\label{sec:introduction}

\subsection{Compactifications of the moduli space of rational elliptic surfaces}

A single moduli problem may admit projective compactifications of very
different origins.  Geometric invariant theory (GIT cf.~\cite{GIT}) constructs compactifications
by imposing a stability condition on a parameter space, whereas period maps arising from Hodge theory can often identify a moduli space with an arithmetic quotient and hence
lead to its Baily--Borel compactification \cite{baily1966compactification}.  These constructions often
produce birational, but quite different, compactifications of the same
moduli space.  Rather than only comparing the two endpoints, the
Hassett--Keel--Looijenga philosophy, formulated by Laza--O'Grady \cite{LazaOGrady2016HKL,LazaOGrady2018DoubleCover} and
inspired by the Hassett--Keel program and the work of Looijenga \cite{Looijenga2003Ball}, asks
whether one can interpolate between these compactifications through a variation of birational models, with the intermediate models carrying natural modular interpretations.

A natural way to realize such an interpolation is to vary the stability
condition defining the moduli problem.  For log Fano pairs, K-stability gives rise to projective moduli spaces of
K-semistable objects, known as K-moduli spaces. Varying the boundary coefficients can produce several K-moduli spaces; the resulting variation is known as K-moduli wall crossing.
K-moduli wall crossing has proved useful in a number of cases in the study of Hassett--Keel--Looijenga programs. For example, K-moduli wall crossing was used to describe the relevant
birational models for degree-two and quartic K3 surfaces by Ascher--DeVleming--Liu \cite{ascher2024wall,AscherDeVlemingLiu2023QuarticK3}.

\begin{comment}

A moduli space of algebro-geometric objects often carries several natural projective
compactifications, reflecting different choices of data retained under
degeneration.  Geometric invariant theory \cite{GIT} retains an embedded or
linearized presentation, Baily-Borel and toroidal compactifications
retain Hodge-theoretic and arithmetic data \cite{baily1966compactification,ash2010smooth}, and KSBA or K-moduli retain
stability of the limiting geometric object (cf.~\cite{kollar-moduli,xu2025k}).  Looijenga's arrangement
compactifications provide a particularly precise bridge between GIT and
arithmetic compactifications.  Building on this perspective, the
Hassett-Keel-Looijenga philosophy studies a one-parameter variation of
log canonical or automorphic models interpolating between such
compactifications; see
\cite{Looijenga2003Ball,LazaOGrady2016HKL,LazaOGrady2018DoubleCover}.
K-moduli realizations of this philosophy have subsequently appeared in
other period-moduli problems, notably for quartic K3 surfaces
\cite{AscherDeVlemingLiu2023QuarticK3}.

\end{comment}

Rational elliptic surfaces with sections form an especially concrete and rich test case for the Hassett--Keel--Looijenga program.  A rational Weierstrass fibration over $\mathbb P^1$
is represented by
\[
 y^2z=x^3+3f_0xz^2+2f_1z^3,
 \qquad
 (f_0,f_1)\in
 H^0(\mathbb P^1,\mathcal O(4))\oplus
 H^0(\mathbb P^1,\mathcal O(6)),
\]
and its discriminant is the binary duodecic
\[
 \Delta_{\mathrm{Wei}}(f_0,f_1)=f_0^3+f_1^2.
\]
Let $\M$ be the moduli space
of rational elliptic surfaces with reduced discriminant.
Miranda constructed a projective GIT compactification $\M^M$ of $\M$ from the
space of Weierstrass equations \cite{Miranda1981}.  The discriminant defines a rational map from $\M^M$ to the GIT quotient $\mathcal{D}^*$ of twelve unordered points on $\mathbb P^1$. By the work of Deligne--Mostow \cite{DeligneMostow1986}, $\mathcal{D}^*$ is the Baily--Borel compactification of a nine-dimensional ball quotient.
Furthermore, the normalization $\M^*$ of the image in $\mathcal{D}^*$ of the rational map coincides with the Baily--Borel compactification $X_o^{\mathrm{bb}}$ of an eight-dimensional ball quotient $X_o$ by \cite{HeckmanLooijenga2002}.
$\M^*$ defines another compactification of $\M$.
In this case, $\mathcal{M}^M$ coincides with the Looijenga compactification $\widehat X_o^{\mathcal K}$ \cite{Looijenga2003Ball}, which is obtained from the
Baily--Borel compactification by a birational modification controlled by certain
distinguished divisors, and can also be described in terms of
meromorphic automorphic forms with prescribed poles along them.
Here $\mathcal K$ denotes the hyperplane arrangement on the ball
corresponding to $\Delta(6)\cup\Delta(9)$.
More precisely, $X_o$ has four irreducible Heegner divisors $\Delta(6)$, $\Delta(9)$, $\Delta(15)$, and $\Delta(18)$ with the automorphic $\Q$-line bundle $\mathcal{L}$.
According to \cite{HeckmanLooijenga2002,Looijenga2003Ball}, $\mathcal{M}^*$ and $\M^M$ are the ample models of $\mathcal{L}$ and $\mathcal{L}+\frac{1}{6}\Delta(6)+\frac{1}{3}\Delta(9)$, that is, $\mathcal{M}^*=\mathrm{Proj}\,R(X_o,\mathcal{L})$ and $\mathcal{M}^M=\mathrm{Proj}\,R(X_o,\mathcal{L}+\frac{1}{6}\Delta(6)+\frac{1}{3}\Delta(9))$, respectively. 

Therefore, we can formulate the Hassett--Keel--Looijenga program in this case as follows.

\begin{conj}[Hassett--Keel--Looijenga program]\label{HKLprogram}
 For any $t\in (0,1)\cap\Q$, $R(X_o,\mathcal{L}+\frac{t}{6}\Delta(6)+\frac{t}{3}\Delta(9))$ is finitely generated. Furthermore, $\mathrm{Proj}\,R(X_o,\mathcal{L}+\frac{t}{6}\Delta(6)+\frac{t}{3}\Delta(9))$ is a projective variety admitting a modular interpretation.  
\end{conj}

Several modular compactifications related to this birational map are known. Heckman--Looijenga \cite{HeckmanLooijenga2002} constructed another modular  compactification $\mathcal{M}^K$ using the theory of stable maps \cite{kontsevich1995enumeration} dominating both compactifications, and Ascher--Bejleri \cite{AscherBejleri2018StablePairRES,AscherBejleri2021WeightedElliptic} constructed a KSBA moduli space of surface pairs with ample canonical class admitting a contraction to the graph of the birational map. These constructions clarify the birational correspondence between the two classical compactifications, but they do not provide the intermediate models predicted by Conjecture \ref{HKLprogram}.

\subsection{A modular interpolation via quasimap K-moduli}

We construct a one-parameter family of projective moduli spaces
interpolating between the Baily--Borel compactification
$\mathcal M^*$ and Miranda's compactification $\mathcal M^M$.
The construction uses K-moduli of log Fano quasimaps to vary
the relative weights of the discriminant divisor and the
quasimap associated with a Weierstrass equation.
We determine all walls and identify the resulting models
with ample models of explicit automorphic divisor classes.
% We do not prove Conjecture \ref{HKLprogram}, which concerns the specific choice of the
% ray of polarizations above.  Instead, we construct a different natural interpolation
% between the same two compactifications.

Our construction uses K-moduli wall crossing for log Fano quasimaps, which is developed in the companion work \cite{HM1}.
For a general rational elliptic surface, we associate a K-semistable log Fano quasimap
\[
q_t:
\left(\mathbb P^1,\frac{1-t}{12}D\right)
\longrightarrow [\mathbb A^2/\mathbb G_m],
\qquad \text{with quasimap weight } \frac{t}{12} \text{ for } 0\le t\le1,
\]
whose data consist of $\mathbb P^1$ with twelve weighted points together with the pencil defined by $[f_0^3:f_0^3+f_1^2]$ associated with the Weierstrass fibration.
Here we put $D:=\operatorname{div}(f_0^3+f_1^2)$.
A basic difficulty is that the two moduli $\mathcal{M}^*$ and $\mathcal{M}^M$ naturally remember different kinds of data: one is
governed by the twelve-point discriminant in $\mathbb P^1$, while the other remembers the
Weierstrass equations of surfaces.  The quasimap framework provides a common setting in which both kinds of data are encoded: the boundary records the discriminant divisor, while the quasimap data record the pencil associated with the Weierstrass fibration. Varying $t$ changes their relative weights and therefore gives a natural interpolation from $\mathcal{M}^*$ at $t=0$ to $\mathcal{M}^M$ at $t=1$.
Let $\mathcal M_t$ denote the corresponding normalization of the closure of $\mathcal{M}$ in the projective K-moduli space.
The construction above gives a modular interpolation
\[
\mathcal M_0 \cong \mathcal M^*, \qquad \mathcal M_1 \cong \mathcal M^M.
\]

The second step is to understand this interpolation on the automorphic
side.  Each $\mathcal M_t$ for rational $t$ admits the Hodge line bundle 
$\Lambda_{\mathrm{Hodge},t}$, a natural ample $\Q$-line bundle coming from the
moduli construction, and at $t=0$ this polarization agrees with the automorphic $\Q$-line bundle 
$\mathcal L$.  
Guided by the variation of $\Lambda_{\mathrm{Hodge},t}$, we consider the ray
\[
A_t:=\mathcal L+\frac t2\Delta(6)+\frac t3\Delta(9)
\]
on the ball quotient $X_o$.
Conjecture~\ref{HKLprogram} concerns a different ray and
remains open. Our aim in this paper is to compare the K-moduli wall-crossing for log Fano quasimaps with the ray $\{A_t\}$ on the ball quotient side.
To make this comparison precise, we analyze the theory of divisors of the relevant compactifications of $X_o$ using automorphic forms, in particular Borcherds products \cite{Borcherds1998Grassmannians}, together
with the geometry of its cusps.

On $\mathcal M_1 \cong \mathcal M^M$, we note that a related comparison appears in the work of
Odaka--Oshima \cite{OO21} for elliptic K3 surfaces with a
section: they identify the GIT compactification of
Weierstrass models by Miranda \cite{Miranda1981} with the Baily--Borel compactification
of the corresponding period space.
After rescaling the quasimap weight into the log Fano
region, the same GIT compactification is identified, up to
seminormalization, with the image of the Weierstrass locus
in the corresponding quasimap K-moduli space
(Remark~\ref{rem:OO}).

\subsection{Main results}
\label{sec:main-results}

We now state our main results, retaining the notation introduced above.
We also use a canonical small birational modification $X_o^{j_\mathcal K}\to X_o^{\mathrm{bb}}$ of the Baily--Borel compactification at its cusps, introduced by Looijenga \cite{Looijenga2003Ball} and known as the semi-toroidal modification. In our setting, the strict transforms of the two distinguished Heegner divisors $\Delta(6)$ and $\Delta(9)$ become $\mathbb Q$-Cartier on this modification.
 Our first main result makes precise the relation between the models $\{\mathcal M_t\}$ arising from K-moduli of log Fano quasimaps and the divisors $\{A_t\}$ on the ball quotient $X_o$.

\begin{thm}[{Theorems \ref{thm:wallcrossing-miranda-to-8points}, \ref{thm:Hodge-line-bundle-stricttrans},
and Corollary \ref{cor-first-model}}]
\label{thm:main-HKL}
The bi-graded ring
\[
 R\left(X_o;A_0,A_1\right)
 =R\left(X_o;
 \mathcal L,
 \mathcal L+\frac12\Delta(6)+\frac13\Delta(9)\right)
\]
is finitely generated over $\C$.
% \Yota{In this paper, we gave a proof of the finite generation only for rational numbers. Is this enough for this theorem?} \Masa{I put the corresponding proof.}
For every
$t\in[0,1]\cap\Q$, there is a natural isomorphism
 \[\M_t\cong\Proj R(X_o,A_t).\]
In particular,
\[
 \M_0\cong\M^*\cong X_o^{\mathrm{bb}},
 \qquad
 \M_t\cong X_o^{j_{\mathcal K}}
 \quad\left(0<t<\frac17\right),
 \qquad
 \M_1\cong\M^M\cong\widehat X_o^{\mathcal K}.
\]
Moreover, for every rational $t\ge1$ one has natural isomorphisms of
section rings
\[
 R(X_o,A_t)
 \cong
 R\left(X_o,\mathcal L+\frac16\Delta(6)+\frac13\Delta(9)\right).
\]
Consequently, the ample model is constant for $t\geq1$ and equals Miranda's compactification
$\mathcal{M}^M$.
\end{thm}

Theorem~\ref{thm:main-HKL} provides compatible modular
and automorphic descriptions of the entire interpolation.
In particular, it gives Looijenga's semi-toroidal
compactification a modular interpretation as the model
associated with the first positive chamber of the
quasimap K-moduli wall crossing.

Another useful feature of the quasimap framework is that all the models $\mathcal M_t$ are dominated by a single modular compactification. Let $\mathcal M^K$ denote the Heckman--Looijenga compactification introduced above. The general wall-crossing theory of quasimaps \cite{HM1} gives, for every $0\leq t\leq1$, a projective birational morphism
\[
\rho_t:\mathcal M^K\longrightarrow \mathcal M_t,
\]
where $\rho_0$ and $\rho_1$ coincide with the morphisms constructed in \cite{HeckmanLooijenga2002} (cf.~Theorem \ref{thm:hm1-fromstable}).
The boundary strata of $\mathcal M^K$ were explicitly classified by \cite{HeckmanLooijenga2002}; we use their notation, such as $\mathcal M^K(I_9)$ and $\mathcal M^K(I^*_{4,4})$. For a boundary stratum $\mathcal M^K(\bullet)$, we write
\[
\mathcal M_t(\bullet):=\rho_t(\mathcal M^K(\bullet))
\]for its image in $\mathcal M_t$.
Tracking these boundary strata gives the following wall-crossing diagram, in which the morphisms between models are made explicit.

\begin{thm}[{K-moduli wall-crossing; Theorem \ref{thm:wallcrossing-miranda-to-8points}  and Corollary \ref{cor-first-model}}]
\label{thm:main-wall-crossing}
Choose arbitrary real numbers
\[
 s_1\in\left(0,\frac17\right),\qquad
 s_2\in\left(\frac17,\frac14\right),\qquad
 s_3\in\left(\frac14,\frac13\right),\qquad
 s_4\in\left(\frac13,1\right).
\]
Then there are projective birational morphisms between projective
normal good moduli spaces fitting into the diagram
\[
  \xymatrix{
  & \M_{s_1}\ar[ld]^{\phi_1} \ar[rd]^{\psi_1}
  && \M_{s_2}\ar[ld]^{\phi_2} \ar[rd]^{\psi_2}
  && \M_{s_3}\ar[ld]^{\phi_3} \ar[rd]^{\psi_3}
  && \M_{s_4}\ar[ld]^{\phi_4} \ar[rd]^{\psi_4}& \\
  \M_0 && \M_{\frac17} && \M_{\frac14}
  && \M_{\frac13} && \M_1 .
  }
\]
Each $\M_{s_i}$ depends only on the chamber containing $s_i$.
Under the identifications of Theorem \ref{thm:main-HKL},
$\phi_1$ coincides with the semi-toroidal modification
$X_o^{j_{\mathcal K}}\to X_o^{\mathrm{bb}}$.
The pair $(\psi_1,\phi_2)$ gives the flip associated with the $I_7$-stratum, while $(\psi_2,\phi_3)$ gives the two flips
associated with the $I_8'$ and $I_8''$-strata.  The morphism
$\psi_3$ is the divisorial contraction of $\mathcal{M}_{s_3}(I_9)$, which is the strict transform of $\Delta(6)$,
$\phi_4$ is an isomorphism, and $\psi_4$ is the divisorial
contraction of $\mathcal{M}_{s_4}(I_{4,4}^*)$, which is the strict transform of $\Delta(9)$, to $\mathcal{M}_1$.
Although $\phi_4$ is an isomorphism of the good moduli spaces, the corresponding stack structures are different.
 All other boundary strata $\mathcal{M}_t(\bullet)$ are
unchanged at their generic points as $t$ varies in $(0,1)$.
\end{thm}

We note that our $\{\mathcal{M}_t\}$ does not dominate both $\mathcal{M}^*$ and $\mathcal{M}^M$, and hence any $\mathcal{M}_t$ does not coincide with $\M^K$ \cite{HeckmanLooijenga2002} or the KSBA moduli space of Ascher--Bejleri \cite{AscherBejleri2021WeightedElliptic,AscherBejleri2018StablePairRES}.

The key link between the K-moduli wall crossing and the ray $\{A_t\}$ of divisors is provided by the Hodge polarization $\Lambda_{\mathrm{Hodge},t}$ on the models $\{\mathcal{M}_t\}$. We compute its variation as follows.

\begin{comment}

A key technical simplification in determining the K-moduli wall-crossing is that one can track the images of the fixed boundary strata of $\mathcal M^K$ as $t$ varies.

This is Theorem \ref{thm:wallcrossing-miranda-to-8points} together with
Remark \ref{rem:flipping contraction} and the semi-toroidal identification of Corollary
\ref{cor-first-model}.
The divisor governing this wall crossing is precisely the Hodge
polarization.

\end{comment}

\begin{thm}[Variation of the Hodge polarization; Theorem \ref{thm:Hodge-line-bundle-stricttrans}]
\label{thm:main-hodge}
For every $t\in[0,1]\cap\Q$, let $\widetilde{A_t}$ denote the strict
transform of $A_t$ on $\M_t$.  Then, the ample Hodge line bundle on
$\M_t$ satisfies $\Lambda_{\mathrm{Hodge},t}\sim_{\Q}\widetilde{A_t}$.
In particular, under the identifications
$\M_0\cong X_o^{\mathrm{bb}}$ and $\M_1\cong\M^M$, we have
$\Lambda_{\mathrm{Hodge},0}\sim_{\Q}\mathcal L$ and
$\Lambda_{\mathrm{Hodge},1}\sim_{\Q}\eta$,
where $\eta$ denotes the rational GIT polarization on $\M^M$
induced by $\mathcal O_{\mathbb P(2^5,3^7)}(1)$.
Thus the Hodge polarizations interpolate between
the automorphic polarization on the Baily--Borel compactification
and the GIT polarization on Miranda's compactification.
\end{thm}

Theorem \ref{thm:main-hodge} is proved in Theorem
\ref{thm:Hodge-line-bundle-stricttrans}; the coefficients of $\Lambda_{\mathrm{Hodge},t}$ along $\Delta(6)$ and $\Delta(9)$ are computed by
the families of log quasimaps in Examples \ref{ex:I9} and \ref{ex:I44*}.
Together with the ampleness of the Hodge polarization, this is the key input identifying the models $\mathcal{M}_t$ arising from K-moduli with the ample models of combinations of the automorphic line bundle and the Heegner divisors in Theorem~\ref{thm:main-HKL}.

% We finally record the arithmetic features of the endpoint models that
% are most directly relevant to the birational program.
% \Yota{Memo for myself: refer to automorphic forms}
% \begin{thm}[{Theorem \ref{thm:picard-numbers-corrected}
% and Corollaries \ref{cor:Miranda-class-group-corrected}
% and \ref{cor:BB-not-Q-Gorenstein}}]
% \label{prop:main-arithmetic}
%  The Baily-Borel compactification
% $X_o^{\mathrm{bb}}$ is neither $\Q$-factorial nor
% $\Q$-Gorenstein.  By contrast, Miranda's compactification $\M^M$
% is $\Q$-factorial and klt and has Picard number one.
% \end{thm}
% \Yota{Shall we refer to a mistake in Dali's paper? $\to$ Masa, please see the explanation below and the remark.}
% \Masa{We have to add another Borcherd's relation and sell it because it was asked by Bruinier. (Yota) How about the following? BTW, do you remember what did Jan exactly say?}
A key input on the automorphic side is the control of the divisor classes of the Heegner divisors. One relation is obtained by restricting Allcock's automorphic form \cite{Allcock2000Leech}. To obtain an independent relation, motivated by a question of Jan Hendrik Bruinier, we construct, via a Borcherds product, a meromorphic automorphic form on the eight-dimensional ball whose divisor yields a second independent relation among the Heegner divisors. 
Together with the geometry at the cusps, these relations give the divisor-theoretic control needed for the birational analysis above. We use $\Delta(d)$ also for the closure of the corresponding reduced divisor in $X_o^{\mathrm{bb}}$.

\begin{thm}[{Theorems \ref{thm:allcock-restriction}, \ref{thm:another Borcherds lift},
and Lemma \ref{lem:two-automorphic-relations-corrected}}]
\label{thm:main-borcherds}
There exists a meromorphic automorphic form $\Psi_{D,o}$
of weight $108$ with respect to the arithmetic subgroup $\U(A_o)$ defining $X_o$
obtained by restricting a Borcherds product.
Together with the restriction of Allcock's automorphic form, it yields
the following two independent relations on
$X_o^{\mathrm{bb}}$:
\begin{align*}
 2\Delta(6)+2\Delta(9)+3\Delta(15)+\Delta(18)
 &\sim_{\Q}132\mathcal L,\\
 54\Delta(6)+30\Delta(9)-3\Delta(15)+3\Delta(18)
 &\sim_{\Q}108\mathcal L.
\end{align*}
\end{thm}

The input for the additional Borcherds product is obtained
from the vector-valued theta series of $A_2(2)$ by applying
two Serre derivatives and dividing by the modular
discriminant; see \eqref{eq:gA-gD-definitions}.
Combining the two divisor relations with the geometry at
the cusps determines the rational divisor classes and Picard
groups of the compactifications.

\begin{cor}[{Theorem \ref{thm:picard-numbers-corrected}
and Corollaries \ref{cor:Miranda-class-group-corrected}
and \ref{cor:BB-not-Q-Gorenstein}}]
\begin{enumerate}
    \item  The Baily-Borel compactification $X_o^{\mathrm{bb}}$ satisfies
\begin{align*}
 \mathrm{Cl}(X_o^{\mathrm{bb}})_{\Q}
 =
 \Q[\mathcal L]\oplus
 \Q[\Delta(6)]\oplus
 \Q[\Delta(9)],\qquad 
 \Pic(X_o^{\mathrm{bb}})_{\Q}=\Q[\mathcal L].
\end{align*}
It is neither $\Q$-factorial nor $\Q$-Gorenstein.

\item The semi-toroidal compactification
$X_o^{j_{\mathcal K}}$ is $\Q$-factorial and satisfies
\[ \Pic(X_o^{j_{\mathcal{K}}})_{\Q}=\Q[\lambda_{\rm aut}]\oplus\Q[D_6]\oplus\Q[D_9].\]
Here $\lambda_{\rm aut}$ is the pullback of $\mathcal L$
under $X_o^{j_{\mathcal K}}\to X_o^{\mathrm{bb}}$, and
$D_d$ is the strict transform of $\Delta(d)$.
\item 
Miranda's compactification $\M^M$ is $\Q$-factorial and klt
and has Picard number one.
\end{enumerate}
\end{cor}

%The klt property of $\M^M$ is incompatible with the non-lc
%assertion in \cite[Corollary 6.2]{Shen2025MMP}.
% We discuss the relation with other results in
% %We discuss this comparison in
% Remark \ref{rem:comparison-shen}.
The images of $\Delta(6)$ and $\Delta(9)$ form a basis of $\mathrm{Cl}(X_o^{\mathrm{bb}})_{\Q}/\Pic(X_o^{\mathrm{bb}})_{\Q}$.
Their strict transforms on $X_o^{j_{\mathcal K}}$ give the two relative Picard
directions used in the family $A_t$.
Thus these calculations explain how the divisor
variation defining our interpolation of compactifications of $\M$ arises.

% The local $\Q$-Cartier criterion at the two cusps is given in
% Proposition \ref{prop:qcartier-at-cusps-corrected}.
% The exceptional divisors and Picard numbers of the full
% arrangement blow-up are determined in
% Proposition \ref{prop:picard-arrangement-blowup}.

% The stronger arithmetic statements are established in Section
% \ref{sec:semitoroidal}.  Lemma
% \ref{lem:two-automorphic-relations-corrected} gives the two Borcherds
% divisor relations, Proposition
% \ref{prop:qcartier-at-cusps-corrected} gives the local
% $\Q$-Cartier criterion at the two cusps, and Theorem
% \ref{thm:picard-numbers-corrected} computes the full rational class
% and Picard groups.  Proposition
% \ref{prop:picard-arrangement-blowup} further determines the exceptional
% divisors and Picard numbers of the full arrangement blow-up. 

\subsection*{Organization of the paper}

Section \ref{sec:preliminaries} recalls good moduli spaces, the relevant
K-stability and quasimaps.  Section \ref{sec:moduli-res} reviews Miranda's GIT compactification,
the discriminant compactification $\mathcal{M}^*$, and the Kontsevich compactification $\mathcal{M}^K$
of $\mathcal{M}$.  Section \ref{sec:Heckman--Looijenga's ball quotient}
introduces the ball quotient and its arrangement divisors,
compares the automorphic and Hodge polarizations, and
describes the cusp geometry and Looijenga's compactification.
Section \ref{sec:automorphic-forms}
computes the ramification indices and canonical divisors
and constructs two Borcherds products.  Section \ref{sec:semitoroidal} determines the geometry of several birational models of ball quotients.  Section \ref{sec:K-moduli wall crossing for quasimaps and the HKL program} computes the K-moduli
wall crossing and proves the comparison of the Hodge line bundle from which
Theorem \ref{thm:main-HKL} follows.

\section*{Notations and convention}
We work over the complex number field $\mathbb{C}$ throughout this paper.

\begin{enumerate}
\item For schemes and fundamental notations, we refer to \cite{Ha}.
Here, we say that $X$ is a variety (over $\C$) if $X$ is an irreducible, reduced, and separated scheme of finite type over $\mathbb{C}$.
We mainly deal with quasi-projective schemes in this paper.
\item Let $f\colon X\to Y$ be a projective morphism of normal varieties such that $f_*\mathcal{O}_X\cong\mathcal{O}_Y$.
Let $\mathrm{Pic}(X/Y)$ and $\mathrm{Cl}(X/Y)$ denote the relative Picard group and the relative class group.
Let $\mathrm{Pic}(X/Y)_{\mathbb{Q}}$ and $\mathrm{Cl}(X/Y)_{\mathbb{Q}}$ denote the localization of $\mathbb{Z}\to\mathbb{Q}$.
If $Y=\mathrm{Spec}\,\mathbb{C}$, then we write them as $\mathrm{Pic}(X)$, $\mathrm{Cl}(X)$, $\mathrm{Pic}(X)_{\mathbb{Q}}$, and $\mathrm{Cl}(X)_{\mathbb{Q}}$. 
For more details and the definition of $\mathbb{Q}$-factoriality, see \cite{komo}.
    \item For singular fibers of smooth minimal elliptic surfaces, we follow Kodaira's notation as \cite{Miranda1981}.
    For relatively minimal elliptic surfaces $f\colon X\to C$ with only canonical singularities, we define the type of $f^{-1}(p)$ for any $p\in C$ as the type of the corresponding fiber of the minimal resolution of $X$.
    \item Let $X$ be a normal variety with an $\mathbb{R}$-divisor $\Delta$.
    Then set the following graded ring over $\mathbb{C}$
    \[
    R(X,\Delta):=\bigoplus_{m\ge0}H^0(X,\mathcal{O}_{X}(\lfloor m\Delta\rfloor)).
    \]
    For any $\Q$-divisor $\Delta_1,\Delta_2$ on $X$, we set 
    \[
    R(X;\Delta_1,\Delta_2):=\bigoplus_{m,l\ge0}H^0(X,\mathcal{O}_{X}(\lfloor m\Delta_1+l\Delta_2\rfloor)).
    \]
\item Let $X$ be a projective normal variety and let $D$ be a $\R$-Cartier $\R$-divisor.
Then we set $\mathbf{B}(D)$ as $\bigcap\mathrm{Supp}(D')$, where $D'$ runs over all effective $\R$-divisors $\R$-linearly equivalent to $D$.
We call $\mathbf{B}(D)$ the {\it stable fixed locus} of $D$.
\item If we want to impose the affine scheme structure of $\mathbb{C}^n$, then we write it as $\mathbb{A}^n$.
\item For $\mathbb{P}^1$, we use the canonical affine coordinate of $\mathbb{A}^1\cong \mathbb{P}^1\setminus\{[1:0]\}$.
We say that $\infty:=[1:0]\in \mathbb{P}^1$.
\end{enumerate}

\section{Preliminaries}
\label{sec:preliminaries}

\subsection{Good moduli space}

In this paper, we will use the terminologies of stacks.
Let $\mathsf{Sch}$ be the category of schemes over $\mathbb{C}$.
    We say that $\mathscr{F}$ is a {\it fibered category} if $\mathscr{F}$ is a fibered category over $\mathsf{Sch}$ with a functor $p\colon \mathscr{F}\to \mathsf{Sch}$ as defined in \cite[Definition 3.1.1]{olsson2016algebraic}.
    For any scheme $X$ over $\mathbb{C}$, which can be regarded as an object of $\mathsf{Sch}$, we can set the category $\mathscr{F}(X)$ in the same way as \cite[3.1.2]{olsson2016algebraic}. 
    If $\mathscr{F}(X)$ is a groupoid for any $X$, then we say that $\mathscr{F}$ is a {\it category fibered in groupoids}.
    In this paper, we will only deal with categories fibered in groupoids $\mathscr{F}$ such that there exists a canonical category equivalence $\mathscr{F}(\coprod_{i\in I}X_i)\to \prod_{i\in I}\mathscr{F}(X_i)$ because such $\mathscr{F}$ behaves well due to \cite[Lemma 4.2.7]{olsson2016algebraic}, where $I$ is an arbitrary set.

    \begin{defn}
        Let $\mathscr{F}$ be a category fibered in groupoids.
        We say that $\mathscr{F}$ is a {\it stack} if for every object $X$ and every \'etale surjective morphism $Y\to X$, the functor
        \[
        \mathscr{F}(X)\to \mathscr{F}(Y\to X)
        \]
        is an equivalence of categories.
        Here, $\mathscr{F}(Y\to X)$ is the natural category of the descent data with respect to $Y\to X$ (cf. \cite[4.2.1]{olsson2016algebraic}). See also the equivalent definition \cite[Proposition 4.6.2]{olsson2016algebraic}.

        If $\mathscr{F}(X)$ is further a set for any $X$, then we say that $\mathscr{F}$ is an {\it \'etale sheaf}. 
    \end{defn}

\begin{defn}
Let $\mathscr{F}$ be an \'etale sheaf over the category of schemes over $\mathbb{C}$.    
We say that $\mathscr{F}$ is an {\it algebraic space} over $\mathbb{C}$ if 
\begin{enumerate}
    \item The diagonal morphism $\Delta\colon \mathscr{F}\to \mathscr{F}\times_{\mathbb{C}}\mathscr{F}$ is representable by schemes (cf. \cite[Definition 5.1.5]{olsson2016algebraic}). Equivalently, for any two morphisms $X\to \mathscr{F}$ and $Y\to \mathscr{F}$ from $\mathbb{C}$-schemes, $X\times_{\mathscr{F}}Y$ is a $\C$-scheme again (cf. \cite[Lemma 5.1.9]{olsson2016algebraic}). 
    \item there exists an \'etale surjective morphism $U\to \mathscr{F}$ from a $\mathbb{C}$-scheme. 
\end{enumerate}
\end{defn}
According to \cite[Proposition 5.2.5]{olsson2016algebraic}, an algebraic space $X$ is just a quotient of a $\mathbb{C}$-scheme $Y$ by an \'etale equivalence relation $R\hookrightarrow Y\times Y$.
Note that the category of algebraic spaces proper over $\mathbb{C}$ is equivalent to the category of compact Moishezon spaces \cite[Theorem 7.3]{artin1970algebraization}.

\begin{defn}
Let $\mathscr{F}$ be a stack.
We say that $\mathscr{F}$ is an {\it Artin stack} if 
\begin{enumerate}
    \item The diagonal morphism $\Delta\colon \mathscr{F}\to \mathscr{F}\times_{\mathbb{C}}\mathscr{F}$ is representable by algebraic spaces,
    \item there exists a smooth surjective morphism $\pi\colon U\to \mathscr{F}$ from a $\mathbb{C}$-scheme. 
\end{enumerate}
If we can further take $\pi$ as an \'etale morphism, then we say that $\mathscr{F}$ is a {\it Deligne-Mumford stack}.
\end{defn}

See \cite{olsson2016algebraic} for properties of algebraic spaces, Artin stacks and Deligne-Mumford stacks for more details. 

In this paper, we will only use quotient stacks, that is, $[X/G]$, where $X$ is a scheme of finite type over $\mathbb{C}$ and $G$ is a linear algebraic group acting on $X$.
It is well-known that $[X/G]$ is naturally an Artin stack (cf. \cite[Example 8.1.12]{olsson2016algebraic}).

\begin{defn}[\cite{Alper}]
Let $\mathfrak{X}$ be an Artin stack of finite type over $\mathbb{C}$. 
Let $\pi\colon \mathfrak{X}\to X$ be a morphism to an algebraic space.
We say that $\pi$ is the morphism of the {\it good moduli space} if
\begin{enumerate}
    \item the natural homomorphism $\mathcal{O}_X\to\pi_*\mathcal{O}_{\mathfrak{X}}$ is isomorphic, and 
    \item $\pi_*\colon\mathrm{QCoh}_{\mathfrak{X}}\to\mathrm{QCoh}_X$ is an exact functor, where $\mathrm{QCoh}_{\mathfrak{X}}$ is the Abelian category of quasi-coherent sheaves over $\mathfrak{X}$.
\end{enumerate}

For any open subset $\mathcal{U}\subset \mathfrak{X}$ of the form that $\mathcal{U}=\pi^{-1}(U)$ for some open subset $U\subset X$, we say that $\mathcal{U}$ is {\it saturated}.
\end{defn}

We note that the above $\pi$ is unique if it exists up to isomorphism and has the universal mapping property (cf. \cite[Theorem 6.6]{Alper}).

In \cite{keel1997quotients}, if $\mathfrak{X}$ is a separated Deligne-Mumford stack of fintie type over $\mathbb{C}$, then they show that there exists the coarse moduli space of $\mathfrak{X}$. We note that the coarse moduli space of $\mathfrak{X}$ is nothing but the good moduli space of $\mathfrak{X}$ in this case. It follows from the assertion of the Keel-Mori theorem \cite[Theorem 11.1.2]{olsson2016algebraic} and \cite[Proposition 11.3.4]{olsson2016algebraic}.

\begin{ex}
    Let $X=\mathrm{Spec}A$ and $G$ a reductive linear algebraic group over $\mathbb{C}$ acting on $X$, where $A$ is a $\mathbb{C}$-algebra of finite type.
    Then we can consider the categorical quotient $X/\!/G=\mathrm{Spec}A^G$, where $A^G$ is the invariant ring (see \cite[Theorem 1.1]{GIT}).
    Note that 
    \[
    [X/G]\to X/\!/G
    \]    
    is the morphism of the good moduli space (cf. \cite[Remark 13.3]{Alper}).
If an Artin stack $\mathfrak{Y}$ of finite type over $\mathbb{C}$ with affine diagonal admits a separated good moduli space $Y$ of finite type over $\C$, then \'etale locally around $Y$, the morphism of the good moduli space is isomorphic to a certain form of $[X/G]\to X/\!/G$ as above. See \cite[Theorem 4.12]{alper2020luna} for more details.

For any $\mathbb{C}$-algebra $A$, we set $\Theta_A:=[\mathbb{A}^1_A/\mathbb{G}_{m,A}]$.
If $A=\mathbb{C}$, then we write it as $\Theta$ simply.
Note that $\Theta$ consists of the image of $1\in\mathbb{A}^1$ and $0\in\mathbb{A}^1$.
We also denote them by $1\in \Theta$ and $0\in \Theta$, respectively.
\end{ex}

The equivalent conditions for the existence of good moduli spaces is given in \cite{AHLH}, which are called $\mathsf{S}$-completeness and $\Theta$-reductivity.
We will not discuss them in this paper (see \cite[Definition 2.3 and Theorem 2.4]{HM1}).

\subsection{Birational geometry}

First, we collect fundamental notions from birational geometry. For more details, see \cite{komo,birkar2010existence}.

Let $X$ be a quasi-projective normal variety with an effective $\mathbb{R}$-divisor $\Delta$ on $X$.
We say that $(X,\Delta)$ is a log pair if $K_X+\Delta$ is $\mathbb{R}$-Cartier.
In this case, for any prime divisor over $X$, we can set the following number
\[
A_{X,\Delta}(E):=\mathrm{coeff}_E(K_Y-\pi^*(K_X+\Delta))+1,
\]
where $\pi\colon Y\to X$ is a log resolution of $(X,\Delta)$ such that $E$ is a prime divisor on $Y$.
We call $A_{X,\Delta}(E)$ the log discrepancy of $E$ with respect to $(X,\Delta)$ and note that this number is independent of the choice of $\pi$.
We say that $(X,\Delta)$ is 
\begin{itemize}
    \item {\it Kawamata log terminal (klt)} if $A_{X,\Delta}(E)>0$ for any prime divisor $E$ over $X$,
    \item {\it log canonical (lc)} if $A_{X,\Delta}(E)\ge0$ for any prime divisor $E$ over $X$, and 
    \item {\it canonical} if $A_{X,\Delta}(E)\ge1$ for prime divisor $E$ over $X$ that is not a prime divisor on $X$. 
\end{itemize}

For the above classes, we can consider the log minimal model program (MMP). See \cite{birkar2010existence} for fundamental terminologies and more details.

If a quasi-projective normal variety $X$ has an effective $\R$-divisor $\Delta$ such that $(X,\Delta)$ is klt, then we say that $X$ is {\it of klt type}.
By \cite[Theorem 1]{braun2024reductive} and \cite[Corollary 1.2]{zhuang2024direct}, for any affine variety $V$ of klt type with a reductive group $G$ that acts on $V$, the categorical quotient $V/\!/G$ (cf. \cite[Definition 0.5]{GIT}) is again of klt type.

Moreover, let $Z$ be an equidimensional reduced projective scheme with only nodal singularities in codimension one. 
Then, we say that $Z$ is {\it deminormal} if $Z$ satisfies Serre's $S_2$-condition.
Suppose that $Z$ is deminormal.
Let $\nu\colon X\to Z$ be the normalization and $\mathfrak{Cond}$ the conductor scheme on $X$.
Let $\Delta_Z$ be an effective $\mathbb{R}$-divisor on $Z$ whose irreducible components are not contained in the singular locus of $Z$.
We say that $(Z,\Delta_Z)$ is {\it semi log canonical (slc)} if $(X,\nu_*^{-1}\Delta_Z+\mathfrak{Cond})$ is lc and $K_Z+\Delta_Z$ is $\mathbb{R}$-Cartier. See \cite[Section 5]{kollar2013singularities} for more details.

\subsection{K-stability of quasimaps and K-moduli}

K-stability of the log pairs is introduced by \cite{tian1997kahler} for Fano varieties and generalized by \cite{donaldson2002scalar,odaka2015testing}. 
See \cite{boucksom2017uniform,HM1} for the precise definition of the K-stability of the log pairs, which is not necessary in this paper.

\begin{defn}
    Let $C$ be a connected proper nodal curve with an effective $\mathbb{R}$-Cartier divisor $\Delta$.
    If $(C,\Delta)$ is slc and $K_C+\Delta$ is ample, we say that $(C,\Delta)$ is {\it stable}.
    If $-K_C-\Delta$ is ample, $C$ is smooth, and $2\mathrm{mult}_p(\Delta)\le \mathrm{deg}\Delta$ (resp. $2\mathrm{mult}_p(\Delta)< \mathrm{deg}(\Delta)$) for any closed point $p\in\mathbb{P}^1$, then we say that $(C,\Delta)$ is a K-semistable (resp. K-stable) log Fano pair. 
    If $(C,\Delta)$ with $-(K_C+\Delta)$ ample is not K-semistable, we say that $(C,\Delta)$ is a {\it K-unstable} pair.
    If $(C,\Delta)$ is klt and a K-unstable log Fano pair, then it is easy to check that there is a unique closed point $p$ such that $2\mathrm{mult}_p(\Delta)> \mathrm{deg}(\Delta)$ (see \cite{zhuang2021optimal} for more general statement).
    We say that such $p$ is the {\it destabilizing center}.
\end{defn}

We set the moduli stacks as follows.

\begin{defn}
    Set $a\in\mathbb{R}_{>0}$ and $d\in\mathbb{Z}_{>0}$.
    Set $\mathcal{M}^{\mathrm{Kss}}_{d,a}$ as the category fibered in groupoids such that for any scheme $S$, the objects of $\mathcal{M}^{\mathrm{Kss}}_{d,a}(S)$ are flat projective morphisms $\pi\colon\mathcal{C}\to S$ with a relative Cartier divisor $\mathcal{D}$ such that for any $s\in S$, $(\mathcal{C}_s,a\mathcal{D}_s)$ is a K-semistable log Fano pair and $\mathrm{deg}(\mathcal{D}_s)=d$. We will write this as $(\mathcal{C},a\mathcal{D})$.
    Isomorphisms $\phi\colon (\mathcal{C},a\mathcal{D})\to (\mathcal{C}',a\mathcal{D}')$ are defined as isomorphisms $\phi\colon \mathcal{C}\to\mathcal{C}'$ over $S$ such that $\phi^*\mathcal{D}'=\mathcal{D}$.

    Set $\mathcal{M}^{\mathrm{st}}_{d,a}$ as the category fibered in groupoids such that for any scheme $S$, the objects of $\mathcal{M}^{\mathrm{st}}_{d,a}(S)$ are flat projective morphisms $\pi\colon\mathcal{C}\to S$ with a relative Cartier divisor $\mathcal{D}$ such that for any $s\in S$, $(\mathcal{C}_s,a\mathcal{D}_s)$ is a stable pair, $p_a(\mathcal{C}_s)=0$, and $\mathrm{deg}(\mathcal{D}_s)=d$. We will write this as $(\mathcal{C},a\mathcal{D})$.
    Isomorphisms $\phi\colon (\mathcal{C},a\mathcal{D})\to (\mathcal{C}',a\mathcal{D}')$ are defined as isomorphisms $\phi\colon \mathcal{C}\to\mathcal{C}'$ over $S$ such that $\phi^*\mathcal{D}'=\mathcal{D}$.
\end{defn}

\begin{thm}
    $\mathcal{M}^{\mathrm{Kss}}_{d,a}$ is an Artin stack of finite type over $\mathbb{C}$ and admits a projective good moduli space ${M}^{\mathrm{Kps}}_{d,a}$.
    $\mathcal{M}^{\mathrm{st}}_{d,a}$ is a Deligne-Mumford stack of finite type over $\mathbb{C}$ and admits a projective coarse moduli space ${M}^{\mathrm{st}}_{d,a}$.
\end{thm}

For the above theorem, more general statements are shown in \cite{kollar-moduli} for $\mathcal{M}^{\mathrm{st}}_{d,a}$. For $\mathcal{M}^{\mathrm{Kss}}_{d,a}$ where $a\in \mathbb{Q}$, the assertion is shown in  \cite{xu2025k} (see \cite[Theorem 1.3]{liu2022finite}).
For the case $a\not\in  \mathbb{Q}$, we can reduce to the rational coefficient case using the result of \cite{zhou2023wall} (see \cite[Lemma 3.4]{HM1}).
See Theorem \ref{thm:HM1:construction} for the statement generalized to quasimaps later.

On the other hand, we explain the K-stability of quasimaps, which is one of the key tools in this paper.
First of all, we explain the definition of quasimaps to $\mathbb{P}^1$.
\begin{defn}
Let $C$ be a proper connected nodal curve with an effective $\mathbb{R}$-divisor $\Delta$ such that $(C,\Delta)$ is slc.
We say that $q\colon C\to [\mathbb{A}^2/\mathbb{G}_{m}]$ is a {\it quasimap} of degree $d$ if all generic points of irreducible components of $C$ are mapped to $\mathbb{P}^1\subset [\mathbb{A}^2/\mathbb{G}_{m}]$ and $q$ corresponds to a homomorphism $\mathcal{O}_{C}^{\oplus2}\to \mathscr{L}$, where $\mathscr{L}$ is a line bundle of degree $d$. 
We say that $\mathscr{L}$ is the line bundle associated with $q$.
For any general $e\in\mathbb{C}^2$, let $f$ be the corresponding section, which is defined by the homomorphism $\mathcal{O}_{C}^{\oplus2}\to \mathscr{L}$.
Note that such $f$ is a non-zero divisor and we can define $\mathrm{div}(f)$ as a Cartier divisor.
We say that $q$ with weight $w\in\mathbb{R}_{>0}$ is {\it slc} (resp. {\it klt}) if for any closed point $p\in C$,  $(C,\Delta+w\mathrm{div}(f))$ is slc (resp. klt) around $p$ for any general $e\in \mathbb{C}^2$.
\end{defn}

We note that the above quasimaps can be regarded as a linear system $[f_0: f_1]$ defining a rational map $C\dashrightarrow \mathbb{P}^1$, where $f_0,f_1\in H^0(C,\mathscr{L})$ is defined as the images of the canonical free basis via the canonical homomorphism $\mathcal{O}_{C}^{\oplus2}\to \mathscr{L}$.
We say that {\it the quasimap $q$ is defined by $[f_0: f_1]$} in this case.

\begin{defn}[K-stability]
Put $q$ as the above definition.
    If $w=0$, then we regard $q$ as the log pair $(C,\Delta)$.
We say that $q$ with weight $w\in\mathbb{R}_{\ge0}$ is {\it stable} (resp. {\it a K-semistable log Fano quasimap}) if $K_C+\Delta+w\mathscr{L}$ is ample and $q$ is slc (resp. $C\cong\mathbb{P}^1$, $-(K_C+\Delta+w\mathscr{L})$ is ample, and for any $p\in C$, there exists an element $e\in\mathbb{A}^2$ such that $2\mathrm{mult}_p(\Delta+w\mathrm{div}(f))\le \mathrm{deg}(\Delta+w\mathscr{L})$, where $f$ is the corresponding section to $e$).
We say that a K-semistable quasimap $q$ is {\it K-stable} if $2\mathrm{mult}_p(\Delta+w\mathrm{div}(f))<\mathrm{deg}(\Delta+w\mathscr{L})$ and $C$ is smooth.
Furthermore, we say that $q$ is a {\it boundary polarized Calabi-Yau (CY) quasimap} with weight $w$ if $q$ is slc, $K_C+\Delta+w\mathscr{L}\sim_{\mathbb{R}}0$ and $w\mathscr{L}+\Delta$ is ample.
\end{defn}

\begin{rem}
    In \cite{HM1}, we consider boundary polarized CY quasimaps with weight $(w_1,w_2)$.
    In this paper, we only need the case where $w_1=0$ and therefore we only write $w_2$ in the above definition.
\end{rem}

We put the following lemma, which is useful to detect the K-moduli wall crossing later.

\begin{lem}\label{lem:useful-criterion}
Let $q\colon (\mathbb{P}^1,aD)\to [\mathbb{A}^2/\mathbb{G}_{m}]$ be a quasimap that is klt for some $a\in (0,1)$ with weight $0<w<1$.
Let $e_1,e_2\in\mathbb{C}^2$ such that $e_1$ and $e_2$ generate $\mathbb{C}^2$ and $f_1$ and $f_2$ the corresponding section to $e_1$ and $e_2$, respectively.
Suppose that $-(K_{\mathbb{P}^1}+(a+w)D)$ is ample. Then the following hold.
\begin{enumerate}
    \item If $\mathrm{div}(f_1)=D$ and $f_2\not\equiv0$, then $q$ is K-semistable if and only if $2\min\{\mathrm{ord}_p(aD+w\mathrm{div}(f_2)),(a+w)\mathrm{ord}_p(D)\}\le(a+w)\mathrm{deg}(D)$ for any closed point $p\in\mathbb{P}^1$.
    \item If $f_1\equiv0$ and $\mathrm{div}(f_2)\sim_{\mathbb{Q}}D$, then $q$ is K-semistable if and only if $2\mathrm{ord}_p(aD+w\mathrm{div}(f_2))\le(a+w)\mathrm{deg}(D)$ for any closed point $p\in\mathbb{P}^1$.
    \label{eq-eq-2}
\end{enumerate}
\end{lem}

\begin{proof}
 In \eqref{eq-eq-2}, we note that $f_2\not\equiv0$.    
 By definition, the condition that $2\min\{\mathrm{ord}_p(aD+w\mathrm{div}(f_2)),(a+w)\mathrm{ord}_p(D)\}\le(a+w)\mathrm{deg}(D)$ for any closed point $p\in\mathbb{P}^1$ implies the K-semistability of $q$.
 Therefore, it suffices to deal with the converse direction.
 Suppose that $q$ is K-semistable and then for any closed point $p\in \mathbb{P}^1$, there exists a section $f_3$ that corresponds to general $e_3\in \mathbb{C}^2$ and satisfies that $2\mathrm{ord}_p(aD+w\mathrm{div}(f_3))\le(a+w)\mathrm{deg}(D)$.
 In the case \eqref{eq-eq-2}, $\mathrm{div}(f_3)=\mathrm{div}(f_2)$ and we obtain the assertion.
 In the other case, $e_3=b_1e_1+b_2e_2$ for some $b_1,b_2\in\mathbb{C}$. 
 Therefore, $f_3=b_1f_1+b_2f_2$.
Since $2\mathrm{ord}_p(aD+w\mathrm{div}(f_3))\le(a+w)\mathrm{deg}(D)$, $\mathrm{ord}_p(f_3)=\mathrm{ord}_p(f_1)$ or $\mathrm{ord}_p(f_3)=\mathrm{ord}_p(f_2)$
implies that $$2\min\{\mathrm{ord}_p(aD+w\mathrm{div}(f_2)),(a+w)\mathrm{ord}_p(D)\}\le(a+w)\mathrm{deg}(D).$$ 
We complete the proof.
\end{proof}

\begin{defn}
% \Yota{This is a general remark: the definition environment seems to be a bit long. Could we remove the relevant properties or remarks outside?}
Let $\pi\colon\mathcal{C}\to S$ be a flat projective morphism with each geometric fiber a connected nodal curve.    
Let $\Delta$ be an effective relative Mumford $\mathbb{R}$-divisor $\Delta$ (i.e. an $\R_{\ge0}$-combination of relative Cartier divisors) such that for any $s\in S$, $(\mathcal{C}_s,\Delta_s)$ is slc.
Let $w\in\mathbb{R}_{>0}$.
We say that $q\colon (\mathcal{C},\Delta)\to [\mathbb{A}^2_S/\mathbb{G}_{m,S}]$ is 
\begin{enumerate}
    \item a {\it locally stable family of log quasimaps} of degree $d$ if $q_s\colon \mathcal{C}_s\to[\mathbb{A}^2_{\kappa(s)}/\mathbb{G}_{m,\kappa(s)}]$ is a quasimap of degree $d$ for any $s\in S$;
    \item a {\it stable family of quasimaps} (resp. {\it family of K-semistable log Fano quasimaps}) if $q_s$ is a stable quasimap (resp. a K-semistable log Fano quasimap);
    \item a {\it family of boundary polarized CY quasimaps with weight $w$} if $q_s$ is a boundary polarized CY quasimap for any $s\in S$ such that $w\mathscr{L}+\Delta$ is relatively ample.
    Here, we set $\mathscr{L}$ as the associated line bundle such that $\mathcal{O}_{\mathcal{C}}^{\oplus 2}\to \mathscr{L}$ corresponds to the morphism $q$.
\end{enumerate}

%We note that $q$ is a family of boundary polarized CY quasimaps with weight $w_1,w_2$ if and only if $q$ is a family of boundary polarized CY quasimaps with weight $w_1',w_2'>0$ such that $w'_1+w'_2=w$.
%Therefore, we simply call $q$ a family of boundary polarized CY quasimaps with weight $w$.

In the above cases, if $w>1$, then we see that $q$ always factors through $\mathbb{P}^1_S\subset [\mathbb{A}^2_S/\mathbb{G}_{m,S}]$.
\end{defn}

Next, we will deal with the following moduli stacks.

\begin{defn}
Set $d_1\in\mathbb{Z}_{>0},d_2\in\mathbb{Z}_{>0}$ and $a,w\in\mathbb{R}_{>0}$.
Set $\mathcal{M}^{\mathrm{Kssqmaps}}_{d_1,a,d_2,w}$ as the category fibered in groupoids such that for any scheme $S$, the objects of $\mathcal{M}^{\mathrm{Kssqmaps}}_{d_1,a,d_2,w}(S)$ are locally stable families of log quasimaps
$q\colon(\mathcal{C},a\mathcal{D})\to [\mathbb{A}^2_S/\mathbb{G}_{m,S}]$ with a relative Cartier divisor $\mathcal{D}$ and weight $w$ and of degree $d_1$ such that for any $s\in S$, $q_s$ is a K-semistable log Fano quasimap and $\mathrm{deg}(\mathcal{D}_s)=d_2$. 
    Isomorphisms $\phi\colon q\to (q'\colon(\mathcal{C}',a\mathcal{D}')\to [\mathbb{A}^2_S/\mathbb{G}_{m,S}])$ are defined as isomorphisms $\phi\colon \mathcal{C}\to\mathcal{C}'$ over $S$ such that $q'\circ\phi=q$ and $\phi^*\mathcal{D}'=\mathcal{D}$.

Set $\mathcal{M}^{\mathrm{stqmaps}}_{d_1,a,d_2,w}$ as the category fibered in groupoids such that for any scheme $S$, the objects of $\mathcal{M}^{\mathrm{stqmaps}}_{d_1,a,d_2,w}(S)$ are locally stable families of log quasimaps $q\colon(\mathcal{C},a\mathcal{D})\to [\mathbb{A}^2_S/\mathbb{G}_{m,S}]$ with a relative Cartier divisor $\mathcal{D}$ and weight $w$ and of degree $d_1$ such that for any $s\in S$, $p_a(\mathcal{C}_s)=0$, $q_s$ is a stable quasimap and $\mathrm{deg}(\mathcal{D}_s)=d_2$. 
    Isomorphisms $\phi\colon q\to (q'\colon(\mathcal{C}',a\mathcal{D}')\to [\mathbb{A}^2_S/\mathbb{G}_{m,S}])$ are defined as isomorphisms $\phi\colon \mathcal{C}\to\mathcal{C}'$ over $S$ such that $q'\circ\phi=q$ and $\phi^*\mathcal{D}'=\mathcal{D}$.

    Set $\mathcal{M}^{\mathrm{bpcyqmaps}}_{d_1,a,d_2,w}$ as the category fibered in groupoids such that for any scheme $S$, the objects of $\mathcal{M}^{\mathrm{bpcyqmaps}}_{d_1,a,d_2,w}(S)$ are locally stable families of log quasimaps $q\colon(\mathcal{C},a\mathcal{D})\to [\mathbb{A}^2_S/\mathbb{G}_{m,S}]$ with a relative Cartier divisor $\mathcal{D}$ and weight $w$ and of degree $d_1$ such that for any $s\in S$, $q_s$ is a boundary polarized CY quasimap and $\mathrm{deg}(\mathcal{D}_s)=d_2$. 
    Isomorphisms $\phi\colon q\to (q'\colon(\mathcal{C}',a\mathcal{D}')\to [\mathbb{A}^2_S/\mathbb{G}_{m,S}])$ are defined as isomorphisms $\phi\colon \mathcal{C}\to\mathcal{C}'$ over $S$ such that $q'\circ\phi=q$ and $\phi^*\mathcal{D}'=\mathcal{D}$.
\end{defn}

\begin{thm}[\cite{HM1}]\label{thm:HM1:construction}
    Set $d_1\in\mathbb{Z}_{>0},d_2\in\mathbb{Z}_{>0}$ and $a,w\in\mathbb{R}_{>0}$.
    Then $\mathcal{M}^{\mathrm{Kssqmaps}}_{d_1,a,d_2,w}$ and $\mathcal{M}^{\mathrm{bpcyqmaps}}_{d_1,a,d_2,w}$ are Artin stacks of finite type over $\mathbb{C}$ with normal projective good moduli spaces $M^{\mathrm{Kpsqmaps}}_{d_1,a,d_2,w}$ and $M^{\mathrm{bpcyqmaps}}_{d_1,a,d_2,w}$, respectively.
    Furthermore, $\mathcal{M}^{\mathrm{stqmaps}}_{d_1,a,d_2,w}$ is a proper Deligne-Mumford stack of finite type over $\mathbb{C}$ with a normal projective coarse moduli space $M^{\mathrm{stqmaps}}_{d_1,a,d_2,w}$.
\end{thm}

\begin{proof}
    This follows from \cite[Theorems 3.6, 5.6, 5.9, 4.6, and Corollary 4.17]{HM1}.
\end{proof}

In \cite{HM1}, the authors studied the wall-crossing phenomena for these moduli stacks and moduli spaces.
We list the following theorems.

\begin{thm}\label{thm:wall-crossing-recall}
Set $d_1\in\mathbb{Z}_{>0},d_2\in\mathbb{Z}_{>0}$ and $a,w\in\mathbb{Q}_{>0}$.
Suppose that $\frac{2}{ad_2+wd_1}<\frac{1}{w}< 1$.
There exist finitely many rational numbers $\frac{2}{ad_2+wd_1}=\alpha_1<\alpha_2<\ldots<\alpha_l=\frac{1}{w}$ that satisfy the following.
\begin{enumerate}
    \item For any $\alpha,\alpha'\in (0,\alpha_1)$, we see that $\mathcal{M}^{\mathrm{Kssqmaps}}_{d_1,\alpha a,d_2,\alpha w}\cong \mathcal{M}^{\mathrm{Kssqmaps}}_{d_1,\alpha' a,d_2,\alpha' w}$.
    Furthermore, there exists an open immersion $\mathcal{M}^{\mathrm{Kssqmaps}}_{d_1,\alpha a,d_2,\alpha w}\hookrightarrow\mathcal{M}^{\mathrm{bpcyqmaps}}_{d_1,\alpha_1 a,d_2,\alpha_1 w}$, which induces an isomorphism $M^{\mathrm{Kpsqmaps}}_{d_1,\alpha a,d_2,\alpha w}\cong M^{\mathrm{bpcyqmaps}}_{d_1,\alpha_1 a,d_2,\alpha_1 w}$,
    \item For any $\alpha,\alpha'\in (\alpha_i,\alpha_{i+1}]$ with $1\le i\le l$, we see that $\mathcal{M}^{\mathrm{stqmaps}}_{d_1,\alpha a,d_2,\alpha w}\cong \mathcal{M}^{\mathrm{stqmaps}}_{d_1,\alpha' a,d_2,\alpha' w}$.
    Furthermore, there exist the canonical morphisms $\mathcal{M}^{\mathrm{stqmaps}}_{d_1,\alpha a,d_2,\alpha w}\to\mathcal{M}^{\mathrm{bpcyqmaps}}_{d_1,\alpha_1 a,d_2,\alpha_1 w}$ for any $\alpha\in (\alpha_1,\alpha_2]$ and $\mathcal{M}^{\mathrm{stqmaps}}_{d_1,\alpha a,d_2,\alpha w}\to\mathcal{M}^{\mathrm{stqmaps}}_{d_1,\alpha_i a,d_2,\alpha_i w}$ for any $\alpha\in (\alpha_i,\alpha_{i+1}]$ with $1<i\le l$.
Here, we set $\alpha_{l+1}:=1$.
\end{enumerate}
\end{thm}

\begin{proof}
 (1) follows from \cite[Remark 3.5 and Theorem 4.20]{HM1}.
 (2) follows from \cite[Theorems 5.10 and 5.11]{HM1}.
\end{proof}

\begin{thm}\label{thm:hm1-fromstable}
 Set $a,w,a',w'\in\mathbb{Q}_{\ge0}$ and $d_1,d_2\in\mathbb{Z}_{>0}$.
 Suppose further that $a\ne0$ or $w\ne0$.
 We set ${M}_1$ as
 \begin{enumerate}
 \item ${M}^{\mathrm{Kps}}_{d_2,a}$ if $d_2a<2$ and $w=0$,
     \item ${M}^{\mathrm{bpcyqmaps}}_{d_1,a,d_2,w}$ if $d_2a+d_1w=2$, or 
     \item ${M}^{\mathrm{Kpsqmaps}}_{d_1,a,d_2,w}$ if $d_2a+d_1w<2$ and $w>0$.
 \end{enumerate}
 Suppose further that $d_2a'+d_1w'>2$, $d_2a+d_1w\le2$, and there exists a positive constant $\epsilon>0$ such that $\frac{2(1+\epsilon)}{d_2a+d_1w}a\le a'$ and $\frac{2(1+\epsilon)}{d_2a+d_1w}w\le w'$.
 Then there exists a morphism ${M}^{\mathrm{stqmaps}}_{d_1,a',d_2,w'}\to M_1$.
\end{thm}

\begin{proof}
This follows from Theorem \ref{thm:wall-crossing-recall} and \cite[Theorem 5.11 and Remark 5.12]{HM1}.    
\end{proof}

\begin{defn}
Let $q$ be an object of $\mathcal{M}^{\mathrm{bpcyqmaps}}_{d_1,a,d_2,w}$ that defines a closed $\mathbb{C}$-valued point.
Then we say that $q$ is {\it bpcy polystable}.
$q$ is bpcy polystable if and only if $q$ is a boundary polarized CY quasimap and for any $f\colon\Theta\to \mathcal{M}^{\mathrm{bpcyqmaps}}_{d_1,a,d_2,w}$
such that $f(1)$ is isomorphic to $q$, then $f(0)$ is also isomorphic to $q$ by \cite[Lemma 3.24]{AHLH}.

Let $q$ be an object of $\mathcal{M}^{\mathrm{Kssqmaps}}_{d_1,a,d_2,w}$ that defines a closed $\mathbb{C}$-valued point.
Then we say that $q$ is K-{\it polystable}.
As above, $q$ is K-polystable if and only if $q$ is a K-semistable log Fano quasimap and for any $f\colon\Theta\to \mathcal{M}^{\mathrm{Kssqmaps}}_{d_1,a,d_2,w}$ such that $f(1)$ is isomorphic to $q$, then $f(0)$ is also isomorphic to $q$.
\end{defn}

We put the following lemma, which is useful to track K-polystability of weighted quasimaps.

\begin{lem}\label{lem:K-ps-remaining}
Let $q\colon(\mathbb{P}^1,aD)\to [\mathbb{A}^2/\mathbb{G}_m]$ with weight $w(1-a)$ be a K-polystable log Fano quasimap for some $0<a<1$ and $w>0$.
Suppose that there exist $0<c_1<a$ and $a<c_2<1$ such that for any $c\in [c_1,c_2]$, $q\colon(\mathbb{P}^1,cD)\to [\mathbb{A}^2/\mathbb{G}_m]$ with weight $(1-c)w$ is a K-semistable log Fano quasimap.
Then, for any $c\in(c_1,c_2)$, $q\colon(\mathbb{P}^1,cD)\to [\mathbb{A}^2/\mathbb{G}_m]$ with weight $(1-c)w$ is K-polystable.
If $q\colon(\mathbb{P}^1,aD)\to [\mathbb{A}^2/\mathbb{G}_m]$ with weight $w(1-a)$ is further K-stable, then for any $c\in(c_1,c_2)$, $q\colon(\mathbb{P}^1,cD)\to [\mathbb{A}^2/\mathbb{G}_m]$ with weight $(1-c)w$ is K-stable.
\end{lem}

\begin{proof}
First, suppose that $q\colon(\mathbb{P}^1,aD)\to [\mathbb{A}^2/\mathbb{G}_m]$ with weight $w(1-a)$ is K-stable.    
By \cite[Proposition 2.40 (3)]{HM1}, for any closed point $p\in \mathbb{P}^1$, $2(1-\mathrm{ord}_p(aD+w(1-a)\mathrm{Bs}(q)))>2-\mathrm{deg}(w(1-a)\mathscr{L}+aD)$, where $\mathscr{L}$ is the associated line bundle to $q$.
By assumption, we obtain that for any closed point $p\in \mathbb{P}^1$ and $c\in (c_1,c_2)$, $2(1-\mathrm{ord}_p(cD+w(1-c)\mathrm{Bs}(q)))>2-\mathrm{deg}(w(1-c)\mathscr{L}+cD)$.

Next, suppose that $q\colon(\mathbb{P}^1,aD)\to [\mathbb{A}^2/\mathbb{G}_m]$ with weight $w(1-a)$ is not K-stable.
Then, $q$ is constant by \cite{HM1}.
In this case, it is not hard to check that the K-polystability of $q\colon(\mathbb{P}^1,cD)\to [\mathbb{A}^2/\mathbb{G}_m]$ with weight $(1-c)w$ coincides with the K-polystability of the log pair $(\mathbb{P}^1,cD+w(1-c)\mathrm{Bs}(q))$. 
Assume that $(\mathbb{P}^1,cD+w(1-c)\mathrm{Bs}(q))$ is not K-polystable for some $c\in (c_1,c_2)$.
Then there exists a $\mathbb{G}_m$-equivariant flat projective morphism $\pi\colon (\mathcal{C},\Delta)\to \mathbb{A}^1$ such that $(\mathcal{C}_t,\Delta_t)\cong (\mathbb{P}^1,cD+w(1-c)\mathrm{Bs}(q))$ for any $t\in \mathbb{A}^1\setminus\{0\}$, $(\mathcal{C}_0,\Delta_0)$ is also a K-semistable log Fano pair but not isomorphic to $(\mathbb{P}^1,cD+w(1-c)\mathrm{Bs}(q))$, and $-K_{\mathcal{C}}-\Delta$ is $\pi$-ample.
This is a test configuration, which we defined in \cite[Definition 2.14]{HM1}.
Then, there exists a point $p$ such that $\mathcal{C}=\mathrm{Proj}_{\mathbb{A}^1}(\sum_{\lambda\in\mathbb{Z},m\in\mathbb{Z}_{\ge0}} t^{-\lambda}H^0(\mathbb{P}^1,\mathcal{O}_{\mathbb{P}^1}(m)(-\lambda\cdot p)))$ and $2(1-\mathrm{ord}_p(cD+w(1-c)\mathrm{Bs}(q)))=2-\mathrm{deg}(w(1-c)\mathscr{L}+cD)$
by \cite[Lemma 4.19]{xu2025k} and the same argument as the proof of \cite[Lemma 3.4]{zhou2023wall}, where $t$ is the canonical coordinate of $\mathbb{A}^1$.
By assumption that $(\mathbb{P}^1,c'D+w(1-c')\mathrm{Bs}(q))$ is K-semistable for any $c'\in [c_1,c_2]$, we see that $2(1-\mathrm{ord}_p(c'D+w(1-c')\mathrm{Bs}(q)))=2-\mathrm{deg}(w(1-c')\mathscr{L}+c'D)$.
Let $\Delta_{c'}$ be the Zariski closure of $\mathbb{G}_m\times (c'D+w(1-c')\mathrm{Bs}(q))$ in $\mathcal{C}$.
Then $(\mathcal{C}_0,\Delta_{c',0})$ is K-semistable by \cite[Lemma 4.19 and Proposition 5.37]{xu2025k}.
This contradicts the assumption that $q\colon(\mathbb{P}^1,aD)\to [\mathbb{A}^2/\mathbb{G}_m]$ with weight $w(1-a)$ is K-polystable.
We complete the proof.
\end{proof}

 Let $q\colon (\mathcal{C},\Delta)\to [\mathbb{A}^2_S/\mathbb{G}_{m,S}]$ be a family of boundary polarized CY quasimaps with weight $w\in\mathbb{Q}_{\ge0}$, $\Delta$ be a relative Mumford $\mathbb{Q}$-divisor and the canonical morphism $\pi\colon \mathcal{C}\to S$.
\begin{defn}\label{defn:Hogde-HM2}
    Let $\mathscr{L}$ be the line bundle associated with $q$.
    For $w>0$, as \cite[Definition 4.14]{HM1}, for any sufficiently divisible positive integer $N$, there exists a line bundle $\lambda_{\mathrm{Hodge},q,N}$ such that $\pi^*\lambda_{\mathrm{Hodge},q,N}\sim \omega^{\otimes N}_{\mathcal{C}/S}(N\Delta)\otimes\mathscr{L}^{\otimes wN}$.
    Let $\lambda_{\mathrm{Hodge},q}:=\frac{1}{N}\lambda_{\mathrm{Hodge},q,N}$ be the Hodge line bundle.
    When $w=0$, we also set $\lambda_{\mathrm{Hodge},q,N}$ as $\pi^*\lambda_{\mathrm{Hodge},q,N}\sim \omega^{\otimes N}_{\mathcal{C}/S}(N\Delta)$ and $\lambda_{\mathrm{Hodge},q}:=\frac{1}{N}\lambda_{\mathrm{Hodge},q,N}$.
\end{defn}
    We note that $\lambda_{\mathrm{Hodge},q}$ is a functorial $\mathbb{Q}$-line bundle, which is determined uniquely up to $\mathbb{Q}$-linear equivalence.
    Indeed, for any morphism $g\colon T\to S$, $\lambda_{\mathrm{Hodge},q_T}=g^*\lambda_{\mathrm{Hodge},q}$.

       As \cite[Definition 4.14 and Lemma 4.15]{HM1}, by the functoriality, we can define the Hodge $\mathbb{Q}$-line bundle $\lambda_{\mathrm{Hodge}}$ on $\mathcal{M}^{\mathrm{bpcyqmaps}}_{d_1,a,d_2,w}$ in a canonical way that for any family of boundary polarized CY quasimaps $q$ over $S$ that corresponds to a morphism $f\colon S\to \mathcal{M}^{\mathrm{bpcyqmaps}}_{d_1,a,d_2,w}$, $\lambda_{\mathrm{Hodge},q}=f^*\lambda_{\mathrm{Hodge}}$.
    Furthermore, $\lambda_{\mathrm{Hodge}}$ is descended to a $\mathbb{Q}$-line bundle $\Lambda_{\mathrm{Hodge}}$ on ${M}^{\mathrm{bpcyqmaps}}_{d_1,a,d_2,w}$ uniquely up to $\mathbb{Q}$-linear equivalence.
    As shown in \cite[Theorem 4.16]{HM1}, we see that $\Lambda_{\mathrm{Hodge}}$ is ample.
    Note that $\Lambda_{\mathrm{Hodge}}$ is affine in ${a}$ and $w$ and we will write it as $\Lambda_{\mathrm{Hodge},{a},w}$.

We note that via ${M}^{\mathrm{bpcyqmaps}}_{d_1,a,d_2,w}\cong {M}^{\mathrm{Kpsqmaps}}_{d_1,ta,d_2,tw}$ for any $0<t<1$ (cf. Theorem \ref{thm:wall-crossing-recall}), 
there exists another ample $\Q$-line bundle, which is known as the CM line bundle.
    By \cite[Proposition 4.23]{HM1}, the Hodge line bundle and the CM line bundle define the same polarization on the same good moduli space.
    Therefore, we will only consider the Hodge line bundle for the K-moduli of weighted quasimaps in this paper via the above isomorphism.

\section{Moduli of rational elliptic surfaces}
\label{sec:moduli-res}

In this section we recall the ball quotient description of the moduli space of rational elliptic surfaces and fix the notation used below.  The main references are Heckman-Looijenga's study of the moduli space of rational elliptic surfaces \cite{HeckmanLooijenga2002} and Miranda's GIT compactification of Weierstrass fibrations \cite{Miranda1981}.  We follow the notation of \cite{HeckmanLooijenga2002} as much as possible; when useful, we also indicate Looijenga's notation from \cite{Looijenga2003Ball}.

\subsection{Miranda's compactification of rational elliptic surfaces}\label{subsec-weierstraa-explanation}

We recall Miranda's compactification of the moduli space of rational elliptic surfaces.

In this paper, we mainly deal with rational elliptic surfaces with no multiple fiber.
They can be realized as blow-ups of the projective plane at special nine points, and it is well known that they always have a section over $\mathbb{P}^1$.

Consider the $\mathbb{P}^2$-bundle $$\pi\colon \mathbb{P}:=\mathbb{P}_{\mathbb{P}^1}(\mathcal{O}_{\mathbb{P}^1}(-2)\oplus\mathcal{O}_{\mathbb{P}^1}(-3)\oplus\mathcal{O}_{\mathbb{P}^1})\to \mathbb{P}^1.$$
Let $f_0\in H^0(\mathbb{P}^1,\mathcal{O}_{\mathbb{P}^1}(4))$ and $f_1\in H^0(\mathbb{P}^1,\mathcal{O}_{\mathbb{P}^1}(6))$.
Consider the following section $-y^2z+x^3+3f_0xz^2+2f_1z^3\in H^0(\mathbb{P},\mathcal{O}_{\mathbb{P}}(3)\otimes\pi^*\mathcal{O}_{\mathbb{P}^1}(6))$ and this section defines a Cartier divisor $W\subset \mathbb P$, where $x\in H^0(\mathcal{O}_{\mathbb P}(1)\otimes\pi^*\mathcal{O}_{\mathbb P^1}(2))$, $y\in H^0(\mathcal{O}_{\mathbb P}(1)\otimes\pi^*\mathcal{O}_{\mathbb P^1}(3))$ and $z\in H^0(\mathcal{O}_{\mathbb P}(1))$ are the canonical coordinates. 
Set the discriminant $D:=f^3_0+f_1^2\in H^0(\mathbb P^1,\mathcal{O}_{\mathbb P^1}(12))$.
For any $p\in \mathbb{P}^1$, if $D(p)\ne 0$, then $\pi^{-1}(p)\cap W$ is a smooth elliptic curve. By \cite{kas1977weierstrass}, we see that if $D\not\equiv0$, then $\mathrm{ord}_p(f_0)\le 3$ or $\mathrm{ord}_p(f_1)\le 5$ if and only if $W$ has only canonical singularities.
Let $S$ be the section of $W\to\mathbb P^1$ defined by $x=z=0$.
If $W$ has only canonical singularities and $\nu\colon \tilde{W}\to W$ is the minimal resolution, then $\tilde{W}$ is rational and $W$ is the relative log canonical model of $(\tilde{W},\nu^{-1}_*S)$ over $\mathbb P^1$.
See \cite{Miranda1981} for more details.
We say that $\pi|_W\colon (W,S)\to\mathbb P^1$ is a {\it rational Weierstrass fibration}.
We note that $\tilde{W}$ is expressed as a blow up at the base points of a general cubic pencil. See \cite{cossec2025enriques} for more details.

Let $Y_{\mathrm{Mir}}\defeq \mathbb P(2^5,3^7)$ be the weighted projective space parameterizing nonzero pairs $(f_0,f_1)$ up to the scaling
\[
(f_0,f_1)\sim(\lambda^2 f_0,\lambda^3 f_1).
\]
We call such elements $(f_0,f_1)$ {\it Weierstrass pairs}.
Miranda \cite{Miranda1981} considered the GIT quotient of $Y_{\mathrm{Mir}}$ by $\mathrm{SL}(2;\mathbb{C})$ as follows. 
The GIT semistable (resp. stable) locus $Y_{\mathrm{Mir}}^{\mathrm{ss}}$ (resp. $Y_{\mathrm{Mir}}^{\mathrm{s}}$) consists of all pairs $(f_0,f_1)$ such that, for every $p\in \mathbb P^1$, either $\mathrm{ord}_p(f_0)\le 2$ (resp. $<2$) or $\mathrm{ord}_p(f_1)\le 3$ (resp. $<3$) by \cite[Proposition 5.1]{Miranda1981}.
We call this GIT (semi)stability {\it Miranda-(semi)stability} in this paper.
We also note that the Miranda-(semi)stability of $(f_0,f_1)$ coincides with the K-(semi)stability of the log Fano quasimap $q\colon\mathbb{P}^1\to[\mathbb{A}^2/\mathbb{G}_m]$ with weight $\frac{1}{12}$ defined by $[f_0^3:f_0^3+f_1^2]$ (cf. \cite[Proposition 2.40]{HM1}). 
Then we can construct Miranda's moduli space $\M^M:=\SL(2)\backslash\!\!\backslash Y_{\mathrm{Mir}}^{\mathrm{ss}}=\PGL(2)\backslash\!\!\backslash Y_{\mathrm{Mir}}^{\mathrm{ss}}$.
Here, we set $\mathscr{M}^M:=[\PGL(2)\backslash Y_{\mathrm{Mir}}^{\mathrm{ss}}]$ as the quotient stack and note that $\mathscr{M}^M\to \mathcal{M}^M$ is the morphism of the good moduli space.
By the above, $[f_0^3:f_0^3+f_1^2]$ on $Y_{\mathrm{Mir}}^{\mathrm{ss}}$ defines a K-semistable quasimap $\mathbb{P}^1\to[\mathbb{A}^2/\mathbb{G}_m]$ of degree twelve with weight $\frac{1}{12}$.
\begin{prop}\label{prop:miranda-k-moduli}
    There exists a canonical morphism $\mathscr{M}^M:=[\PGL(2)\backslash Y_{\mathrm{Mir}}^{\mathrm{ss}}]\to \mathcal{M}^{\mathrm{Kssqmaps}}_{12,0,0,\frac{1}{12}}$ of Artin stacks that induces a homeomorphism of the good moduli space to the image, which is the normalization.
\end{prop}

\begin{proof}
    As the above argument, it is easy to construct a morphism $F\colon \mathscr{M}^M:=[\PGL(2)\backslash Y_{\mathrm{Mir}}^{\mathrm{ss}}]\to\mathcal{M}^{\mathrm{Kssqmaps}}_{12,0,0,\frac{1}{12}}$ by the description of $Y_{\mathrm{Mir}}^{\mathrm{ss}}$.
The morphism maps the point defined by $(f_0,f_1)$ to the quasimap defined by $[f_0^3:f_0^3+f_1^2]$. 
    Let $\overline{F}\colon \mathcal{M}^M\to{M}^{\mathrm{Kpsqmaps}}_{12,0,0,\frac{1}{12}}$ be the morphism induced by the good moduli spaces (cf. \cite[Theorem 6.6]{Alper}).
    By the linear transformation on $\mathbb P^1$ such that $[u:v]\mapsto[u:u-v]$, we can regard $F$ as the morphism, which maps $(f_0,f_1)$ to the quasimap defined by $[f_0^3:f_1^2]$. 
    From now on, we consider $F$ to map $(f_0,f_1)$ to the quasimap defined by $[f_0^3:f_1^2]$ in this proof.
    To show that this morphism is a closed immersion in the open locus $U\subset S$ such that $(F_0)_s\ne0$ and $(F_1)_s\ne0$ for any $s\in U$, it suffices to show that for any projective flat morphism $f\colon \mathcal{P}\to S$ such that each geometric fiber is isomorphic to $\mathbb P^1$ with sections $F_0,F_1\in f_*\omega_{\mathcal{P}/S}^{\otimes -6}$, there exists a closed subscheme $S'\subset U$ such that for any morphism $g\colon T\to U$, $g$ factors through $S'$ if and only if there exist local sections $f'_0\in f_*\omega_{\mathcal{P}/S}^{\otimes -2}$ and $f'_1\in f_*\omega_{\mathcal{P}/S}^{\otimes -3}$ such that $f'^3_0=F_0$ and $f'^2_1=F_1$ in the locus.
    To show this, we may replace $U$ with its affine \'etale cover and assume that $\mathcal{P}\cong \mathbb P^1_U$. 
    Note that $(F_0)_u\ne0$ and $(F_1)_u\ne0$ for any $u\in U$.
Note also the fundamental fact that the morphisms induced by cubic power $\mathbb{P}(H^0(\mathbb{P}^1,\mathcal{O}_{\mathbb{P}^1}(4)))\to \mathbb{P}(H^0(\mathbb{P}^1,\mathcal{O}_{\mathbb{P}^1}(12)))$ and induced by square power $\mathbb{P}(H^0(\mathbb{P}^1,\mathcal{O}_{\mathbb{P}^1}(6)))\to \mathbb{P}(H^0(\mathbb{P}^1,\mathcal{O}_{\mathbb{P}^1}(12)))$ are closed immersions because we are working over $\C$.
Therefore, using the closed immersions $\mathbb{P}_U(H^0(\mathbb{P}_U^1,\mathcal{O}_{\mathbb{P}_U^1}(4)))\to \mathbb{P}_U(H^0(\mathbb{P}_U^1,\mathcal{O}_{\mathbb{P}_U^1}(12)))$ and $\mathbb{P}_U(H^0(\mathbb{P}^1,\mathcal{O}_{\mathbb{P}_U^1}(6)))\to \mathbb{P}_U(H^0(\mathbb{P}_U^1,\mathcal{O}_{\mathbb{P}_U^1}(12)))$, we see the above claim.
To show the assertion, we have to check the points $s\in \mathscr{M}^M$ such that $(F_0)_s=0$ or $(F_1)_s=0$.
Let $U^M\subset {M}^{\mathrm{Kpsqmaps}}_{12,0,0,\frac{1}{12}}$ be the open subset that corresponds to quasimaps defined by $F_0,F_1\in H^0(\mathbb{P}^1,\mathcal{O}_{\mathbb{P}^1}(12))$ such that $F_0\ne0$ and $F_1\ne0$. 
By what we have shown, we see that $\overline{F}^{-1}(U^M)\hookrightarrow U^M$ is a closed immersion.
We claim that $\overline{F}$ is entirely injective.
To show this, it suffices to check that $\overline{F}^{-1}(s)$ consists of at most one point for any closed point $s\in {M}^{\mathrm{Kpsqmaps}}_{12,0,0,\frac{1}{12}}$ corresponding to a quasimap defined by the linear system $[g_0^3:g_1^2]$, where $g_0\in H^0(\mathbb{P}^1,\mathcal{O}(4))$ and $g_1\in H^0(\mathbb{P}^1,\mathcal{O}(6))$ such that $g_0\equiv0$ or $g_1\equiv0$.
Suppose $g_0\equiv0$ but $g_1\not\equiv0$.
By the description of $\mathcal{M}^M$ in \cite{Miranda1981}, $\overline{F}^{-1}(s)$ defined by the coefficients of $g_0$ and $g_1$ up to the action of $\PGL(2)$ and constant multiple.
In this case, it is not hard to check that $\overline{F}^{-1}(s)$ consists of at most one point.
On the other hand, we can see that $\overline{F}^{-1}(s)$ consists of at most one point when $g_1\equiv0$ but $g_0\not\equiv0$.
Because $\mathcal{M}^M$ is proper and $\overline{F}$ is injective, we see that this is a homeomorphism.
Since $Y^{\mathrm{ss}}_{\mathrm{Mir}}$ is normal, it follows that $\mathcal{M}^M$ is also normal.
We complete the proof.
\end{proof}
From now on, for a Weierstrass pair $(f_0,f_1)$, we always associate a quasimap defined by $[f_0^3:f_0^3+f_1^2]$.
Note that the symbol $\frac{f_0^3}{f_0^3+f_1^2}$ is different from the classical $j$-invariant by the multiple of $1728$. 
By Proposition \ref{prop:miranda-k-moduli}, we see that $\M^M$ is the normalization of a closed subscheme of $M^{\mathrm{Kpsqmaps}}_{12,0,0,\frac{1}{12}}$.

\begin{rem}\label{rem:OO}
  We can also consider the Weierstrass models defined by $-y^2z+x^3+3g_0xz^2+2g_1z^3\in H^0(\mathbb{P}',\mathcal{O}_{\mathbb{P}'}(3)\otimes\pi^*\mathcal{O}_{\mathbb{P}^1}(12))$, where $\mathbb{P}':=\mathbb{P}_{\mathbb{P}^1}(\mathcal{O}_{\mathbb{P}^1}(-4)\oplus\mathcal{O}_{\mathbb{P}^1}(-6)\oplus\mathcal{O}_{\mathbb{P}^1})$, $g_0\in H^0(\mathbb{P}^1,\mathcal{O}_{\mathbb{P}^1}(8))$, and $g_1\in H^0(\mathbb{P}^1,\mathcal{O}_{\mathbb{P}^1}(12))$.
  Miranda \cite{Miranda1981} also considered the GIT stability of the pairs $(g_0,g_1)\in \mathbb{P}(2^9,3^{13})$ and Odaka--Oshima \cite{OO21} showed that the corresponding
GIT moduli space $M^{\rm GIT}$ is isomorphic to the
Baily--Borel compactification of the moduli space of
elliptic K3 surfaces with sections.
  
  As in Proposition \ref{prop:miranda-k-moduli}, using Theorem \ref{thm:wall-crossing-recall}, one can show that there exists a canonical morphism
  \[
  {M}^{\mathrm{GIT}}\to M^{\mathrm{Kpsqmaps}}_{24,0,0,\frac{a}{12}}\cong M^{\mathrm{bpcyqmaps}}_{24,0,0,\frac{1}{12}} 
  \]
  for any $a\in(0,1)$ that induces an isomorphism of $({M}^{\mathrm{GIT}})^{\mathrm{sn}}$ to the seminormalization of its image.
\end{rem}

\subsection{The GIT moduli space of twelve points of $\mathbb{P}^1$}

Note that $\mathrm{SL}(2)$ naturally acts on $\mathbb{P}^{12}=\mathbb{P}(H^0(\mathbb{P}^1,\mathcal{O}_{\mathbb{P}^1}(12)))$ and we can consider the GIT quotient $\mathrm{SL}(2)$
\[
\mathcal{D}^*:=\mathrm{SL}(2)\backslash\!\backslash (\mathbb{P}^{12})^{\mathrm{ss}}.
\]
Let $\mathscr{D}^*:=[\mathrm{PGL}(2)\backslash (\mathbb{P}^{12})^{\mathrm{ss}}]$ be an Artin stack with the good moduli space $\mathcal{D}^*$.
Note that $\mathscr{D}^*\cong\mathcal{M}^{\mathrm{Kss}}_{12,\frac{\epsilon}{12}}\cong \mathcal{M}^{\mathrm{Kss}}_{12,\frac{1}{12}}$ for any sufficiently small $\epsilon>0$ by \cite[Theorem 3.2 and Remark 3.5]{HM1}.
% \Yota{These objects are out of scope of Definition 3.7. Perhaps the indices are converse.} \Masa{I did it.}
Let $\mathscr{M}\subset\mathscr{M}^M$ be an open substack that parameterizes all quasimaps defined by Weierstrass pairs $(f_0,f_1)$ such that $(f_0^3+f_1^2=0)$ is a reduced divisor.
As shown in \cite[Section 2.3]{HeckmanLooijenga2002}, we see that there exists a natural morphism $\mathcal{M}\to \mathcal{D}^*$ defined by $(f_0,f_1)\mapsto (f_0^3+f_1^2=0)$, where $\mathcal{M}$ is the coarse moduli space of $\mathscr{M}$, which is exhibited as an open subset of $\mathcal{M}^M$.
This morphism comes from the natural morphism $\mathscr{M}\to \mathscr{D}^*$ of Artin stacks.
Let $\mathscr{M}^{*}$ be the normalization of the scheme-theoretic closure of the image as an Artin stack.
Since $\mathscr{M}^*\to \mathscr{D}^*$ is finite, we see that $\mathscr{M}^*$ also admits a good moduli space $\mathcal{M}^*$ finite over $\mathcal{D}^*$.
This is nothing but $\mathcal{M}^*$ appeared in \cite{HeckmanLooijenga2002}.

\subsection{Kontsevich compactification}
\label{sec:Kontsevich compactification}

We recall the definition of the Kontsevich compactification of $\mathcal{M}$, which is named in \cite{HeckmanLooijenga2002}.
Consider $\mathcal{M}^{\mathrm{stqmaps}}_{12,1,12,1}$ the moduli stack.
This contains the closed substack that parameterizes all $J\colon (P,D)\to \mathbb{P}^1$ such that $D\subset J^{-1}(\infty)$ with reduced structure.
% Take the irreducible component of the normalization of this stack that is birational to $\mathscr{M}^M$ and we write it as $\mathscr{M}^{K}$.
Let $\mathscr{M}^K$ be the normalization of the closure of the image of $\mathcal{M}$ in the Kontsevich moduli space of degree-twelve genus-zero stable maps to $\mathbb{P}^1$ with twelve unordered marked poles.
Let $\mathcal{M}^K$ be its coarse moduli space and then $\mathcal{M}^K$ coincides with the one introduced in \cite{HeckmanLooijenga2002}.
By Theorem \ref{thm:hm1-fromstable}, we see that there exist canonical morphisms $\rho\colon \mathcal{M}^K\to \mathcal{M}^M$ and $\sigma\colon\mathcal{M}^K\to\mathcal{M}^*$.
We see that $\sigma$ and $\rho$ are birational by \cite[Section 2]{HeckmanLooijenga2002}.

As explained in \cite[Section 3]{HeckmanLooijenga2002}, $\mathcal{M}^K$ contains $\mathcal{M}$ as an open subset and $\mathcal{M}^K\setminus\mathcal{M}$ consists of finitely many divisors and all of them are classified.
For each irreducible component of the boundary divisor, we will use the same notation as \cite{HeckmanLooijenga2002} (e.g. $\mathcal{M}^K(I_8')$).
According to \cite[Section 3.3]{HeckmanLooijenga2002}, $\mathcal{M}^K$ has boundary divisors, $\mathcal{M}^K(I_{k})$ for $2\le k\le9$, $\mathcal{M}^K(II)$, $\mathcal{M}^K(III)$, $\mathcal{M}^K(IV)$, $\mathcal{M}^K(I_0^*)$, and $\mathcal{M}^K(I^*_{k,k'})$ for $1\le k,k'\le4$.
By the canonical morphisms $\rho$ and $\sigma$, we can define the locus of the boundary for example $\mathcal{M}^M(I_8'):=\rho(\mathcal{M}^K(I_8'))$ and $\mathcal{M}^*(I_8'):=\sigma(\mathcal{M}^K(I_8'))$. 
See also the table of \cite[Section 3]{HeckmanLooijenga2002}.

\section{Heckman-Looijenga's ball quotient}
\label{sec:Heckman--Looijenga's ball quotient}
\subsection{Looijenga's compactification}
\label{sec:Looijenga's compactification}
We use the following notation for ball quotients in this paper.
Put $\omega=e^{2\pi i/6}$ and
$\theta:=\omega-\bar\omega=\sqrt{-3}$.
Let $\mathscr{O}\defeq \mathbb Z[\omega]$ be the ring of integers of $E:=\Q(\omega)$.
\begin{defn}
\label{def:unitary-heegner-divisors}
Let $(L,\psi_L)$ be an arbitrary Hermitian $\mathscr O$-lattice
of signature $(n,1),$ and put
\[
V_L:=L\otimes_{\mathscr O}\mathbb C,
\qquad
\mathbb B_L:=
\{[z]\in\mathbb P(V_L):\psi_L(z,z)<0\}.
\]
For a vector $r\in L\otimes_{\mathscr O}E$ satisfying
$\psi_L(r,r)>0$, define
\[
H_L(r):=\mathbb P(r^\perp)\cap\mathbb B_L.
\]
This is a totally geodesic hypersurface of $\mathbb B_L$,
which we call a \emph{hyperball} in the ball. 
\end{defn}

We recall the general theory of Looijenga's compactification of ball quotients \cite{Looijenga2003Ball}. 
 Let $\B_L$ be a complex ball, let $\Gamma$ be an arithmetic subgroup of its automorphism group, and put $X\defeq \Gamma\backslash \B_L$.
The automorphic $\Q$-line bundle of weight one on $X$ will be denoted by $\mathcal L$.  Its Baily-Borel compactification is denoted by $X^{\mathrm{bb}}$.
The boundary $X^{\mathrm{bb}}\setminus X$ consists of a finite set of cusps since the Hermitian symmetric domain $\B_L$ is associated with $\U(n,1)$ \cite{baily1966compactification}; see also \cite[Section 4.1]{Looijenga2003Ball}. 

Following \cite[Definition 4.2]{Looijenga2003Ball},
an \emph{arithmetic system}
is a $\Gamma$-equivariant assignment $I\mapsto j(I)$,
where $I$ ranges over the rational isotropic lines
and $j(I)\subset V_L$ is defined over $E$ and satisfies
$I\subseteq j(I)\subseteq I^\perp$.
The associated compactification is denoted by $X^j$,
which we call the \emph{semi-toroidal compactification}.
The choices $j(I)=I^\perp$ and $j(I)=I$
recover the Baily--Borel and toroidal compactifications,
respectively, respectively \cite[Definition 4.2]{Looijenga2003Ball}.  
%Thus the word ``semi-toroidal'' will always mean Looijenga's intermediate cusp modification $X^j$, not the final arrangement compactification.

An \emph{arithmetic arrangement} is a $\Gamma$-invariant collection
of hyperballs, that is codimension $1$ subballs, arising from rational hyperplanes in the
underlying Hermitian space, with finitely many
$\Gamma$-orbits and locally finite union in the ball.
When writing $I\subset H$, we use $H$ for the corresponding
linear hyperplane in the Hermitian space.
Now, let $\mathcal H$ be a $\Gamma$-invariant arithmetic arrangement 
For a rational isotropic line $I$ corresponding to a cusp of $X^{\mathrm{bb}}$, define
\[
        j_{\mathcal H}(I)
        \defeq
        I^\perp\cap \bigcap_{\substack{H\in\mathcal H\\ I\subset H}} H.\]
Looijenga proved that $j_{\mathcal H}$ is an arithmetic system and that $X^{j_{\mathcal H}}\to X^{\mathrm{bb}}$
is the smallest normalized blow-up which makes the fractional ideals defined by the arrangement divisors invertible.  In particular, the strict transforms of the arrangement hypersurfaces become $\mathbb{Q}$-Cartier divisors on $X^{j_{\mathcal H}}$ \cite[Lemma 5.2]{Looijenga2003Ball}.

Starting from the semi-toroidal compactification $X^{j_{\mathcal H}}$, one separates the arrangement strata by a further arrangement blow-up and then contracts by a semiample line bundle.  We shall use the notation
\[
        X^{\mathrm{bb}}
        \longleftarrow
        X^{j_{\mathcal H}}
        \longleftarrow
        \widetilde X^{\mathcal H}
        \longrightarrow
        \widehat X^{\mathcal H}.
\]
Here, $\widetilde X^{\mathcal H}$ is the arrangement blow-up and $\widehat X^{\mathcal H}$ is Looijenga's compactification of the arrangement complement
        $X^\circ \defeq X-\B_L(\mathcal H)_\Gamma$.
Here $\B_L(\mathcal H)\defeq \sum_{H\in\mathcal H}H$
denotes the reduced arrangement divisor on $\B_L$, and $\B_L(\mathcal H)_\Gamma$ denotes its image in $X=\Gamma\backslash\B_L$, endowed with the reduced structure.  Let us introduce the graded algebra 
\[
        A_{\mathcal H,\bullet}^{\Gamma}
        \defeq
        \bigoplus_{k\in\mathbb Z} A_{\mathcal H,k}^{\Gamma},
        \qquad
        A_{\mathcal H,k}^{\Gamma}
        \defeq
        H^0\bigl(\B_L^\circ,
        \mathcal O_{\B_L^\circ}^{\mathrm{an}}(\mathcal L^{\otimes k})\bigr)^\Gamma,
\]
where, by abuse of notation, $\mathcal L$ is also the automorphic line bundle on $\B_L$ and  $\B_L^\circ\defeq \B_L\setminus \bigcup_{H\in\mathcal H}H$.
Thus, the degree-$k$ part consists of meromorphic automorphic forms, whose possible poles are supported on the arrangement $\mathcal {H}$.
% \Yota{Need to check; compare this with \cite[Section 2.2]{Looijenga2003Ball}.}
Assume that every nonempty intersection of members of
$\mathcal H$ in $\mathbb B_L$ has complex dimension at least two.
Then $A_{\mathcal H,\bullet}^{\Gamma}$
is finitely generated with positive-degree generators, and $\widehat X^{\mathcal H}
 \cong\Proj A_{\mathcal H,\bullet}^{\Gamma}$
by \cite[Corollary 5.8]{Looijenga2003Ball}.
%\Yota{Need to be defined}
%; on the semi-toroidal model its degree-$m$ part is the section space with pole order bounded by $m$ times the reduced arrangement divisor after coherent pull-back.  Thus this is the bounded-pole arrangement algebra, not the larger union of all finite pole orders in each fixed automorphic weight.  
The construction leaves $X^\circ$ unchanged, but changes the boundary and the arrangement strata in a way controlled by the intersection lattice of $\mathcal H$ \cite[Theorem 5.7]{Looijenga2003Ball}.

% In the rational elliptic surface case below, the arrangement relevant for Miranda's compactification is not the full mirror-trace arrangement.  After the four mirror-trace components have been introduced, we denote the Miranda arrangement by $\mathcal K=\mathcal H_o(6)\cup\mathcal H_o(9)$.  The corresponding semi-toroidal compactification will be denoted by $X_o^{j_{\mathcal K}}$, and the final Looijenga compactification by $\widehat X_o^{\mathcal K}$.  Looijenga's comparison theorem identifies $\widehat X_o^{\mathcal K}$ with Miranda's GIT compactification of rational elliptic surfaces.

\subsection{The Deligne-Mostow ball quotient and rational elliptic surfaces}
We recall the setting in \cite[Sections 7.3, 7.4]{Looijenga2003Ball}.
Let $P$ be a projective line $\mathbb{P}^1$ and put
\[
        H_k \defeq H^0(P,\mathcal O_P(k)),\qquad  P_k\defeq \mathbb P(H_k).
\]
Thus $P_{12}$ parametrizes effective divisors of degree twelve
on $P$. Denote its GIT-stable and semistable loci by
$P_{12}^{\mathrm s}$ and $P_{12}^{\mathrm{ss}}$, respectively.
The GIT quotient
$\mathcal{D}^*=\SL(H_1)\backslash\!\backslash P_{12}^{\mathrm{ss}}$
compactifies the geometric quotient
$\SL(H_1)\backslash P_{12}^{\mathrm s}$.
Its complement consists of a single point, represented by
$6p+6q$ for distinct points $p,q\in P$.
To describe the period map, let $D\subset P$ be a reduced divisor of degree $12$.
Consider the cyclic cover $C\to P$ of degree six totally
ramified along $D$.
Choose the deck transformation $\tau$ whose action on the
tangent space at each ramification point is multiplication
by $\omega$.
The {\it Eisenstein homology} of $C$ is the torsion-free quotient
\[
H_1(C,\mathbb Z)_{\mathscr O}
:=
\left(
H_1(C,\mathbb Z)
\otimes_{\mathbb Z[\langle\tau\rangle]}\mathscr O
\right)/\mathrm{torsion},
\]
where $\tau$ acts on homology by $\tau_*$ and on
$\mathscr O$ by multiplication by $\omega$.
This is a free $\mathscr O$-module of rank ten (see \cite[Section 4.2]{HeckmanLooijenga2002}).
The intersection pairing induces a Hermitian form on it
as described in \cite[Section 4.1]{HeckmanLooijenga2002}.
We denote the resulting
Hermitian lattice by $(A,\psi)$.  
We retain the normalization in which roots have
$\psi$-norm $3$. See \cite[Section 4]{HeckmanLooijenga2002} for more details.
We list the positive index first in the signature.
Thus $\psi$ has signature $(9,1)$, and the associated
ball parametrizes negative lines.

% Let $\Gamma\subset\mathrm U(L)$ be an arithmetic subgroup,
% and let
% $\pi_\Gamma:\mathbb B_L\to X_\Gamma
% :=\Gamma\backslash\mathbb B_L$
% be the quotient map.
% We call
% \[
% Z_\Gamma(r):=
% \bigl(\pi_\Gamma(H_L(r))\bigr)_{\mathrm{red}}
% \]
% the \emph{Heegner divisor} attached to the
% $\Gamma$-orbit of the rational line $Er$.
% In particular, $H_L(ar)=H_L(r)$ for every $a\in E^\times$.
Put
\[
X_A:=\mathrm U(A)\backslash\mathbb B_A.
\]
We denote its Baily-Borel compactification by
$X_A^{\mathrm{bb}}$.
This is called the ancestral Deligne-Mostow ball quotient \cite{DDH18,HKM24}; all other Eisenstein Deligne-Mostow ball quotients arise as sub-ball quotients of it and have attracted considerable interest in their own right \cite{HM25Universe}. 
Define
\[
\mathcal H_A:=
\{H_A(r):r\in A,\ \psi(r,r)=3\}.
\]
Here $\mathcal H_A$ is a set of distinct hyperballs:
vectors defining the same hyperball do not give
different members of the arrangement.
The associated reduced arrangement divisor is $\mathbb B_A(\mathcal H_A)=
\sum_{H\in\mathcal H_A}H$.
This is an arithmetic arrangement in the sense of
\cite[Example 1.8 and Lemma 5.4]{Looijenga2003Ball}.

By the Deligne-Mostow theory, the period map identifies
$\SL(H_1)\backslash P_{12}^{\mathrm s}$ with $X_A$
and extends to an isomorphism
$\mathcal{D}^*\cong X_A^{\mathrm{bb}}$
\cite[Theorems 7.7 and 7.8]{HeckmanLooijenga2002};
see also \cite{DeligneMostow1986,Mostow1986generalized}.
The class represented by $6p+6q$ maps to the unique cusp.
On the stable locus, the locus of nonreduced divisors
corresponds to the reduced image of
$\B_A(\mathcal H_A)$ in $X_A$.

% A rational Weierstrass fibration is represented by a Weierstrass equation 

For the Weierstrass pairs $(f_0,f_1)$ introduced in Subsection \ref{subsec-weierstraa-explanation}, write
\[
\Delta_{\mathrm{Wei}}(f_0,f_1):=[f_0^3+f_1^2]\in P_{12}\cup\{0\}.
\]
This defines a rational map
$\Delta_{\mathrm{Wei}}\colon Y_{\mathrm{Mir}}\dashrightarrow P_{12}$,
whose domain of definition consists of the pairs with
$f_0^3+f_1^2\not\equiv0$.
% \[
%         y^2=x^3+3f_0(t)x+2f_1(t),
%         \qquad f_0\in H_4,\quad f_1\in H_6,
% \]
% with discriminant
% \[
%         \Delta_{\mathrm{Wei}}(f_0,f_1)=f_0^3+f_1^2\in H_{12}.
% \]
Thus the discriminant of a rational elliptic surface gives a degree $12$ divisor on $P$, but it satisfies the additional condition of being expressible as a sum of a cube and a square.  On the ball quotient side this condition cuts out a hyperball.  Namely, Heckman-Looijenga \cite[Theorem 8.2]{HeckmanLooijenga2002} construct a distinguished norm-$6$ vector $z_o\in A$ and put $A_o\defeq z_o^\perp\subset A$.
The lattice $A_o$ has signature $(8,1)$, and so determines an $8$-dimensional ball $\B_{A_o}\subset \B_A$. For simplicity, we denote it by $\B_o$.  We denote the corresponding arithmetic quotient, \emph{Heckman-Looijenga's ball quotient}, by
\[
        \Gamma_o\defeq \U(A_o),\qquad X_o\defeq \Gamma_o\backslash \B_o .
\]

\subsection{Arrangement divisors}
The arrangement $\mathcal H_A$ on $\B_A$ restricts to an arrangement on the hyperball $\B_o$.  Heckman-Looijenga call the hyperplanes arising in this way \emph{mirror traces} \cite[Section 9]{HeckmanLooijenga2002}.  We denote the collection of these hyperplanes by $\mathcal H_o$.
More explicitly, put
\[
\mathcal H_o:=
\{H_A(r)\cap\mathbb B_o:
  r\in A,\ \psi(r,r)=3,\
  H_A(r)\cap\mathbb B_o\ne\emptyset\},
\]
again regarded as a set of distinct hyperballs.

For $H=H_A(r)\cap\mathbb B_o\in\mathcal H_o$, consider  an $\mathscr O$-basis $(e_1,e_2)$ of 
the primitive rank-two normal lattice $N_H:=A\cap(Ez_o+Er)$, the saturated lattice of $\mathscr{O} z_o+\mathscr{O} r$, and 
set
\[
d(H):=\det\bigl(\psi(e_i,e_j)\bigr)_{i,j=1}^{2},
\]
which is independent of the choice of the basis.
We call $d(H)$ the \emph{$d$-invariant} of $H$.
By \cite[Lemma A.9 and Proposition A.12]
{HeckmanLooijenga2002},
the possible values are $6,9,15,18$.

Unlike the original arrangement in $\B_A$, the arrangement $\mathcal H_o$ has four $\Gamma_o$-orbits, distinguished by the invariant $d(H)\in\{6,9,15,18\}$ attached to a mirror trace $H$.  We write
$\mathcal H_o(d)\defeq \{H\in \mathcal H_o : d(H)=d\}$ and $\mathcal H_o(S)\defeq \bigcup_{d\in S}\mathcal H_o(d)$ for $S\subset\{6,9,15,18\}$.  We use the convention
\[
        \B_o(\mathcal A)\defeq \sum_{H\in\mathcal A}H
\]
for any locally finite arrangement $\mathcal A$ on $\B_o$.  The corresponding divisors on the ball quotient are denoted by
\[
        \Delta(d)\defeq \Gamma_o\backslash \B_o(\mathcal H_o(d)),
        \qquad d=6,9,15,18.
\]
Their closures in $X_o^{\mathrm{bb}}$ will be denoted by the same symbols when no confusion is possible. Finally, we put
\[
        \mathcal K\defeq \mathcal H_o(\{6,9\})=\mathcal H_o(6)\cup\mathcal H_o(9)
\]
and call it the \emph{Miranda arrangement}.

Let $\M$ be the moduli space of rational
Weierstrass fibrations with reduced discriminant.  The period map identifies it with the complement of the full mirror-trace arrangement:
\begin{equation}\label{eq:open-res-ball-quotient}
        \M \cong X_o-\bigl(\Delta(6)\cup\Delta(9)\cup\Delta(15)\cup\Delta(18)\bigr).
\end{equation}
The compactification $\M^*$ considered in \cite{HeckmanLooijenga2002} is the normalization of the closure of $\M$ inside $\mathcal{D}^*$.  Under the period map it is identified with the Baily-Borel compactification of $X_o$:
        \[\M^*\cong X_o^{\mathrm{bb}}\]
by \cite[Theorem 9.2]{HeckmanLooijenga2002}.
% For a degeneration label $F$ in
% \cite[Section 3.3]{HeckmanLooijenga2002}, we write
% \[
%  \mathcal M^*(F)
%  :=\bigl(\sigma(\mathcal M^K(F))\bigr)_{\mathrm{red}},
% \]
% where $\sigma\colon\mathcal M^K\to\mathcal M^*$ is the
% morphism introduced in Section \ref{sec:Kontsevich compactification} and $\mathcal M^K(F)$
% is the boundary divisor whose general point represents
% a stable limit of type $F$ \Masa{I explained before.}.
Under this isomorphism, the four divisorial boundary components of $\M$ in $\M^*$ are matched with $\M^*(\bullet)$ given as above as follows \cite[Theorem 9.3 and Example 10.3]{HeckmanLooijenga2002}:
\[
\begin{array}{c|c}
\text{boundary component in }\M^* & \text{mirror-trace component} \\
\hline
\M^*(I_2) & \Delta(18)\\
\M^*(II) & \Delta(15)\\
\M^*(I_{4,4}^*) & \Delta(9)\\
\M^*(I_9) & \Delta(6).
\end{array}
\]
Moreover, the two zero-dimensional Baily-Borel boundary strata correspond to $\M^*(I_6)$ and $\M^*(I_0^*)$ \cite[Theorem 9.3]{HeckmanLooijenga2002}.

\begin{rem}
    We note that the canonical morphism $X_o^{\mathrm{bb}}\to X_A^{\mathrm{bb}}$ is not a closed immersion but finite and birational to the image as shown by Vakil \cite[Theorem 6.1]{vakil2001twelve}.
    Therefore, when we define $\mathcal{M}^*$, it is necessary to take the normalization of the closure of the image.
\end{rem}

Set $X_o^\circ:=X_o\setminus(\Delta(6)\cup\Delta(9))$.
The period map induces an isomorphism
\begin{equation}\label{eq:common-open-res}
\operatorname{PGL}_2\backslash
\left(
  Y_{\mathrm{Mir}}^{\mathrm s}
  \cap\Delta_{\mathrm{Wei}}^{-1}(P_{12}^{\mathrm s})
\right)
\cong X_o^\circ;
\end{equation}
see \cite[Sections 7.4-7.5]{Looijenga2003Ball}.
The inverse image is taken on the domain of definition
of $\Delta_{\mathrm{Wei}}$.
Thus, both the Weierstrass pair and its discriminant
are required to be GIT-stable.
Via this identification, we also regard $X_o^\circ$
as an open subset of $\mathcal M^M$.
The corresponding rational Weierstrass fibrations have only singular fibers of type $I_k$ with
$1\leq k\leq5$, $II$, $III$, and $IV$ (see \cite{HeckmanLooijenga2002,Miranda1981}).

We fix the normalization of the automorphic line
bundle. Let $h\colon X_o^{\mathrm{bb}}\to \mathcal{D}^*$
be the extended discriminant morphism, and put
$\xi:=\overline{\mathcal O}_{\mathcal{D}^*}(1)$.
Here $\xi$ is the GIT polarization descending from
the usual polarization $\mathcal O_{P_{12}}(1)$.
Let $\lambda_A$ be the natural automorphic $\Q$-line bundle
on the Baily-Borel
compactification of the ancestral ball quotient, and set
\[
 \lambda_{\rm aut}:=h^*\lambda_A=\mathcal L.
\]
On the ball quotient side, the natural automorphic line bundle is the restriction
of $\mathcal O(-1)$ on the compact dual of the ball.
Thus an automorphic form of weight $k$ is homogeneous of degree $-k$.
 In particular, the canonical weight, which contributes to the expression of the canonical bundle as a coefficient of $\mathcal{L}$, 
 is $n+1$ if the dimension of the ball is $n$.

% For a binary duodecic $F$, consider
% \[
%  C_F:\ y^6=F(x),
%  \qquad
%  \alpha_F=\frac{dx}{y}.
% \]
% The differential $\alpha_F$ spans the one-dimensional
% primitive Hodge eigenspace
 
Let $F$ be a binary duodecic with reduced zero divisor,
and let $C_F$ be the smooth projective curve with affine
equation $y^6=F(x)$.
Projection to the $x$-coordinate defines a cyclic cover
$\pi_F\colon C_F\to\mathbb{P}^1$ of degree six.
A deck transformation of this cover is an automorphism
of $C_F$ over $\mathbb{P}^1$.
The deck transformation group is cyclic of order six, generated by $\tau(x,y)=(x,\omega y)$.
The period line is the one-dimensional eigenspace \cite[Lemma 7.1]{HeckmanLooijenga2002} 
\[
\{\alpha\in H^0(C_F,\Omega^1_{C_F}):
  \tau^*\alpha=\omega^{-1}\alpha\}
=\mathbb C\alpha_F,
\qquad
\alpha_F:=\frac{dx}{y}.
\]
For $c\in\mathbb C^\times$, choose a sixth root $c^{1/6}$.
The isomorphism
$C_F\to C_{cF}$ given by $(x,y)\mapsto(x,c^{1/6}y)$
pulls $\alpha_{cF}$ back to $c^{-1/6}\alpha_F$.
Consequently, coefficient degree $r$ corresponds to
natural automorphic weight $6r$.
% defining the period line.
% Under $F\mapsto cF$ \Masa{What is $c$?}, the identification
% $y\mapsto c^{1/6}y$ gives
% $\alpha_{cF}=c^{-1/6}\alpha_F$
% after transport to $C_F$.
% Consequently, coefficient degree $r$ corresponds to
% natural automorphic weight $6r$.
% Equivalently, $\alpha_F^{\otimes6}=(dx)^6/F(x)$
% identifies the sixth power of the period line with
% $H^0(\mathbb P^1,\omega_{\mathbb P^1}^{\otimes6}(D_F))$.
% This is the equivariant construction of
% \cite[Section 7]{HeckmanLooijenga2002}. \Masa{Would it construct a correspondence of the period?}
% The comparison extends over the stable discriminant locus
% and then to the Baily-Borel compactification, since the
% remaining boundary has codimension at least two.
% Hence
% \begin{equation}\label{eq:O1-automorphic-normalization}
%  h^*\xi\sim_{\Q}6\lambda_{\rm aut}.
% % \end{equation}
% \Yota{I added an explanation for the coincidence, referring to HL.}
% Choose a marking of the primitive Eisenstein cohomology. 
Put $E_F:=
\ker\bigl(\tau^*-\omega^{-1}\mathrm{id}\bigr)
\subset H^1(C_F,\mathbb C)$.
Choose an identification of the integral first cohomology
of $C_F$ with that of a fixed reference cover, preserving
the intersection pairing and commuting with the deck action.
Such a marking identifies $E_F$ with $V_A=A\otimes_{\mathscr{O}}\mathbb{C}$, and changes
of marking act through $\mathrm U(A)$;
see \cite[Section 7]{HeckmanLooijenga2002}.
The cohomology class of $\alpha_F$ determines a negative
line in $E_F$, hence a point of $X_A$ independent of
the marking.
Under this period construction, the Hodge eigenline
$\mathbb C\alpha_F$ is identified with the tautological
automorphic line bundle.
% The cohomology class of $\alpha_F$ then spans a negative
% line in $V_A$, and its $\mathrm U(A)$-orbit is independent of the marking. 
% % The $\mathrm U(A)$-orbit of this negative line is the image
% % of the divisor defined by $F$ under the Deligne-Mostow
% % period map recalled at the beginning of this section.
% Its equivariant lift identifies the Hodge eigenline
% $\mathbb C\alpha_F$ with the tautological automorphic line.
The sixth tensor power $\alpha_F^{\otimes6}$ descends to
\[
 \frac{(dx)^6}{F(x)}
 \in H^0\bigl(
   \mathbb P^1,
   \omega_{\mathbb P^1}^{\otimes6}(\operatorname{div}(F))
 \bigr).
\]
By \cite[Lemma 7.2]{HeckmanLooijenga2002}, this identifies
the GIT polarization $\xi$ with the sixth power of the
automorphic line on the reduced-discriminant locus.

By \cite[Theorem 7.7]{HeckmanLooijenga2002},
the equivariant period isomorphism extends this comparison
over the GIT-stable locus.
The resulting identity of $\Q$-line bundles then extends
across the Baily-Borel boundary, which has codimension
at least two.
Pulling back by $h$, we obtain
\begin{equation}\label{eq:O1-automorphic-normalization}
 h^*\xi\sim_{\Q}6\lambda_{\rm aut}.
\end{equation}

To compare the Hodge line bundle with the GIT polarization, let $\mathscr D_{\rm univ}$ be
the universal divisor on $\mathbb P^1\times P_{12}^{\rm ss}$,
with projections $p_1,p_2$.
Since it is well-known that
\[
 \mathcal O(\mathscr D_{\rm univ})
 \cong
 p_1^*\mathcal O_{\mathbb P^1}(12)
 \otimes p_2^*\mathcal O_{P_{12}}(1),
\]
we have
\[
 \omega_{\mathbb P^1\times P_{12}^{\rm ss}/P_{12}^{\rm ss}}^{\otimes6}
 (\mathscr D_{\rm univ})
 \cong p_2^*\mathcal O_{P_{12}}(1).
\]
At $t=0$, the relevant boundary-polarized CY
family is
$(\mathbb P^1,\frac16\mathscr D_{\rm univ})$,
with quasimap weight zero.
Applying Definition \ref{defn:Hogde-HM2},
descending to $\mathcal{D}^*$, and pulling back by $h$, we obtain
\begin{equation}\label{eq:O1-Hodge-normalization}
 6\Lambda_{\rm H}\sim_{\Q}h^*\xi,
 \qquad
 \Lambda_{\rm H}:=\Lambda_{\mathrm{Hodge},0}.
\end{equation}
Combining the two comparisons, we obtain
\begin{equation}\label{eq:CY-Hodge-equals-natural-automorphic}
 \Lambda_{\rm H}\sim_{\Q}\lambda_{\rm aut}.
\end{equation}

Let $\eta$ be the rational GIT polarization on $\mathcal M^M$
induced by $\mathcal O_{Y_{\mathrm{Mir}}}(1)$.
More precisely, for a sufficiently divisible positive integer
$m$, the line bundle
$\mathcal O_{Y_{\mathrm{Mir}}}(m)|_{Y_{\mathrm{Mir}}^{\mathrm{ss}}}$
descends to a line bundle representing $m\eta$.
We use the grading in which the coefficients of $f_0$
and $f_1$ have weights two and three, respectively.

Under \eqref{eq:common-open-res}, the restriction
$h|_{X_o^\circ}$ is the morphism induced by
$\Delta_{\mathrm{Wei}}$.
Since $f_0^3+f_1^2$ has weighted degree six,
\[
  (h|_{X_o^\circ})^*\xi
  \sim_{\Q}6\eta|_{X_o^\circ}.
\]
Combining this with
\eqref{eq:O1-automorphic-normalization}, we obtain
$\eta|_{X_o^\circ}
  \sim_{\Q}\lambda_{\rm aut}|_{X_o^\circ}$.

\subsection{Cusp behaviour}

We apply the construction of
Subsection \ref{sec:Looijenga's compactification}
to the Miranda arrangement
$\mathcal K=\mathcal H_o(6)\cup\mathcal H_o(9)$.
The resulting arithmetic system is denoted by $j_{\mathcal K}(I)=I^\perp\cap\bigcap_{H\in\mathcal K,\ I\subset H}H$.
The semi-toroidal modification associated with $\mathcal K$ is the morphism $X_o^{j_{\mathcal K}}\to X_o^{\mathrm{bb}}$. Recall that it is the smallest normalized blow-up for which the strict transforms of $\Delta(6)$ and $\Delta(9)$ are $\Q$-Cartier.
For the rest of this paper, we write $D_d:=\Delta(d)^{j_{\mathcal K}}$ for the strict transform of $\Delta(d)$ on $X_o^{j_{\mathcal K}}$, and use the same notation $\lambda_{\rm aut}$ and $\Lambda_{\rm H}$ for their pullbacks to $X_o^{j_{\mathcal K}}$.

The Baily-Borel compactification $X_o^{\mathrm{bb}}$ has two cusps, denoted by $\infty_\theta$ and $\infty_0$.  They correspond to the two $\U(A_o)$-orbits of primitive isotropic lines of types $(\theta)$ and $(0)$ classified in \cite[Proposition A.6]{HeckmanLooijenga2002}.  Under the modular interpretation, they are the images of the $I_6$- and $I_0^*$-strata.

\begin{prop}
\label{prop:cusp-incidence-all-components}
The cusp $\infty_\theta$ belongs to all four divisors $\Delta(d)^{\rm bb}$, the closure of $\Delta(d)$ in $X_o^{\rm bb}$, with $d=6,9,15,18$.  The cusp $\infty_0$ belongs to $\Delta(9)^{\rm bb}$, $\Delta(15)^{\rm bb}$, and $\Delta(18)^{\rm bb}$, but not to $\Delta(6)^{\rm bb}$.
\end{prop}

\begin{proof}
Let $I=\mathscr{O}e$ represent the cusp.  A mirror trace defined by a norm-three vector $r$ contains $I$ if and only if $r\in I^\perp$, and the image of $r$ in $I^\perp/I$ has the same $d$-invariant as $r$.

For a cusp of type $(\theta)$, take the standard representative $z_o=r_1'+r_2'$ and $I_A^\perp/I\cong\Lambda'\perp\Lambda''$.    \cite[Proposition A.10]{HeckmanLooijenga2002} gives vectors of $d$-invariants $6$ and $9$.  Since $\psi(z_o,r_3')=\sqrt{-3}$, Lemma A.9 of loc. cit. gives a vector of $d$-invariant $15$, while every norm-three vector in $\Lambda''$ is orthogonal to $z_o$ and has $d$-invariant $18$.  Thus all four divisors contain $\infty_\theta$.

For a cusp of type $(0)$, take $z_o=r_1'+r_1''$ and $I_A^\perp/I\cong\Lambda'\perp\Lambda''$.  The vector $r_1'$ has $d$-invariant $9$, the equality $\psi(z_o,r_2')=\theta$ gives $d$-invariant $15$, and $\psi(z_o,r_3')=0$ gives $d$-invariant $18$.  Finally, Proposition A.10 of loc. cit. gives $I(6)=I$ in their notation, so no mirror trace of $d$-invariant $6$ contains $\infty_0$.
\end{proof}

For a primitive isotropic line $I\subset A_o$ and $d\in\{6,9\}$, put 
\[V_{d,I}:=\operatorname{Span}_{\C}\{[\operatorname{pr}_{z_o^\perp}(r)]\in \bigl(I_{A_o}^{\perp}/I\bigr)
\otimes_{\mathscr O}\mathbb C : r\in A\cap I_A^{\perp},\ \psi(r,r) = 3,\ d(H_A(r)\cap \mathbb{B}_o) = d\}.\]
Choose a normal neat subgroup $\Gamma'\subset\Gamma_o$ of finite
index. Let $I$ represent a cusp, and let $\Gamma'_I$ be its
stabilizer in $\Gamma'$.
For the rational intermediate subspaces
$I\subseteq J\subseteq I^\perp$ occurring below, write
\[
 X(J):=
 \Gamma'_I\backslash
 \left(
 \mathbb P(V_{A_o}/J)
 \setminus \mathbb P(I^\perp/J)
 \right).
\]
This is an abelian torsor, as in
\cite[Section 4.3]{Looijenga2003Ball}.
Unless another intermediate subspace is specified, we take
$J=j_{\mathcal K}(I)$.

On this neat cover, $X(I)$ is the toroidal boundary stratum,
and the natural morphism $X(I)\to X(J)$ gives the intermediate
cusp modification.
The fibre over the corresponding cusp of $X_o^{\mathrm{bb}}$
is a finite quotient of $X(J)$.

For $H\in\mathcal K$ containing $I$, the boundary trace of its
strict transform is a fibre of $X(J)\to X(H\cap I^\perp)$,
whose target is a genus-one curve.
After choosing origins, this trace is a translate of the kernel
of a homomorphism to an elliptic curve.
Write $H=H_A(r)\cap\mathbb B_o$ and put
$r_o:=\operatorname{pr}_{z_o^\perp}(r)$.
Since $r_o$ is orthogonal to $I$, the formula $\ell_{H,I}([v]):=\psi(r_o,v)$
defines a complex-linear functional on
$(I_{A_o}^{\perp}/I)\otimes_{\mathscr O}\mathbb C$.
It is determined by $H$ up to a nonzero scalar.
Since $J\subset H$, this functional further descends
to $I^\perp/J$.

After choosing a nonzero vector in $V_{A_o}/I^\perp$,
the tangent space of $X(J)$ is identified with
$I^\perp/J$.
Under this identification, the descended functional
represents, up to a nonzero scalar, the differential
of $X(J)\to X(H\cap I^\perp)$;
see \cite[Section 4.3 and Lemma 5.2]{Looijenga2003Ball}. 
Restrictions of $D_d$ to $X(J)$ mean restrictions of their
pullbacks to this finite cover.
% Whenever we study a cusp fibre, we pass to a neat finite index subgroup of $\Gamma_o$ and denote by $\widetilde B_I$ the abelian torsor appearing in Looijenga's intermediate cusp modification. 
  
% For every mirror trace $H$ containing $I$, the closure of $H$ in the toroidal compactification cuts $\widetilde B_I$ in a translate of the kernel of a homomorphism from the abelian variety acting on $\widetilde B_I$ to an elliptic curve; its differential is, up to a nonzero scalar, the normal functional defined by $H$ \cite[Section 4 and Lemma 5.2]{Looijenga2003Ball}.

\begin{lem}
\label{lem:abelian-kernel-divisors}
Let $A$ be an abelian variety and let $\varphi_i\colon A\to E_i$, $1\leq i\leq r$, be nonzero homomorphisms to elliptic curves.  For points $a_i\in E_i$ and positive integers $m_i$, the divisor $\sum_i m_i\varphi_i^{-1}(a_i)$ is nef.  Its connected translation stabilizer is $\bigl(\bigcap_i\ker\varphi_i\bigr)^0$, and $\sum_i m_i\varphi_i^{-1}(a_i)$ is ample if and only if $\bigcap_i\ker\varphi_i$ is finite, equivalently, if and only if the differentials $d\varphi_i$ span the cotangent space of $A$.
Here, $(\bullet) ^0$ denotes the identity component.
\end{lem}

\begin{proof}
Let $\Phi=(\varphi_1,\ldots,\varphi_r)$ and let $p_i$ denote the $i$-th projection from $\prod_iE_i$.  The divisor $\sum_i m_i p_i^*(a_i)$ is ample on $\prod_iE_i$, and its restriction to the abelian subvariety $\Phi(A)$ is ample.  Its pullback to $A$ is the divisor in the statement, so that divisor is nef.  The image of its connected translation stabilizer under $\Phi$ lies in the translation stabilizer of an ample divisor on $\Phi(A)$, which is finite.  Hence the connected stabilizer is exactly $(\ker\Phi)^0$.  It follows that the divisor is ample when $\ker\Phi$ is finite, because then $\Phi$ is finite onto its image; conversely, ampleness forces its translation stabilizer, and hence $\ker\Phi$ to be finite.  The last equivalence follows by taking tangent spaces at the origin.
\end{proof}

\begin{prop}
\label{prop:cusp-fibres-corrected}
Let $I$ represent a cusp and put $J=j_{\mathcal K}(I)$.
If $I$ is of type $(\theta)$, the restrictions
$D_6|_{X(J)}$ and $D_9|_{X(J)}$ are nef but not ample,
whereas $(D_6+D_9)|_{X(J)}$ is ample.
If $I$ is of type $(0)$, then $X(J)$ is a curve and
$D_9|_{X(J)}$ has positive degree.
\end{prop}

\begin{proof}
Suppose first that $I$ is of type $(\theta)$.
By \cite[Proposition A.10]{HeckmanLooijenga2002},
\[
 \dim V_{6,I}=1,\qquad
 \dim V_{9,I}=2,\qquad
 V_{6,I}\perp V_{9,I}.
\]
On $I^\perp/I$, the common kernel of the functionals
$\ell_{H,I}$, for $H\in\mathcal K$ containing $I$,
is $J/I$ by the definition of $J$. Hence the abelian variety acting on $X(J)$ has dimension
three.

After multiplication by a positive integer, each restriction
$D_d|_{X(J)}$ is a positive sum of translates of kernels of
homomorphisms to elliptic curves.
The corresponding differentials span a space of dimension
$\dim V_{d,I}$.
Lemma \ref{lem:abelian-kernel-divisors} therefore shows that
$D_6|_{X(J)}$ and $D_9|_{X(J)}$ are nef, with connected
translation stabilizers of dimensions two and one, respectively.
Neither restriction is ample.
The combined differentials span the cotangent space of the
acting abelian variety, so the same lemma shows that
$(D_6+D_9)|_{X(J)}$ is ample.

If $I$ is of type $(0)$, then
$V_{6,I}=0$ and $\dim V_{9,I}=1$ by
\cite[Proposition A.10]{HeckmanLooijenga2002}.
Thus $X(J)$ is one-dimensional, and the differentials associated
with $D_9$ span its cotangent space.
Lemma \ref{lem:abelian-kernel-divisors} implies that
$D_9|_{X(J)}$ is ample.
The numerical assertions descend through the finite quotient.
\end{proof}

In particular,
\begin{equation}
 \dim X(J)=
 \begin{cases}
  3,&\text{if $I$ is of type $(\theta)$},\\
  1,&\text{if $I$ is of type $(0)$},
 \end{cases}
 \qquad J=j_{\mathcal K}(I).
 \label{eq:cusp-fibre-dimensions}
\end{equation}

\begin{cor}
\label{cor:cusp-intersection-matrix-corrected}
There are curves $C_{\theta,6}$ and $C_{\theta,9}$ contracted by $X_o^{j_{\mathcal K}}\to X_o^{\mathrm{bb}}$ over $\infty_\theta$, and a contracted curve $C_0$ over $\infty_0$, such that
\begin{equation}
\begin{array}{c|ccc}
 &C_{\theta,6}&C_{\theta,9}&C_0\\ \hline
D_6&>0&0&0\\
D_9&0&>0&>0.
\end{array}
\label{eq:cusp-intersection-matrix-corrected}
\end{equation}
In particular, $D_6$ and $D_9$ are relatively nef, neither is relatively ample, and $D_6+tD_9$ is relatively ample for $X_o^{j_{\mathcal K}}\to X_o^{\mathrm{bb}}$ for any $t\in\mathbb{R}_{>0}$.
\end{cor}

\begin{proof}
Let $I$ be of type $(\theta)$ and put $J=j_{\mathcal K}(I)$.
Choose a point of $X(J)$ outside the supports of the pullbacks
of $D_6$ and $D_9$.
Its orbit under the one-dimensional connected translation
stabilizer of $D_9|_{X(J)}$ is a complete curve.
Take its finite image as $C_{\theta,6}$.
The restriction of $D_9$ to this orbit is trivial, while
the ampleness of $(D_6+D_9)|_{X(J)}$ gives
$D_6\cdot C_{\theta,6}>0$.
Similarly, a complete curve in the orbit of the connected
translation stabilizer of $D_6|_{X(J)}$ gives $C_{\theta,9}$
with
$D_6\cdot C_{\theta,9}=0$ and
$D_9\cdot C_{\theta,9}>0$.

The fibre over $\infty_0$ is a curve.  Proposition \ref{prop:cusp-fibres-corrected} gives $D_9\cdot C_0>0$, whereas $D_6\cdot C_0=0$ because $\Delta(6)$ does not meet $\infty_0$ by Proposition \ref{prop:cusp-incidence-all-components}.  The restrictions of $D_6$ and $D_9$ to every positive-dimensional fibre are nef, and their sum is ample; hence the relative nefness and ampleness assertions follow from the fibrewise criterion for a projective morphism (cf. \cite[Theorem 1.44]{komo}).
\end{proof}

\subsection{From the Baily-Borel model to Miranda's compactification}
\label{subsec:looijenga-res-construction}
We now apply Looijenga's construction to the common
open subset $X_o^\circ$ in \eqref{eq:common-open-res}.
For the Miranda arrangement
$\mathcal K=\mathcal H_o(6)\cup\mathcal H_o(9)$,
the construction gives
\begin{equation}\label{eq:looijenga-four-models-res}
 X_o^{\mathrm{bb}}
 \xleftarrow{c}X_o^{j_{\mathcal K}}
 \xleftarrow{\pi}\widetilde X_o^{\mathcal K}
 \xrightarrow{r}\widehat X_o^{\mathcal K}.
\end{equation}
The morphism $c$ modifies the cusps so that the arrangement
divisors become $\Q$-Cartier.
The morphism $\pi$ separates their intersection strata.
The morphism $r$ is defined by a globally generated line bundle
obtained as some multiple of the arrangement sheaf, as described below.

By \eqref{eq:cusp-fibre-dimensions}, the two cusp fibres
of $c$ have dimensions three and one.
Since $c$ is an isomorphism away from the cusps and
$\dim X_o=8$, it is small.
Following \cite[Section 4.3]{Looijenga2003Ball},
we also write $\mathcal L$ for the canonical extension
of the automorphic line bundle to $X_o^{j_{\mathcal K}}$.
This extension is the pullback of $\mathcal L$ on
$X_o^{\mathrm{bb}}$ under $c$.

For each arrangement divisor, let $D_H^j$ be the strict
transform on $X_o^{j_{\mathcal K}}$ of its closure in $X_o^{\mathrm{bb}}$.
These are Cartier divisors on the chosen cover by
\cite[Lemma 5.2]{Looijenga2003Ball}.
As in \cite[Section 2.2]{Looijenga2003Ball}, put
\[
\mathcal F_{\mathcal K}
:=\sum_H\mathcal L(D_H^j),
\]
where the sum is taken inside the sheaf of rational
sections of $\mathcal L$ and runs over the distinct
arrangement divisors.

On the arrangement blow-up $\pi\colon \widetilde X_o^{\mathcal K}\to X_o^{j_{\mathcal K}}$, the sheaf
$\pi^*\mathcal F_{\mathcal K}$ 
is $\Q$-Cartier and semiample.
Its associated contraction is the lift of $r$;
see \cite[Theorem 5.7]{Looijenga2003Ball}.

% Locally on an arrangement chart, the second step uses the
% fractional ideal
% \[
%  \mathcal F_{\mathcal K}
%  :=\sum_H\mathcal L^j(D_H^j).
% \]
% \Masa{Where are the definitions of $\mathcal L^j$ or $(D_H^j)$ before?}
% We use the full arrangement blow-up: first the strata of
% type $(6,9,9)$, and then the strict transforms of the strata
% of types $(6,9)$ and $(9,9)$.
% Their codimensions are $3,2,2$, respectively.
% The coherent pullback \Masa{What is ``coherent pullback''} of $\mathcal F_{\mathcal K}$ to $Z$
% is invertible and semiample, and defines $r$; see
% \cite[Theorem 5.7]{Looijenga2003Ball}.

The interior strata have the following images:
\[
\begin{array}{c|c|c}
 \text{arrangement type}&\text{Miranda stratum}&\text{dimension}\\
 \hline
 (6)&I_9&0\\
 (9)&I_4^*&0\\
 (6,9)&I'_8&1\\
 (9,9)&I''_8&1\\
 (6,9,9)&I_7&2.
\end{array}
\]
The two cuspidal contributions give the $I_6$- and
$I_0^*$-strata, of dimensions $3$ and $1$.
See \cite[Remark 9.5 and Example 10.3]{HeckmanLooijenga2002}.

Every nonempty intersection in the ball has dimension at
least $5$.
Moreover, the complement of $X_o^\circ$ in $\mathcal M^M$
has codimension $5$.
Together with the polarization comparison
$\eta|_{X_o^\circ}\sim_{\Q}\lambda_{\rm aut}|_{X_o^\circ}$,
these facts give
\[
 \widehat X_o^{\mathcal K}\cong\mathcal M^M
\]
by \cite[Theorem 7.1 and Corollary 7.2]{Looijenga2003Ball}.

\section{Automorphic forms}
\label{sec:automorphic-forms}
In this section, we compute the Borcherds products and resulting Borcherds relation.
For the basic notion of automorphic forms, see \cite{diamond2005first}.

\subsection{Ramification and canonical divisors}
\label{subsec:ramification-canonical}
For the Eisenstein lattices considered below, duals are
taken with respect to the underlying integral bilinear form
\[
B(x,y):=\frac13
\operatorname{Tr}_{E/\mathbb Q}\psi(x,y).
\]
Equivalently, we use
\[
L^\vee
=
\{x\in L\otimes_{\mathscr O}E:
      \psi(x,L)\subset\theta\mathscr O\}.
\]
With this convention, $A$ is unimodular;
see \cite[Appendix A]{HeckmanLooijenga2002}.
Let $C(L):=L^{\vee}/L$ be the discriminant group and $q_L$ be the discriminant quadratic form induced by $B$.
Let $q:\B_o\to X_o$ be the quotient map and $e_d$ the ramification index of $q$ along the analytic divisor $\mathcal H_o(d)$.

\begin{lem}\label{lem:ramification-indices-eight-ball}
At a general point of the four mirror-trace divisors $\Delta(d)$ for $d\in\{6,9,15,18\}$, the ramification indices $e_d$ of
$q$ are
\[
        e_6=6,\qquad e_9=3,\qquad e_{15}=1,\qquad e_{18}=3,
\]
which implies that the branch divisors of $q$ is
\begin{equation}\label{eq:branch-boundary-eight-ball}
        B_{\rm br}:=
\frac56\Delta(6)+\frac23\Delta(9)+\frac23\Delta(18).
\end{equation}
\end{lem}
\begin{proof}
Since the scalar subgroup $\mathscr{O}^\times\subset \Gamma_o=\U(A_o)$
acts trivially on the projective ball $\mathbb B_o$, we put
$\overline{\Gamma}_o:=\Gamma_o/\mathscr{O}^\times$.
Let $H\in\mathcal H_o(d)$ be a mirror trace.  The ramification index of
the quotient map $q:\mathbb B_o\to X_o=\Gamma_o\backslash\mathbb B_o$
along $H$ is the order of the generic inertia group
$
I_H:=
\{\gamma\in\overline{\Gamma}_o:
\gamma|_H=\id_H\}.
$
% Equivalently, if $x\in H$ is a general point, then
% $I_H=(\overline{\Gamma}_o)_x$.
We now describe $I_H$ in lattice-theoretic terms. 
Let
$N_H=A\cap
\bigl(Ez_o+Er\bigr)$
introduced above and $M_H:=N_H^\perp\subset A$. 
Thus $N_H$ is the primitive rank-two positive-definite Eisenstein
lattice normal to $H$.  Since $z_o\in N_H$, we have $M_H\subset z_o^\perp=A_o$.
Moreover, putting one-dimensional subspace $\ell_H:=(N_H\otimes_{\mathscr{O}}\mathbb C)
       \cap (z_o^\perp\otimes_{\mathscr{O}}\mathbb C)$, we have the decomposition $V_{A_o}
=
(M_H\otimes_{\mathscr{O}}\mathbb C)\oplus\ell_H$.

The hyperball $H$ is the ball associated with
$M_H\otimes_{\mathscr{O}}\mathbb C$.
By \cite[Proposition A.8]{HeckmanLooijenga2002}, every element of $\U(A_o)$ admits a
unique extension to an element of $\U(A)$ that preserves $A_o$.  Thus let
$\gamma_0\in I_H$ and choose a representative in $\gamma\in\U(A_o)$, and let
$\widetilde\gamma\in \U(A)$ be its unique extension of $\gamma$.  Since $\gamma_0$ acts trivially on $H$, the restriction of $\widetilde\gamma$ to $M_H$ is a
scalar:
\[\widetilde\gamma|_{M_H}
=
\beta\id_{M_H}
\]
for some $\beta\in\mathscr{O}^\times$.
 Put $g:=\widetilde\gamma\vert_{N_H}\in \U(N_H)$.
Since $\widetilde\gamma$ preserves
$A_o=z_o^\perp$, it preserves its orthogonal line
$\mathbb Cz_o$, and hence $g(\mathbb Cz_o)=\mathbb Cz_o
$.
Conversely, suppose that $
g\in \U(N_H)$ with $g(\mathbb Cz_o)=\mathbb Cz_o$  and $\beta\in\mathscr{O}^\times$.
% The isometry $g\oplus\beta\id_{M_H}$ of $N_H\oplus M_H$ does not automatically extend to an isometry of
% $A$, because in general $N_H\oplus M_H\subsetneq A$.
% We therefore have to impose the gluing condition.
 Since $A$ is unimodular and
$N_H\subset A$ is primitive with orthogonal complement $M_H$, there is an $\mathscr O$-linear isomorphism $
\iota_H:C(N_H)\xrightarrow{\sim}C(M_H)$ with $q_{M_H}\circ\iota_H=-q_{N_H}$ 
such that the image of $A/(N_H\oplus M_H)$ is $\Gamma_{\iota_H}
:=
\{(x,\iota_H(x)):x\in C(N_H)\}
\subset C(N_H)\oplus C(M_H)$; see also \cite[Lemmas 1.4.4 and 1.4.5]{Dolgachev1983Quadratic}.
Consequently, $
g\oplus\beta\id_{M_H}
$
extends to an element of $\U(A)$ if and only if it preserves
$\Gamma_{\iota_H}$.  Letting $
\bar g:C(N_H)\to C(N_H)$
denote the action induced by $g$, then this condition is 
\begin{align}\label{eq:gluing-inertia}
\iota_H\circ\bar g=\beta\iota_H.
\end{align}
 Thus, the elements of $I_H$ are represented by the pairs $(g,\beta)$ satisfying $g\in \U(N_H)$, $g(\mathbb Cz_o)=\mathbb Cz_o$, and 
$\bar g=\beta\id_{C(N_H)}$,
modulo the diagonal scalar pairs
$(\epsilon\id_{N_H},\epsilon)$ with $\epsilon\in\mathscr{O}^\times$,
which act trivially on $\mathbb B_o$.

For the chosen lift $\widetilde\gamma$, let $\alpha$
and $\beta$ be its eigenvalues on $\ell_H$ and
$M_H\otimes_{\mathscr O}\mathbb C$, respectively.
Both belong to $\mathscr O^\times$.
For every $[v]\in H$, the orthogonal decomposition $V_{A_o}
=(M_H\otimes_{\mathscr O}\mathbb C)\oplus\ell_H$
gives $N_{H/\mathbb B_o,[v]}
\cong\operatorname{Hom}_{\mathbb C}(\mathbb Cv,\ell_H)$.
The induced action on this normal space is multiplication
by $\alpha\beta^{-1}$.
This ratio is unchanged when the lift $\gamma$ of $\gamma_0$ is multiplied
by $\mathscr{O}^\times$, and hence defines a character
\[
\chi_H\colon I_H\longrightarrow\mu_6,
\qquad
\chi_H(\gamma_0)=\alpha\beta^{-1}.
\]
The character is injective, so
$e_H=|I_H|=|\chi_H(I_H)|$.
For a generator of $I_H$, this is the multiplicative
order of $\alpha\beta^{-1}$ in $\mathbb C^\times$.

% Finally, suppose that $g$ acts on the one-dimensional normal space
% $\ell_H$ by $g\vert_{\ell_H}
% =
% \alpha\id_{\ell_H}$. \Masa{What is $\alpha$?}
% At a point $[v]\in H$ \Masa{Don't you need ``general'' condition for $[v]$?}, the normal space of $H$ in $\mathbb B_o$ is
% naturally identified with $N_{H/\mathbb B_o,[v]}
% \cong
% \operatorname{Hom}(\mathbb Cv,\ell_H)$.

% Since $\widetilde\gamma$ acts by $\beta$ on $\mathbb Cv\subset
% M_H\otimes\mathbb C$ and by $\alpha$ on $\ell_H$, its action on the
% normal line \Masa{What is it? I want you to use symbols.} is multiplication by $\chi_H(\gamma)=\alpha\beta^{-1}$.
% Hence the generic inertia group acts faithfully on the normal line,
% and the ramification index is $e_H=|\chi_H(I_H)|$.
% Equivalently, for a generator of $I_H$, $e_H$ is the order of the
% relative scalar $\alpha/\beta$. \Masa{It might be better to put $\alpha$ like $\alpha_d$ here for example.}
% \Masa{Where do you consider the order of $\alpha/\beta$?}

For each $d\in\{6,9,15,18\}$, choose an
$\mathscr O$-basis $(e_1,e_2)$ of $N_H$ as in
\cite[Lemma A.9]{HeckmanLooijenga2002}.
We write $G_d$ for the Gram matrix of the Hermitian form
in this basis, using column coordinates and the convention
$\psi(x,y)=^{t}\!\overline{x}G_dy$.
The corresponding expressions for $z_o$ and $G_d$
are listed below.

\medskip
\noindent
\textbf{The case $d=6$.}
Here, we put as \cite[Lemma A.9]{HeckmanLooijenga2002}
\[
    z_o=e_1+e_2,
    \qquad
    G_6=
    \begin{pmatrix}
        3&\sqrt{-3}\\
        -\sqrt{-3}&3
    \end{pmatrix}.
\]
Consider
\[
    g_6=
    \begin{pmatrix}
        0&-1\\
        -\bar\omega&-\omega
    \end{pmatrix}.
\]
A direct calculation gives
$g_6^*G_6g_6=G_6$ and $g_6(z_o)=-z_o$.
The vector
$n_6:=e_1-\bar\omega e_2$ spans $\ell_H$, and
$g_6(n_6)=\bar\omega n_6$.
Thus, if the action on $M_H$ is the identity,
the induced action on
$\operatorname{Hom}_{\mathbb C}(\mathbb Cv,\ell_H)$
is multiplication by $\bar\omega$.
We also obtain $N_H^\vee=\theta G_6^{-1}\mathscr O^2$. 
Moreover,
\[
(g_6-I_2)\theta G_6^{-1}
=
\begin{pmatrix}
\bar\omega&-\omega\\
0&\bar\omega
\end{pmatrix}.
\]
Hence $(g_6-I_2)N_H^\vee\subset N_H$, so $g_6$
acts trivially on $C(N_H)$.
By \eqref{eq:gluing-inertia}, the isometry
$g_6\oplus\id_{M_H}$ extends to $A$.
By letting $\gamma_6\in I_H$ be the image of $g_6\oplus\id_{M_H}|_{N_H}$, then $\chi_H(\gamma_6)$ has order six.
Since $\chi_H(I_H)\subset\mu_6$, we obtain $e_6=6$.

% A direct computation gives $g_6^*G_6g_6=G_6$ and $g_6(z_o)=-z_o$.
% The normal line is generated by $n_6=e_1-\bar\omega e_2$, and $g_6(n_6)=\bar\omega n_6$. \Masa{I thought that $N_{H/\mathbb B_o,[v]}$ is a quotient. Why it can be represented by an element of $A\otimes_{\mathscr{O}}\C$? Is it just the image?}
% The element $g_6$ is obtained from ambient triflections and
% therefore satisfies the gluing condition. \Masa{What is the gluing condition? I think we should put eqref... I suspect that it would be $\Gamma_{\iota_H}$ and $\tilde{\gamma}$ should satisfy...} Since $\bar\omega$
% has order six, one obtains $e_6=6$.

\medskip
\noindent
\textbf{The case $d=9$.}
Here, we put
\[
    z_o=e_1+e_2,
    \qquad
    G_9=
    \begin{pmatrix}
        3&0\\
        0&3
    \end{pmatrix},
    \qquad
    C(N_H)\cong(\mathscr{O}/\sqrt{-3})^2.
\]
Set $g_9:=-\omega I_2$.
Since $-\omega\equiv1\pmod{\theta}$, the pair $(g_9,1)$
satisfies \eqref{eq:gluing-inertia}.
Thus $g_9\oplus\id_{M_H}$ extends to an isometry of $A$
and induces an element $\gamma_9\in I_H$.
Since $\ell_H=\mathbb C(e_1-e_2)$, we have
$\chi_H(\gamma_9)=-\omega$, which has order three.

Conversely, the norm-$3$ vectors of $N_H$ are the unit multiples
of $e_1$ and $e_2$.
Hence every isometry preserving $\mathbb C(e_1+e_2)$
is of the form $a\id_{N_H}$ or $as$, where
$a\in\mathscr O^\times$ and $s$ is the linear automorphism interchanging $e_1$ and $e_2$.
The action of $as$ on $(\mathscr O/\theta)^2$ is not scalar,
so it cannot satisfy \eqref{eq:gluing-inertia}.
For $g=a\id_{N_H}$, that condition becomes
$a\beta^{-1}\equiv1\pmod{\theta}$.
The units satisfying this congruence belong to $\mu_3$.
Therefore, $\chi_H(I_H)=\mu_3$ and $e_9=3$.
% Set $g_9:=-\omega I_2$.
% Since $-\omega\equiv1\pmod{\sqrt{-3}}$,
% the induced action on $C(N_H)$ is trivial, and hence $g_9$
% satisfies the gluing condition. \Masa{What is the gluing condition again? $g_9\oplus \mathrm{id}$?} It induces an element of $I_H$ whose image of $\chi_H$ has order three. Thus $e_9=3$.
% \Masa{Is it necessary to check the image of $\chi_H$?}

% The involution interchanging $e_1$ and $e_2$ is an isometry of the
% abstract rank-two lattice, but it acts nontrivially on
% $(\mathcal O/\sqrt{-3})^2$ and does not belong to the generic
% ambient pointwise stabilizer.

\medskip
\noindent
\textbf{The case $d=18$.}
Here, we put
\[
    z_o=e_1,
    \qquad
    G_{18}=
    \begin{pmatrix}
        6&0\\
        0&3
    \end{pmatrix},
\]
and $C(N_H)
    \cong
    \mathscr{O}/(2\sqrt{-3})\oplus\mathscr{O}/\sqrt{-3}$.
    An isometry of $N_H$ preserving $\mathbb Ce_1$ has the form
$g=\operatorname{diag}(a,\alpha)$, with
$a,\alpha\in\mathscr O^\times$.
If $(g,\beta)$ satisfies \eqref{eq:gluing-inertia},
its action on the second discriminant-module summand gives
$\alpha\beta^{-1}\equiv1\pmod{\theta}$.
Thus $\chi_H(I_H)\subseteq\mu_3$.

Conversely, $g_{18}=\operatorname{diag}(1,-\omega)$
acts trivially on $C(N_H)$, so $(g_{18},1)$ satisfies
\eqref{eq:gluing-inertia}.
Since $\ell_H=\mathbb Ce_2$, the resulting element of $I_H$
has normal character $-\omega$.
Hence, $\chi_H(I_H)=\mu_3$ and $e_{18}=3$.
% It fixes $z_o$, acts on the normal line \Masa{$\ell_H$?} by $-\omega$, and satisfies
% the gluing condition because $-\omega\equiv1\pmod{\sqrt{-3}}$.  Only the
% cube-root subgroup of $\mu_6$ acts trivially on the relevant
% discriminant-module summand. \Masa{Isn't it necessary to check $\chi_H$?} Hence $e_{18}=3$.

\medskip
\noindent
\textbf{The case $d=15$.}
In this case, we put

\[
z_o=e_1,\qquad
G_{15}=
\begin{pmatrix}
6&\sqrt{-3}\\
-\sqrt{-3}&3
\end{pmatrix}.
\]

Let $g\in \U(N_H)$ preserve the line $\mathbb C e_1$. With respect to the basis $(e_1,e_2)$, write

$$
g=
\begin{pmatrix}
a&b\\
0&d
\end{pmatrix},
\qquad a,d\in\mathscr{O}^\times,\quad b\in\mathscr{O}.
$$

The relation $g^*G_{15}g=G_{15}$ gives $6b=\sqrt{-3}(a-d)$, and hence $2\sqrt{-3}\mid(a-d)$. Since no two distinct units of $\mathscr{O}$ are congruent modulo $2\sqrt{-3}$, we obtain $a=d$ and $b=0$. Thus $g=a\mathrm{id}_{N_H}$.

By \eqref{eq:gluing-inertia}, the isometry
$a\id_{N_H}\oplus\beta\id_{M_H}$ extends to $A$
if and only if $a\beta^{-1}$ acts trivially on $C(N_H)$.
Since $C(N_H)\cong\mathscr O/(5)$, this is equivalent to
$a\beta^{-1}\equiv1\pmod5$.
The only unit satisfying this congruence is $1$.
Thus $a=\beta$, so the extension is scalar and induces
the identity on $\mathbb B_o$.
Consequently, $I_H$ is trivial and $e_{15}=1$.
% At a general point of $H$, an element of the generic inertia group acts on $M_H=N_H^\perp$ by a scalar $\beta\in\mathscr{O}^\times$. By the gluing criterion above \Masa{This part is also}, the pair $a,\mathrm{id}*{N_H}\oplus\beta,\mathrm{id}*{M_H}$ extends to an isometry of $A$ if and only if the relative scalar $a\beta^{-1}$ acts trivially on $C(N_H)$. Since $C(N_H)\cong\mathscr{O}/(5)$, this is equivalent to $a\beta^{-1}\equiv1\pmod{5}$. Among the elements of $\mathscr{O}^\times$, this congruence forces $a=\beta$. Hence the resulting isometry of $A$ is the scalar $a\mathrm{id}_A$. Therefore the effective generic inertia group along $H$ is trivial, and $e_{15}=1$. \Masa{I think such an argument would be necessary for $d=9,15$.}
\end{proof}

Let $X_o^{\mathrm{tor}}$ be the unique toroidal compactification of $X_o$ and
$T$ be its reduced toroidal boundary.
Note that the centre of the unipotent part of the stabilizer of a cusp is one-dimensional, so the fan decomposition is unique, which implies that there is only one choice of toroidal compactifications in this setting. For more details of toroidal compactifications, see \cite{ash2010smooth}, especially \cite[III, Theorem 5.2]{ash2010smooth}.
Let $\Delta(d)^{\mathrm{tor}}$ (resp. $\Delta(d)^{\mathrm{bb}}$) be the closure of $\Delta(d)$ in $X_o^{\mathrm{tor}}$ (resp. $X_o^{\mathrm{bb}}$).  
% Since $\B_o$ has dimension eight, the
% canonical weight is $9$. 
% Hence the automorphic $\Q$-line bundle is 
% \[
%         K_{X_o^{\mathrm{tor}}}+T+
%         \sum_i\frac{e_d-1}{e_d}\Delta(d)^{\mathrm{tor}},
% \]
% and its descended Baily-Borel counterpart is
% \[
%         K_{X_o^{\mathrm{bb}}}+
%         \sum_d\frac{e_d-1}{e_d}\Delta(d)^{\mathrm{bb}};
% \]
% see \cite[Section 2]{MaedaOdaka2022FanoShimura}.  
Applying this to our setting gives the following formula; compare it with the nine-dimensional
ancestral Deligne-Mostow calculation in \cite[Remark 4.3]{HKM24}, where the
single branch component has ramification index three.

\begin{prop}\label{prop:toroidal-canonical-eight-ball}
One has the log-canonical relation
\begin{equation*}
        K_{X_o^{\mathrm{tor}}}+T
        +\frac56\Delta(6)^{\mathrm{tor}}
        +\frac23\Delta(9)^{\mathrm{tor}}
        +\frac23\Delta(18)^{\mathrm{tor}}
        \sim_\Q 9\mathcal L.
\end{equation*}
\end{prop}

\begin{proof}
This is a direct consequence of the Hirzebruch proportionality principle \cite[Proposition 3.4 (a)]{mumford1977hirzebruch}.
For completeness, we include the proof below.
Choose a neat normal subgroup
$\Gamma'\subset\Gamma_o\cap\mathrm{SU}(V_{A_o})$
of finite index, and let $p:Y^{\mathrm{tor}}\to X_o^{\mathrm{tor}}$
be the induced finite morphism.
Write $T_Y$ for the reduced boundary and
$\mathcal L_Y:=p^*\mathcal L$.

In affine ball coordinates $z\in\mathbb C^8$, let $j(\gamma,z)$ be the automorphic factor for $\gamma\in\Gamma'$.
The Jacobian is
\[
\det\left(\frac{\partial(\gamma z)}{\partial z}\right)
=
\frac{\det\gamma}{j(\gamma,z)^9}.
\]
Thus, for $\gamma\in\Gamma'$, a weight-$9$ automorphic
form $F$ satisfies
\[
\gamma^*\bigl(F(z)\,dz_1\wedge\cdots\wedge dz_8\bigr)
=
F(z)\,dz_1\wedge\cdots\wedge dz_8.
\]
This gives the canonical weight $9$;
compare \cite[Section 4.1]{Looijenga2003Ball}. 
Note that this is equivalent to the compact dual of $\B_o$ being $\mathbb{P}^8$.

At a cusp, direct computation shows that a weight-$9$ automorphic form 
determines a logarithmic top form.
If it is a cusp form, its vanishing
 cancels this logarithmic pole.
Consequently, $K_{Y^{\mathrm{tor}}}+T_Y
\sim 9\mathcal L_Y$.

Logarithmic Riemann-Hurwitz formula gives
\[
K_{Y^{\mathrm{tor}}}+T_Y
=
p^*\left(
K_{X_o^{\mathrm{tor}}}+T+
\sum_{d\in\{6,9,15,18\}}
\left(1-\frac1{e_d}\right)\Delta(d)^{\mathrm{tor}}
\right).
\]
Descending the preceding identity and using
Lemma \ref{lem:ramification-indices-eight-ball}
proves the assertion.
\end{proof}

On the Baily-Borel compactification, we have
\[
        K_{X_o^{\mathrm{bb}}}
        +\frac56\Delta(6)^{\mathrm{bb}}
        +\frac23\Delta(9)^{\mathrm{bb}}
        +\frac23\Delta(18)^{\mathrm{bb}}
        \sim_\Q 9\mathcal L.
\]
Since $9\mathcal L$ is $\Q$-Cartier, the Baily-Borel compactification
$X_o^{\mathrm{bb}}$ is $\Q$-Gorenstein if and only if the branch divisor
\[
        B_{\rm br}^{\rm bb}
        =\frac56\Delta(6)^{\mathrm{bb}}
        +\frac23\Delta(9)^{\mathrm{bb}}
        +\frac23\Delta(18)^{\mathrm{bb}}
\]
is $\Q$-Cartier.  The local computation in 
Proposition \ref{prop:qcartier-at-cusps-corrected} implies that $K_{X_o^{\rm bb}}$ is not $\Q$-Cartier.

\subsection{Allcock's Borcherds product}
In the rest of this section, we use the Borcherds product to produce appropriate relations of hyperballs in the Picard groups.
For details of the Borcherds product, see \cite{Borcherds1998Grassmannians,Bruinier2002BorcherdsProducts}.
Allcock constructed, by means of Borcherds products, an automorphic form on the
corresponding $9$-dimensional complex ball quotient whose divisor is
exactly $\mathbb{B}_A(\mathcal{H}_A)$ \cite[Theorem 7.1]{Allcock2000Leech}.  We denote it by $\Psi_A$.
It has weight $44$ and satisfies
        $\operatorname{div}_{\B_A}(\Psi_A)=\B_A(\mathcal H_A)$.
        In Allcock's construction, the orthogonal Borcherds
product has weight $132$.
Its restriction to $\mathbb B_A$ vanishes to order three
along each root hyperball.
Taking a cube root gives the weight-$44$ automorphic form $\Psi_A$
with simple zeros along these hyperballs;
see the proof of \cite[Theorem 7.1]{Allcock2000Leech}.
This convention is compatible with Heckman-Looijenga's comparison of the discriminant divisor with Allcock's weight convention; see \cite[Remark 7.10]{HeckmanLooijenga2002} and \cite[Remarks 4.3-4.4 and Section 8.6]{HKM24}.

We now restrict this to the hyperball $\B_o\subset \B_A$ attached to the norm-$6$ vector $z_o$.  By the extension property for the stabilizer of $\C z_o$ \cite[Proposition A.8]{HeckmanLooijenga2002}, every element of $\U(A_o)$ extends uniquely to an element of $\U(A)$ preserving $\C z_o$.  Therefore the restriction of $\Psi_A$ to $\B_o$ is again automorphic, now for the group $\U(A_o)$.
Let $\Psi_{A,o}\defeq \Psi_A|_{\B_o}$ be the restriction of $\Psi_A$ to the eight-dimensional ball $\B_o$.
\begin{thm}\label{thm:allcock-restriction}
The restriction $\Psi_{A,o}$ is a non-zero automorphic form of weight $44$ for $\U(A_o)$.  Its zero divisor on $\B_o$ is supported on the full mirror-trace arrangement with the following multiplicity:
\begin{equation}\label{eq:divisor-allcock-restriction}
        \operatorname{div}_{\B_o}(\Psi_{A,o})
        =4\B_o(\mathcal H_o(6))
        +2\B_o(\mathcal H_o(9))
        +\B_o(\mathcal H_o(15))
        +\B_o(\mathcal H_o(18)),
\end{equation}
\end{thm}

\begin{proof}
The restriction is non-zero because the hyperball $\B_o$ is not contained in any hyperball with respect to a norm $3$ vector.  Since $\U(A_o)$ is the restriction of the stabilizer of $\C z_o$ in $\U(A)$, the automorphy of $\Psi_{A,o}$ follows from the automorphy of $\Psi_A$.

It remains to determine the divisor.  A hyperball on $\B_A$ defined by a norm-$3$ vector $r\in A$ meets $\B_o$ in a mirror trace precisely when the rank-two lattice generated by $z_o$ and $r$ is positive definite.  Heckman-Looijenga's rank-two classification gives the four possible types $\delta_6,\delta_9,\delta_{15},\delta_{18}$, whose discriminants are $6,9,15,18$ respectively \cite[Lemma A.9 and Proposition A.12]{HeckmanLooijenga2002}.  

Fix a general point of a mirror trace $H$.
Let $H_1,\ldots,H_m$ be the distinct hyperballs in $\mathbb{B}_A$
whose intersections with $\mathbb B_o$ equal $H$ around the generic point.
In a local trivialization, we have $\Psi_A=u\prod_{i=1}^m\ell_i$,
where $u$ is a unit and $\ell_i$ is a local linear
equation for $H_i$ by \cite[Theorem 7.1]{Allcock2000Leech}. 
In an affine chart of $\mathbb P(V_A)$, the function $\ell_i$
is affine-linear and $\mathbb P(V_{A_o})$ is a linear subspace.
Because $\mathbb B_o$ is not contained in $H_i$,
the restriction $\ell_i|_{\mathbb B_o}$ is a nonzero
affine-linear equation defining $H$.
It therefore vanishes to order one along $H$.
Therefore
\[
\operatorname{ord}_H(\Psi_A|_{\mathbb B_o})=m.
\]
This number is the number of
$\mathscr O^\times$-orbits of norm-$3$ vectors in $N_H$ by definition.

The above discussion shows that for a fixed mirror trace $H\in\mathcal H_o(d)$, the vanishing order of the restriction is the number of distinct hyperballs whose restriction to $\B_o$ is $H$.  Equivalently, it is the number of orbits of $\{r\in N_H:\psi(r,r)=3\}$
under scalar multiplication by
$\mathscr O^\times$.  The finite count from the four rank-two Gram matrices is
\[
\begin{array}{c|c|c}
 d & \text{rank-two type of } N_H  & \#\{\text{norm-}3\text{ hyperballs restricting to }H\} \\
\hline
 6  & \delta_6   & 4 \\
 9  & \delta_9   & 2 \\
 15 & \delta_{15} & 1 \\
 18 & \delta_{18} & 1 .
\end{array}
\]
Here the last column is the number of unit-orbits of norm-$3$ vectors in the corresponding positive definite rank-two Eisenstein lattice; it is computed directly from the Gram matrices in \cite[Lemma A.9 and Proposition A.12]{HeckmanLooijenga2002}.  Since $\Psi_A$ vanishes with order one along each hyperball, this proves \eqref{eq:divisor-allcock-restriction}. 
\end{proof}

The multiplicities computed in
\eqref{eq:divisor-allcock-restriction} are $(4,2,1,1)$.  Dividing by
the ramification indices $(6,3,1,3)$ of
Lemma \ref{lem:ramification-indices-eight-ball} gives the 
multiplicities $(2/3,2/3,1,1/3)$.  
This shows the following corollary.

\begin{cor}
\label{cor:allcock-relation-corrected}
On the Baily-Borel compactification $X_o^{\rm bb}$, we have
\begin{equation}
 \frac23\Delta(6)+\frac23\Delta(9)+\Delta(15)
 +\frac13\Delta(18)
 \sim_{\Q}44\lambda_{\rm aut}.
\label{eq:allcock-linear-equivalence-BB}
\end{equation}
\end{cor}

\subsection{Another Borcherds product}

We next construct a weakly holomorphic vector-valued modular form as an appropriate input of the Borcherds product. 
Here we recall that weakly holomorphic means holomorphic
on the upper half-plane and meromorphic at the cusp,
so its Fourier expansion has only finitely many
negative powers; see \cite[Sections 2 and 4]{Borcherds1998Grassmannians}.

For any nondegenerate even integral lattice $L_{\Z}$, let $L_{\Z}^\vee$
be its dual with respect to the specified integral bilinear form
$(\ ,\ )_{L_{\Z}}$, and write $C(L_{\Z}):=L_{\Z}^\vee/L_{\Z}$ for its discriminant group.
We put $Q_{L_{\Z}}(x):=(x,x)_{L_{\Z}}/2$ and denote the induced discriminant
quadratic form by $q_{L_{\Z}}:C(L_{\Z})\to\mathbb Q/\mathbb Z$,
where $q_{L_{\Z}}(x+L_{\Z}):=Q_{L_{\Z}}(x)\bmod\mathbb Z$. 
We will choose a representative of $q_{L_{\Z}}(x+L_{\Z})\in \mathbb Q/\mathbb Z$ from
$[0,1)$.

Let $A_{o,\mathbb Z}$ be the free abelian group
underlying $A_o$, equipped with the symmetric
bilinear form
\[
(x,y):=
-\frac13\operatorname{Tr}_{E/\mathbb Q}
\psi(x,y).
\]
It makes $A_{o,\mathbb Z}$ an even lattice of
signature $(2,16)$.
We abbreviate $Q_{A_{o,\mathbb Z}}$ to $Q$, so that
$Q(x)=(x,x)/2=-\psi(x,x)/3$. Equip the underlying abelian group of $A$ with
the same bilinear form and denote it by $A_{\mathbb Z}$.
By \cite[Appendix A]{HeckmanLooijenga2002},
it is even unimodular of signature $(2,18)$.
The orthogonal complement of $A_{o,\mathbb Z}$ in
$A_{\mathbb Z}$ is isometric to $-A_2(2)$. Here $A_2$ denotes the positive-definite root lattice
with Gram matrix
$\left(\begin{smallmatrix}2&-1\\-1&2\end{smallmatrix}\right)$,
and $A_2(2)$ is obtained by multiplying its bilinear
form by $2$.
Thus the positive-definite lattice $K:=A_2(2)$
has Gram matrix
\[
\begin{pmatrix}
4&-2\\
-2&4
\end{pmatrix}.
\] 
Fix an isometry from $-K$ to $\mathscr O z_o$, and give $K$
the induced $\mathscr O$-module structure.
Since $A_{o,\mathbb Z}$ and $-K$ are primitive orthogonal
complements in the even unimodular lattice $A_{\mathbb Z}$,
together with $q_{-K}=-q_K$, this gives
an isometry
\begin{equation}
\iota_o:
\bigl(C(A_{o,\mathbb Z}),q_{A_{o,\mathbb Z}}\bigr)
\xrightarrow{\sim}
\bigl(C(K),q_K\bigr).
\label{eq:discriminant-isometry-Ao-K}
\end{equation}
which we use throughout this subsection.

Now we explain the restriction of automorphic forms to ball quotients from orthogonal modular varieties, which was also implicitly used in the construction by Allcock above.
Let $\mathcal D_o$ be the connected component of
\[
\{[Z]\in\mathbb P(A_{o,\mathbb Z}\otimes\mathbb C):
  (Z,Z)=0,\ (Z,\overline Z)>0\}
\]
containing the image of $\mathbb B_o$.
The chosen embedding $E\hookrightarrow\mathbb C$
identifies $V_{A_o}$ with the corresponding eigenspace
for multiplication by $\omega$ in
$A_{o,\mathbb Z}\otimes\mathbb C$.
This induces a $\mathrm U(A_o)$-equivariant holomorphic
embedding $\iota:\mathbb B_o\hookrightarrow\mathcal D_o$.
The tautological line bundle on $\mathcal D_o$ restricts to the
tautological line bundle on $\mathbb B_o$.
This implies that the restriction of automorphic forms preserves their weight.
For $x\in A_{o,\mathbb Z}^{\vee}$ with $Q(x)<0$, put a divisor
$D_x:=\{[Z]\in\mathcal D_o:(Z,x)=0\}$.
Its local equation restricts, up to a nonzero scalar,
to $\psi(x,z)$ on $V_{A_o}$.
Consequently, $\iota^*D_x=H_{A_o}(x)$
as reduced divisors on $\mathbb B_o$; see \cite[Lemmas 6.1 and 6.2]{Hofmann2014}.

Let $\{\mathbf e_\gamma:\gamma\in C(K)\}$ be the
standard basis of the group algebra $\mathbb C[C(K)]
=\bigoplus_{\gamma\in C(K)}\mathbb C\mathbf e_\gamma$.
For $r\in\{0,1/6,1/2,2/3\}$, put \[\mathbf e_r:=
\sum_{\substack{\gamma\in C(K)\\q_K(\gamma)=r}}
\mathbf e_\gamma.\]

Let $\tau$ be an element of the upper-half plane and put
$q=e^{2\pi i\tau}$. 
Define
\[
 \Theta_K(\tau):=\sum_{\gamma\in C(K)}\left(\sum_{x\in K+\gamma}q^{Q_K(x)}\right)\mathbf e_\gamma.
\]
Via the identification of discriminant forms above,
let $\rho\colon\operatorname{Mp}_2(\mathbb Z)
\to\operatorname{GL}(\mathbb C[C(K)])$
be the Weil representation associated with
$A_{o,\mathbb Z}$.
Here, $\operatorname{Mp}_2(\mathbb Z)$ is the
metaplectic double cover of $\operatorname{SL}_2(\mathbb Z)$.
We use the convention of
\cite[Section 4]{Borcherds1998Grassmannians}; in particular,
the standard lift $T$ of $\tau\mapsto\tau+1$ acts by
$\rho(T)\mathbf e_\gamma
 =e^{2\pi i q_K(\gamma)}\mathbf e_\gamma$.
The theta series $\Theta_K$ is a holomorphic
vector-valued modular form of weight $1$ for $\rho$. Putting $\sigma_1(n):=\sum_{\substack{d\mid n\\d>0}}d$, let $E_2(\tau)=1-24\sum_{n\geq1}\sigma_1(n)q^n$ be the Eisenstein series of weight 2 and $\Delta_{\rm mod}(\tau)=q\prod_{n\geq1}(1-q^n)^{24}$ be the modular discriminant. 
For a weakly holomorphic vector-valued modular form
$F$ of weight $k$ for $\rho$, define its Serre derivative
componentwise by
\[
\mathcal D_kF
:=
\frac{1}{2\pi i}\frac{dF}{d\tau}
-\frac{k}{12}E_2F
=
q\frac{dF}{dq}-\frac{k}{12}E_2F.
\]
It is a weakly holomorphic modular form of weight
$k+2$ for the same representation.
Define
\begin{equation}
g_D:=
144\,\frac{\mathcal D_3\mathcal D_1\Theta_K}
{\Delta_{\rm mod}}.
\label{eq:gA-gD-definitions}
\end{equation}
To use $g_D$ as an input of the Borcherds product below, first we analyze its required property.

\begin{lem}
\label{lem:gD-fourier-expansion}
$g_D$ is a weakly holomorphic modular form of weight $-7$.  Its principal part, that is the terms with negative powers of $q$, is
\begin{equation}
3q^{-1}\mathbf e_0-q^{-5/6}\mathbf e_{1/6}+30q^{-1/2}\mathbf e_{1/2}+105q^{-1/3}\mathbf e_{2/3},
\label{eq:gD-principal-part-corrected}
\end{equation}
and the constant term of its zero component is $216$.
\end{lem}

\begin{proof}
The construction of $g_D$ is designed to produce a modular
form of weight $1-16/2=-7$, as required for the Borcherds
lift associated with the lattice $A_{o,\mathbb Z}$ of
signature $(2,16)$.
The theta series $\Theta_K$ has weight $1$.
The Serre derivative $\mathcal D_k$, applied componentwise,
preserves the Weil representation and raises the weight
from $k$ to $k+2$; this follows from the usual transformation
calculation, as in \cite[Section 5.1]{Zagier2008}.
Thus $\mathcal D_3\mathcal D_1\Theta_K$ has weight $5$,
and division by the weight-$12$ automorphic form $\Delta_{\rm mod}$
gives weight $-7$.
Since $\Delta_{\rm mod}$ has no zeros on the upper
half-plane and has a simple zero at the cusp,
$g_D$ is weakly holomorphic.

It remains to compute the terms with negative powers
of $q$ and the constant term of the zero component.
The Serre derivatives do not decrease Fourier exponents,
whereas $\Delta_{\rm mod}^{-1}$ has a simple pole at $q=0$.
Consequently, these coefficients depend only on the terms
of $\Theta_K$ with exponents at most $1$.
They can therefore be determined by counting vectors
$x\in K^\vee$ with $Q_K(x)\leq1$ in each coset of $K^\vee/K$.

In coordinates relative to the basis defining the Gram
matrix of $K$, we have
\[
K^\vee
=
\left\{
\frac{1}6(a,b):
a,b\in\mathbb Z,\quad a+b\equiv0\pmod3
\right\},
\qquad
Q_K\left(\frac{1}6(a,b)\right)
=
\frac{a^2-ab+b^2}{18}.
\]
The classes in $C(K)=K^\vee/K$ are obtained by reducing
$(a,b)$ modulo $6\mathbb Z^2$, giving twelve classes.
The following table records representatives for these
classes, the minimum $r$ of $Q_K$ on each class, and the
number $c$ of vectors in that class attaining this minimum:
\[
\begin{array}{c|c|c|c}
r&
\text{representatives }(a,b)&
\text{number of classes}&c\\ \hline
0&(0,0)&1&1\\
1/6&\pm(1,-1),\ \pm(2,1),\ \pm(1,2)&6&1\\
1/2&(3,0),\ (0,3),\ (3,3)&3&2\\
2/3&\pm(2,-2)&2&3
\end{array}
\]
Here $(a,b)$ represents the class $(a,b)/6+K$.
The entries follow directly by enumerating the integer
pairs with $a+b\equiv0\pmod3$ and
$a^2-ab+b^2\leq18$.
For example, the three vectors attaining the minimum
in the class $(2,-2)/6+K$ are
$(2,-2)/6$, $(-4,-2)/6$, and $(2,4)/6$.

For $x\in K+\gamma$, the exponent $Q_K(x)$ is congruent
to $q_K(\gamma)$ modulo $\mathbb Z$.
Thus, in each nonzero component, any exponent beyond
the minimum is at least one larger.
For the zero component, a nonzero vector $(a,b)\in K$
satisfies
$Q_K(a,b)=2(a^2-ab+b^2)\geq2$.
It follows that
\[
(\Theta_K)_0=1+O(q^2),
\]
and
\[
\Theta_K
=
\mathbf e_0
+q^{1/6}\mathbf e_{1/6}
+2q^{1/2}\mathbf e_{1/2}
+3q^{2/3}\mathbf e_{2/3}
+O(q^{7/6}).
\]

Since $E_2=1+O(q)$, applying
$144\mathcal D_3\mathcal D_1$ to a term $cq^\alpha$
gives
\[
144\mathcal D_3\mathcal D_1(cq^\alpha)
=
144\left(\alpha-\frac1{12}\right)
   \left(\alpha-\frac14\right)cq^\alpha
+O(q^{\alpha+1}).
\]
For
$(\alpha,c)=(0,1),(1/6,1),(1/2,2),(2/3,3)$,
the resulting leading coefficients are respectively
$3,-1,30,105$.
Division by $\Delta_{\rm mod}=q+O(q^2)$ lowers each
leading exponent by one.
All remaining terms have nonnegative exponents,
which proves \eqref{eq:gD-principal-part-corrected}.

Finally, using $(\Theta_K)_0=1+O(q^2)$ and
$E_2=1-24q+O(q^2)$, we obtain
\[
\mathcal D_1(\Theta_K)_0
=
-\frac1{12}+2q+O(q^2),
\qquad
\mathcal D_3\mathcal D_1(\Theta_K)_0
=
\frac1{48}+q+O(q^2).
\]
Since $\Delta_{\rm mod}^{-1}=q^{-1}+24+O(q)$,
the zero component of $g_D$ is
\[
(g_D)_0
=
(3+144q+O(q^2))(q^{-1}+24+O(q))
=
3q^{-1}+216+O(q).
\]
This proves the assertion about its constant term.
\end{proof}

% \begin{proof}
% \Masa{It looks very technical. So, could you remind the argument in some papers? I'd like you to explain to me and I'll think about how to write it.}
% The inverse Gram matrix of $K$ is \[\frac1{12}\left(\begin{matrix}4&2\\2&4\end{matrix}\right).\]  Smith reduction \Masa{What?} and direct substitution give the following distribution of the twelve classes of $C(K)$:
% \[
% \begin{array}{c|cccc}
% q_K(\gamma)&0&1/6&1/2&2/3\\ \hline
% \#\{\gamma\}&1&6&3&2\\
% \text{coefficient of the first term}&1&1&2&3.
% \end{array}
% \]
% Consequently, $\Theta_K=\mathbf e_0+q^{1/6}\mathbf e_{1/6}+2q^{1/2}\mathbf e_{1/2}+3q^{2/3}\mathbf e_{2/3}+O(q^{7/6})$.
% The two Serre derivatives \Masa{What?} raise the weight from $1$ to $5$, and division by $\Delta_{\rm mod}$ lowers it to $-7$.  On a leading term $cq^\alpha$, the leading coefficient of $144\mathcal D_3\mathcal D_1(cq^\alpha)$ is $144(\alpha-1/12)(\alpha-3/12)c$.  Applying this for $\alpha=0,1/6,1/2,2/3$ and then dividing by $\Delta_{\rm mod}=q+O(q^2)$ gives \eqref{eq:gD-principal-part-corrected}.

% Since $K$ has minimum $4$, the zero component satisfies $(\Theta_K)_0=1+O(q^2)$. \Masa{Where $O(q^2)$ comes from? I'm not sure if there is by the definition of $\Theta_K$.}  Using $E_2=1-24q+O(q^2)$, one obtains $\mathcal D_1(\Theta_K)_0=-1/12+2q+O(q^2)$ and $\mathcal D_3\mathcal D_1(\Theta_K)_0=1/48+q+O(q^2)$.  Since $\Delta_{\rm mod}^{-1}=q^{-1}+24+O(q)$, it follows that the zero component $(g_D)_0=(3+144q+O(q^2))(q^{-1}+24+O(q))=3q^{-1}+216+O(q)$.
% \end{proof}

For a general $H\in\mathcal H_o(d)$, let $N_H$ be the primitive rank-two normal Eisenstein lattice used in the proof of Lemma \ref{lem:ramification-indices-eight-ball}.  The following calculation determines all rational quadratic divisors occurring in the principal part of $g_D$ whose restriction to $\B_o$ is the fixed mirror trace $H$.

\begin{lem}
\label{lem:gD-normal-vector-count}
Let $H\in\mathcal H_o(d)$ with $d\in\{6,9,15,18\}$,
and choose an $\mathscr O$-basis $(e_1,e_2)$ of $N_H$
as in \cite[Lemma A.9]{HeckmanLooijenga2002}.
Then
$A_{o,\Z}^\vee\cap(N_H\otimes_{\mathscr O}E)
\cap z_o^\perp=\mathscr O x_d$,
where $z_o$, $x_d$, and $-Q(x_d)$ are given by
\[
\begin{array}{c|c|c|c}
d & z_o & x_d & -Q(x_d) \\ \hline
6  & e_1+e_2 & (e_1-\bar\omega e_2)/\theta & 1/3 \\
9  & e_1+e_2 & (e_1-e_2)/2                & 1/2 \\
15 & e_1     & e_2-(\theta/6)e_1          & 5/6 \\
18 & e_1     & e_2                        & 1.
\end{array}
\]
The distinct rational quadratic divisors $D_x$ with
$x\in\mathscr O x_d$ and $0<-Q(x)\leq1$ are precisely
$D_{x_d},D_{\omega x_d},D_{\omega^2x_d}$,
together with
$D_{\theta x_6},D_{\omega\theta x_6},
D_{\omega^2\theta x_6}$ when $d=6$.
In the latter case, $Q(\theta x_6)=-1$.
Each of these divisors restricts to $H$
with multiplicity one.
\end{lem}

\begin{proof}
Let $\operatorname{pr}$ be the orthogonal projection
from $A_{\mathbb Z}\otimes\mathbb Q$ onto
$A_{o,\mathbb Z}\otimes\mathbb Q$.
Since $A_{o,\mathbb Z}$ is primitive in $A_{\mathbb Z}$,
every integer-valued linear functional on $A_{o,\mathbb Z}$
extends to one on $A_{\mathbb Z}$.
Unimodularity of $A_{\mathbb Z}$ therefore implies $\operatorname{pr}(A_{\mathbb Z})
=A_{o,\mathbb Z}^{\vee}$.
Indeed, restriction of linear functionals corresponds
to orthogonal projection under the identifications
given by the bilinear forms.

We claim that
\[
\operatorname{pr}(N_H)
=
A_{o,\mathbb Z}^{\vee}
\cap(N_H\otimes_{\mathscr O}E)
\cap z_o^\perp.
\]
The inclusion $\operatorname{pr}(N_H)
\subset
A_{o,\mathbb Z}^{\vee}
\cap(N_H\otimes_{\mathscr O}E)
\cap z_o^\perp.$ is immediate.
Conversely, let $x\in A_{o,\mathbb Z}^{\vee}
\cap(N_H\otimes_{\mathscr O}E)
\cap z_o^\perp$.
By the surjectivity just proved, choose
$a\in A_{\mathbb Z}$ with $\operatorname{pr}(a)=x$.
Then $a-x\in Ez_o$.
Since both $x$ and $z_o$ belong to
$N_H\otimes_{\mathscr O}E$, we obtain $a\in A\cap(N_H\otimes_{\mathscr O}E)=N_H$,
where the last equality follows from the primitivity
of $N_H$ in $A$.
Thus $x\in\operatorname{pr}(N_H)$, proving the claim.
% Let $\operatorname{pr}$ denote orthogonal projection from the underlying rational vector space of $A$ to $A_{o,\Z}\otimes\Q$.
% Unimodularity identifies the gluing quotients of $A_{o,\Z}$ \Masa{What does it mean?} and its orthogonal complement and, in particular, makes the projection of the underlying lattice of $A$ onto $A_{o,\Z}^\vee$ surjective.  Since $N_H$ is primitive, this gives
% \[
%  \operatorname{pr}(N_H)=A_{o,\Z}^\vee\cap(N_H\otimes_{\mathscr{O}}E)\cap z_o^\perp.
% \]
% Indeed, the inclusion $\operatorname{pr}(N_H)\subset A_{o,\Z}^\vee\cap(N_H\otimes_{\mathscr{O}}E)\cap z_o^\perp$ is immediate.  Conversely, an element of the right-hand side has a lift in the underlying lattice of $A$; the lift lies in the rational span of $N_H$ because its two orthogonal components lie in the lines spanned by $z_o$ and $x_d$, and therefore it belongs to $N_H$ by primitivity. \Masa{Is it well-known?}

Substitution in the four Gram matrices of \cite[Lemma A.9]{HeckmanLooijenga2002} shows that the displayed $x_d$ is orthogonal to $z_o$, that $\operatorname{pr}(N_H)=\mathscr{O} x_d$, and that its norm is the one shown in the table.  For $d=6$, one also has $\theta x_6=e_1-\bar\omega e_2\in A_{o,\Z}$ and $-Q(\theta x_6)=1$.  The nonzero elements of $\mathscr{O} x_d$ with $-Q(x)\leq1$ are the unit multiples of $x_d$, together with the unit multiples of $\theta x_6$ in the case $d=6$.  
For a nonzero vector $x$, its orbit is
$\mathscr O^\times x
=\{\pm x,\pm\omega x,\pm\omega^2x\}$.
Since $D_v=D_{-v}$, this orbit determines three distinct
divisors $D_x,D_{\omega x},D_{\omega^2x}$ on $\mathcal D_o$.
They are distinct because $x,\omega x,\omega^2x$
are pairwise nonproportional over $\mathbb R$.
All three restrict to the same hyperball $H_{A_o}(x)$,
each with multiplicity one.  The classes of $x_6,x_9,x_{15},x_{18}$ have quadratic values $2/3,1/2,1/6,0$ by $Q$ modulo $\Z$, respectively, while $\theta x_6$ is integral.  
The vectors $x_6,x_9,x_{15},x_{18}$ and $\theta x_6$
therefore contribute with leading coefficients
$105,30,-1,3$ and $3$ of $g_D$, respectively, from
\eqref{eq:gD-principal-part-corrected}.
The preceding enumeration exhausts the nonzero vectors
in $\mathscr O x_d$ with $-Q(x)\leq1$.
%Thus these shells \Masa{What is the ``Shell''?} occur with the coefficients $105,30,-1,3$ and $3$ appearing in \eqref{eq:gD-principal-part-corrected}, and there are no further contributions of norm at most $1$.
\end{proof}

\begin{thm}
\label{thm:another Borcherds lift}
Let $\Psi_D$ be the Borcherds lift of $g_D$.
Then $\Psi_D$ is a meromorphic automorphic form of
weight $108$ for the discriminant kernel of
$A_{o,\mathbb Z}$, with a finite-order character.
Moreover, its restriction
$\Psi_{D,o}:=\Psi_D|_{\mathbb B_o}$
is a nonzero meromorphic automorphic form of weight $108$
for $\mathrm U(A_o)$, with a finite-order character, and
\begin{equation}
 \operatorname{div}_{\mathbb B_o}(\Psi_{D,o})
 =
 324\mathbb B_o(\mathcal H_o(6))
 +90\mathbb B_o(\mathcal H_o(9))
 -3\mathbb B_o(\mathcal H_o(15))
 +9\mathbb B_o(\mathcal H_o(18)).
\label{eq:gD-upstairs-divisor-corrected}
\end{equation}
\end{thm}

\begin{proof}
The lattice $A_{o,\mathbb Z}$ has signature $(2,16)$.
By Lemma \ref{lem:gD-fourier-expansion}, the form $g_D$
has the required weight $1-16/2=-7$, its principal part
has integral coefficients, and the constant coefficient
of its zero component is $216$.
Since the zero class is the only isotropic class in
$C(K)$, all its Fourier coefficients of nonpositive
index are integral.
Thus the Borcherds product theorem
\cite[Theorem 13.3]{Borcherds1998Grassmannians}
gives a meromorphic automorphic form $\Psi_D$ on
$\mathcal D_o$ of weight $216/2=108$.

The group for which automorphy is automatic is the
discriminant kernel
\[
\widetilde{\operatorname O}^{+}(A_{o,\mathbb Z})
:=
\ker\!\left(
\operatorname O^{+}(A_{o,\mathbb Z})
\longrightarrow
\operatorname O(C(A_{o,\mathbb Z}),q_{A_{o,\mathbb Z}})
\right).
\]
An element of this kernel fixes every discriminant class
and therefore fixes every component of the input $g_D$.
More generally, the same Borcherds theorem gives
automorphy, with a character, for the subgroup of
$\operatorname O^{+}(A_{o,\mathbb Z})$ that fixes $g_D$
under the induced permutation of its discriminant
components.
Thus an isometry need not fix every discriminant class:
it suffices that the resulting permutation leave $g_D$
unchanged.

Every element of $\mathrm U(A_o)$ preserves the underlying
integral bilinear form and the chosen component
$\mathcal D_o$, and hence belongs to
$\operatorname O^{+}(A_{o,\mathbb Z})$.
It therefore remains to prove that
$u g_D=g_D$ for every $u\in\mathrm U(A_o)$.
We establish this by comparing the action on the
discriminant group with an isometry of the lattice $K$
used to construct $\Theta_K$.

Fix $u\in\mathrm U(A_o)$.
By \cite[Proposition A.8]{HeckmanLooijenga2002},
it extends uniquely to an isometry
$\widetilde u\in\mathrm U(A)$ preserving $A_o$.
The extension also preserves the orthogonal complement
$\mathscr O z_o$.
Since this complement has rank one over $\mathscr O$,
we have $\widetilde u(z_o)=\epsilon z_o$
for some $\epsilon\in\mathscr O^\times$.
Under the fixed identification
$\mathscr O z_o\cong-K$, its action on the complement
is therefore multiplication by $\epsilon$.

The primitive orthogonal sublattices
$A_{o,\mathbb Z}$ and $-K$ are contained in the
unimodular lattice $A_{\mathbb Z}$, and the image of $A_{\mathbb Z}/(A_{o,\mathbb Z}\oplus(-K))$ in $C(A_{o,\mathbb Z})\oplus C(-K)$
is the graph of $\iota_o$.
We identify $C(-K)$ with $C(K)$ as abelian groups.
Consequently, for
$x\in A_{o,\mathbb Z}^{\vee}$ and
$y\in(-K)^\vee$, the condition that their sum belong
to the original lattice is $x+y\in A_{\mathbb Z}$.
Here $K$ and $-K$ have the same underlying abelian group,
but opposite quadratic forms.

Let $\bar u$ be the action of $u$ on
$C(A_{o,\mathbb Z})$.
Applying $\widetilde u$ to a vector $x+y\in A_{\mathbb Z}$
gives $ux+\epsilon y\in A_{\mathbb Z}$.
The preceding condition then gives
\begin{equation}
\iota_o\circ\bar u=\epsilon\,\iota_o.
\label{eq:gluing-Ao-K-action}
\end{equation}
Thus, after identifying the discriminant groups by
$\iota_o$, the action of $u$ on discriminant classes
is precisely multiplication by $\epsilon$.
For a class $\gamma\in C(K)$, its theta component is
\[
(\Theta_K)_\gamma(\tau)
=
\sum_{\substack{v\in K^\vee\\v+K=\gamma}}
q^{Q_K(v)}.
\]
Multiplication by $\epsilon$ is an isometry of $K$
and induces a bijection between the vectors representing
$\gamma$ and those representing $\epsilon\gamma$.
It follows that
\[
(\Theta_K)_{\epsilon\gamma}(\tau)
=
\sum_{\substack{v\in K^\vee\\v+K=\gamma}}
q^{Q_K(\epsilon v)}
=
(\Theta_K)_\gamma(\tau).
\]
In these formulas, $\gamma$ always denotes a class in
$C(K)$, while $v$ denotes an actual vector in $K^\vee$.

By \eqref{eq:gluing-Ao-K-action}, the induced action on
the group algebra is
$u\mathbf e_\gamma=\mathbf e_{\epsilon\gamma}$.
Consequently, for
$F=\sum_\gamma F_\gamma\mathbf e_\gamma$,
the $\gamma$-component of $uF$ is
$F_{\epsilon^{-1}\gamma}$.
The equality of theta components just proved therefore
implies $u\Theta_K=\Theta_K$.
This action only permutes the components and does not
change the variable $\tau$.
It commutes with the componentwise Serre derivatives
and with division by 
$\Delta_{\rm mod}$.
Hence, \eqref{eq:gA-gD-definitions} implies $u g_D = g_D$.
Thus every element of $\mathrm U(A_o)$ fixes the input
of the Borcherds lift.

It follows from
\cite[Theorem 13.3]{Borcherds1998Grassmannians}
that $\Psi_D$ transforms under $\mathrm U(A_o)$
with a character.
More precisely, this is the restriction of the
character on the subgroup of
$\operatorname O^{+}(A_{o,\mathbb Z})$ preserving $g_D$.
Since the weight is the integer $108$, the multiplier
is an ordinary character.
Its restriction to the discriminant kernel has finite
order by
\cite[Remark 3.19]{Bruinier2002BorcherdsProducts},
applied in signature $(2,16)$.
The discriminant kernel has finite index in the
subgroup preserving $g_D$, so the character on that
subgroup, and hence on $\mathrm U(A_o)$, also has
finite order.

We next justify the restriction to $\mathbb B_o$.
The zeros and poles of $\Psi_D$ lie on rational
quadratic divisors
$D_x$ with $x\in A_{o,\mathbb Z}^{\vee}$ and $Q(x)<0$.
As explained above, each such divisor restricts to
the proper hyperball $H_{A_o}(x)$ with multiplicity one;
see also \cite[Lemmas 6.1 and 6.2]{Hofmann2014}.
In particular, $\mathbb B_o$ is not contained in the
zero or pole divisor of $\Psi_D$, so the restriction
$\Psi_{D,o}$ is a nonzero meromorphic automorphic form.
The embedding
$\mathbb B_o\hookrightarrow\mathcal D_o$
is $\mathrm U(A_o)$-equivariant, and the tautological
line bundle on $\mathcal D_o$ restricts to the
tautological line bundle on $\mathbb B_o$.
Therefore $\Psi_{D,o}$ has weight $108$ and is
automorphic for $\mathrm U(A_o)$ with the restricted
finite-order character.

It remains to compute its divisor.
Write the Fourier expansion as
\[
g_D(\tau)
=
\sum_{\gamma\in C(K)}
\sum_{m\in\mathbb Z+q_K(\gamma)}
c(\gamma,m)q^m\mathbf e_\gamma.
\]
The divisor formula for the Borcherds lift is
\[
\operatorname{div}_{\mathcal D_o}(\Psi_D)
=
\frac12
\sum_{\substack{x\in A_{o,\mathbb Z}^{\vee}\\Q(x)<0}}
c\bigl(\iota_o(x+A_{o,\mathbb Z}),Q(x)\bigr)D_x.
\]
The sum is locally finite, and only coefficients in
the principal part contribute.
The factor $1/2$ accounts for the vectors $x$ and $-x$,
which define the same divisor and have equal
coefficients.
Thus each pair $\{x,-x\}$ contributes its Fourier
coefficient once.
If further vectors define the same divisor, their
contributions must also be added.

We first show that every divisor obtained after
restriction belongs to the mirror-trace arrangement.
Take a vector $x$ with nonzero coefficient in this
divisor formula, and put
$\gamma=\iota_o(x+A_{o,\mathbb Z})$.
By \eqref{eq:gD-principal-part-corrected}, the possible
values of $Q(x)$ are
$-1,-5/6,-1/2,-1/3$.
Since $\iota_o$ preserves the discriminant quadratic
forms, $q_K(\gamma)$ is the residue of $Q(x)$ modulo
$\mathbb Z$.

The coset calculation in
Lemma \ref{lem:gD-fourier-expansion}
shows that each such class $\gamma$ has a representative
$y\in K^\vee$ with the following norm:
\[
\begin{array}{c|c|c|c}
Q(x) & q_K(\gamma) & Q_K(y) & c(\gamma,Q(x))\\ \hline
-1   & 0   & 0   & 3\\
-5/6 & 1/6 & 1/6 & -1\\
-1/2 & 1/2 & 1/2 & 30\\
-1/3 & 2/3 & 2/3 & 105
\end{array}
\]
In particular, we choose $y$ so that
$y+K=\gamma$ and $Q_K(y)=1+Q(x)$.
In the first row, $\gamma=0$ and we may take $y=0$.

Now regard $y$ as a vector of $(-K)^\vee$ inside
$A_{\mathbb Z}\otimes\mathbb Q$.
The equality
$y+K=\iota_o(x+A_{o,\mathbb Z})$
is exactly the gluing condition established above.
It follows that $r:=x+y$ belongs to $A_{\mathbb Z}$.
The two summands are orthogonal, and the quadratic
form on $-K$ is $-Q_K$, so
\[
Q_{A_{\mathbb Z}}(r)
=
Q(x)-Q_K(y)
=
-1.
\]
By the normalization
$Q_{A_{\mathbb Z}}(r)=-\psi(r,r)/3$,
this means that $\psi(r,r)=3$.
Thus $r$ defines a root hyperball in $\mathbb B_A$.

Moreover, $y$ lies in $Ez_o$ and is orthogonal to
$A_o$.
Hence the equations defined by $r=x+y$ and by $x$
agree on $V_{A_o}$, giving $H_A(r)\cap\mathbb B_o=H_{A_o}(x)$.
Thus the restriction of every contributing rational
quadratic divisor is a mirror trace.
There are therefore no divisor components outside
$\mathcal H_o$.

Finally, we determine the multiplicity along each
mirror trace.
For a nonzero vector $x$, the unit orbit is $\mathscr O^\times x
=
\{\pm x,\pm\omega x,\pm\omega^2x\}$.
The three pairs modulo sign define the three distinct
orthogonal divisors
$D_x,D_{\omega x},D_{\omega^2x}$:
the vectors $x,\omega x,\omega^2x$ are pairwise
nonproportional over $\mathbb R$.
On the other hand, they span the same line over $E$,
so all three divisors restrict to the same hyperball
$H_{A_o}(x)$.
Each restriction has multiplicity one.
This explains why their three contributions must be
added on $\mathbb B_o$.

For a fixed mirror trace $H\in\mathcal H_o(d)$,
Lemma \ref{lem:gD-normal-vector-count} lists all the
vectors that can contribute to its multiplicity.
When $d=6$, there are three divisors arising from the
unit multiples of $x_6$, with $Q(x_6)=-1/3$, and
three more arising from the unit multiples of
$\theta x_6$, with $Q(\theta x_6)=-1$.
The corresponding coefficients are $105$ and $3$,
so the total order along $H$ is
$3\cdot105+3\cdot3=324$.

For $d=9,15,18$, the same lemma gives only the three
divisors arising from the unit multiples of $x_d$.
Their coefficients are respectively $30,-1,3$.
The resulting orders are therefore
$3\cdot30=90$, $3\cdot(-1)=-3$, and $3\cdot3=9$.
In particular, the negative order for $d=15$
describes a pole of order three.
These computations prove
\eqref{eq:gD-upstairs-divisor-corrected}.
\end{proof}

\section{Birational geometry of Looijenga's and semi-toroidal compactifications}
\label{sec:semitoroidal}

We use the two Borcherds products constructed in
Section \ref{sec:automorphic-forms}
together with the cusp intersection calculations of
Corollary \ref{cor:cusp-intersection-matrix-corrected}
to determine divisor relations, local $\Q$-Cartier
conditions, and the rational class and Picard groups.
We then study canonical divisors and the arrangement blow-up.

We now arrange the divisor relations as a consequence of two Borcherds products.

\begin{lem}
\label{lem:two-automorphic-relations-corrected}
On $X_o^{j_{\mathcal K}}$ one has
\begin{align}
 D_{15}&\sim_{\Q}24\Lambda_{\rm H}+4D_6+2D_9,
 \label{eq:D15-corrected}\\
 D_{18}&\sim_{\Q}60\Lambda_{\rm H}-14D_6-8D_9.
 \label{eq:D18-corrected}
\end{align}

\end{lem}

\begin{proof}
Multiplying \eqref{eq:allcock-linear-equivalence-BB} by $3$ gives the first relation on $X_o^{\mathrm{bb}}$.  The morphism $X_o^{j_{\mathcal K}}\to X_o^{\mathrm{bb}}$ is an isomorphism away from the two cusps, and its fibres there have dimensions $3$ and $1$ by \eqref{eq:cusp-fibre-dimensions}; hence it has no exceptional divisor.  Taking strict transforms of the divisor relation gives 
\[2D_6+2D_9+3D_{15}+D_{18}\sim_{\Q}132\lambda_{\rm aut}.\]

On the other hand, divide the four coefficients in \eqref{eq:gD-upstairs-divisor-corrected} by the ramification indices $(6,3,1,3)$ from Lemma \ref{lem:ramification-indices-eight-ball}. The weight is $108$ in the  normalization of $\Lambda_{\rm H}$.  Taking strict transforms, as above, proves 
\[54D_6+30D_9-3D_{15}+3D_{18}\sim_{\Q}108\lambda_{\rm aut}.\]
Now solving the system of the two equations above yields the conclusion.
\end{proof}
 $\Lambda_{\rm H}$ is pulled back from $X_o^{\mathrm{bb}}$, so its relative numerical class is zero.
Hence in $N^1(X_o^{j_{\mathcal K}}/X_o^{\mathrm{bb}})_{\Q}$ this becomes
\[ [D_{15}]=4[D_6]+2[D_9],\qquad [D_{18}]=-14[D_6]-8[D_9].\]

\begin{prop}
\label{prop:qcartier-at-cusps-corrected}
Let $D=\sum_{d=6,9,15,18}a_d\Delta(d)$ be a rational Weil divisor on $X_o^{\mathrm{bb}}$, and put
\begin{equation}
 u(D):=a_6+4a_{15}-14a_{18},\qquad v(D):=a_9+2a_{15}-8a_{18}.
\label{eq:uv-cartier-coordinates}
\end{equation}
Then $D$ is $\Q$-Cartier at $\infty_\theta$ if and only if $u(D)=v(D)=0$, and it is $\Q$-Cartier at $\infty_0$ if and only if $v(D)=0$.  Consequently, $D$ is globally $\Q$-Cartier if and only if $u(D)=v(D)=0$.
\end{prop}

\begin{proof}
Equations \eqref{eq:D15-corrected} and \eqref{eq:D18-corrected} give
\begin{equation}
 D^j\sim_{\Q}(24a_{15}+60a_{18})\lambda_{\rm aut}+u(D)D_6+v(D)D_9.
\label{eq:D-in-uv-coordinates}
\end{equation}
If $D$ is $\Q$-Cartier near a cusp, then its pullback to $X_o^{j_{\mathcal K}}$ has degree zero on every contracted curve.  The intersection matrix \eqref{eq:cusp-intersection-matrix-corrected} gives the stated necessary conditions.

Conversely, if $u(D)=v(D)=0$, then \eqref{eq:D-in-uv-coordinates} makes $D^j$ $\Q$-linearly equivalent to a multiple of the $\Q$-Cartier divisor $\Lambda_{\rm H}$.  Since $X_o^{j_{\mathcal K}}\to X_o^{\mathrm{bb}}$ is small, pushing forward a principal multiple of this equality shows that $D$ is $\Q$-Cartier.  Near $\infty_0$, choose a neighborhood disjoint from $\Delta(6)$; if $v(D)=0$, the same argument applies because the term $u(D)D_6$ is absent there.  Away from the cusps, $X_o^{\mathrm{bb}}$ is locally a finite quotient of a smooth germ.  Therefore, every Weil $\Q$-divisor is locally $\Q$-Cartier there by \cite[Lemma 5.16]{komo}.
\end{proof}
This yields the following concrete corollary.
\begin{cor}
\label{cor:individual-qcartier-corrected}
The divisor $\Delta(6)$ is not $\Q$-Cartier at $\infty_\theta$ and is Cartier near $\infty_0$, which it does not meet.  Each of $\Delta(9)$, $\Delta(15)$, and $\Delta(18)$ is not $\Q$-Cartier at either cusp.
\end{cor}

Now we determine the invariants of $X_o^{j_{\mathcal K}}$.
\begin{thm}
\label{thm:picard-numbers-corrected}
One has
\[\mathrm{Cl}(X_o^{j_{\mathcal K}})_{\Q}=\Pic(X_o^{j_{\mathcal K}})_{\Q}=\Q[\lambda_{\rm aut}]\oplus\Q[D_6]\oplus\Q[D_9].\]
Thus $X_o^{j_{\mathcal K}}$ is $\Q$-factorial and $\rho(X_o^{j_{\mathcal K}})=3$.  Moreover,
\[\Pic(X_o^{\mathrm{bb}})_{\Q}=\Q[\lambda_{\rm aut}],\qquad \rho(X_o^{\mathrm{bb}})=1,\qquad \rho(X_o^{j_{\mathcal K}}/X_o^{\mathrm{bb}})=2.\]
\end{thm}

\begin{proof}
Let $\mathscr U:=X_o^{j_{\mathcal K}}\setminus(D_6\cup D_9\cup D_{15}\cup D_{18})$.  Put $G=\SL_2\times\mathbb G_m$, where the second factor acts on $H_4\oplus H_6$ with weights $(2,3)$, and let $U\subset H_4\oplus H_6$ be the principal open subset on which $\operatorname{Disc}(f_0^3+f_1^2)\neq0$, where $\operatorname{Disc}$ denotes the discriminant.  \cite[Theorem 7.2]{Miranda1981} identifies $U/G$ with $\M$, and $[U/G]$ is a smooth Deligne-Mumford stack.

The ring of $U$ is a localization of a polynomial ring and hence is factorial, so $\Pic(U)=0$.  By \cite[Theorem 4.2.2]{brion2018linearization},
there is an exact sequence
\[
 \mathcal O(U)^*/\C^*\longrightarrow X^*(G)\longrightarrow\Pic^G(U)\longrightarrow0
\]
shows that $\Pic^G(U)_{\Q}=0$, where we utilized the fact that the group $\SL_2$ has no characters, while the invertible discriminant semi-invariant has nonzero $\mathbb G_m$-weight and therefore spans $X^*(G)_{\Q}$.  By \cite[Proposition 6.1]{Vistoli1989Intersection}, $\mathrm{Cl}(\M)_{\Q}\cong\Pic([U/G])_{\Q}=0$.

The symmetric difference between $\mathscr U$ and $\M$ is contained in the two cusp fibres.  By \eqref{eq:cusp-fibre-dimensions}, these have dimensions $3$ and $1$ inside the eight-dimensional variety $X_o^{j_{\mathcal K}}$, hence codimensions at least $5$ and $7$.  Removing or adding these subsets does not change the class group, so $\mathrm{Cl}(\mathscr U)_{\Q}=0$.  \cite[II, Proposition 6.5]{Ha} therefore gives a surjection $\bigoplus_{d=6,9,15,18}\Q[D_d]\to\mathrm{Cl}(X_o^{j_{\mathcal K}})_{\Q}$.  Equations \eqref{eq:D15-corrected} and \eqref{eq:D18-corrected} show that this class group is generated by $\lambda_{\rm aut},D_6,D_9$, all of which are $\Q$-Cartier by Looijenga's construction of the semi-toroidal compactification $X_o^{j_{\mathcal K}}$.

These three classes are independent. 
Indeed, by Corollary \ref{cor:cusp-intersection-matrix-corrected}, if $a\lambda_{\rm aut}+bD_6+cD_9\equiv0$, intersection with $C_{\theta,6}$ gives $b=0$, and intersection with $C_0$ gives $c=0$, because $\lambda_{\rm aut}$ is pulled back from $X_o^{\mathrm{bb}}$.  Then $a=0$ because $\lambda_{\rm aut}$ is the pullback of a nonzero ample class on $X_o^{\mathrm{bb}}$.  This proves the claim for $\mathrm{Cl}(X_o^{j_{\mathcal K}})_{\Q}$.

The only nontrivial fibres of $X_o^{j_{\mathcal K}}\to X_o^{\mathrm{bb}}$ have dimensions $3$ and $1$, so the exceptional locus has codimension greater than one and the morphism is small.  Taking strict transforms therefore identifies $\mathrm{Cl}(X_o^{\mathrm{bb}})_{\Q}$ with $\mathrm{Cl}(X_o^{j_{\mathcal K}})_{\Q}$.  The pullback of a $\Q$-Cartier $\Q$-divisor on $X_o^{\mathrm{bb}}$ has degree zero on all contracted curves, and \eqref{eq:cusp-intersection-matrix-corrected} forces its $D_6$- and $D_9$-coefficients to vanish.  Conversely, $\lambda_{\rm aut}$ is $\Q$-Cartier on $X_o^{\mathrm{bb}}$.  This proves the rest of the claims.
\end{proof}
As a direct consequence of the above theorem, the variety $X_o^{\mathrm{bb}}$ is not $\Q$-factorial; indeed, $\operatorname{rank}\mathrm{Cl}(X_o^{\mathrm{bb}})_{\Q}=3$ and $\operatorname{rank}(\mathrm{Cl}(X_o^{\mathrm{bb}})_{\Q}/\Pic(X_o^{\mathrm{bb}})_{\Q})=2$.
\begin{cor}
\label{cor:picard-semitoric-exhibition}
One has $N^1(X_o^{j_{\mathcal K}}/X_o^{\mathrm{bb}})_{\Q}=\Q[D_6]\oplus\Q[D_9]$ and $\rho(X_o^{j_{\mathcal K}}/X_o^{\mathrm{bb}})=2$.
The classes $\lambda_{\rm aut},D_6,D_9$ form a basis of $\Pic(X_o^{j_{\mathcal K}})_{\Q}$.  Equivalently, by \eqref{eq:D18-corrected}, one may replace $\lambda_{\rm aut}$ by $D_{18}$.
\end{cor}

\begin{proof}
This follows immediately by taking the quotient of the two Picard groups by $\Q[\lambda_{\mathrm{aut}}]$ in Theorem \ref{thm:picard-numbers-corrected}. On the other hand, the linear independence follows from \eqref{eq:cusp-intersection-matrix-corrected}.
\end{proof}

\begin{rem}
At $\infty_\theta$, the restrictions of $D_6$ and $D_9$ give two independent local numerical directions, whereas at $\infty_0$ only the $D_9$-direction occurs.  The latter is the restriction of the same global class $D_9$, not a third global relative class.  This is why the relative Picard number is $2$ rather than $3$.
\end{rem}
Let $D_{15}^M,D_{18}^M$ denote the divisors of generic types $II$ and $I_2$.
\begin{prop}\label{prop:miranda-boundary-classes-corrected}
Put $\mathcal R=\operatorname{Res}(f_0,f_1)$ and $\mathcal P=\operatorname{Disc}(f_0^3+f_1^2)$.
There is a polynomial $\mathrm{SL}_2$-invariant $\mathcal T$ such that
$\mathcal P=\mathcal R^3\mathcal T$.
The weighted degrees of $\mathcal R,\mathcal T,\mathcal P$ are
$24,60,132$, respectively, and their divisors on Miranda's
compactification are
\begin{equation}\label{eq:classical-invariant-divisors}
 \operatorname{div}(\mathcal R)=D_{15}^M,\qquad
 \operatorname{div}(\mathcal T)=D_{18}^M,\qquad
 \operatorname{div}(\mathcal P)=3D_{15}^M+D_{18}^M.
\end{equation}
In particular, $D_{15}^M\sim_{\mathbb Q}24\eta$ and
$D_{18}^M\sim_{\mathbb Q}60\eta$.
\end{prop}

\begin{proof}
The coefficients of $f_0$ and $f_1$ have weights two and three.
The resultant has bidegree $(6,4)$ and therefore weight $24$.
It vanishes to order one along the common-root divisor $D_{15}^M$.
The discriminant of a binary form of degree twelve has coefficient
degree $22$, so $\mathcal P$ has weight $22\cdot6=132$.
Its zero divisor is supported on $D_{15}^M\cup D_{18}^M$.

To compute its order along $D_{15}^M$, take a general transverse
analytic family at a pair with one common simple root.
After choosing a coordinate and a transverse parameter, write $f_0=xu(x,t)$, $f_1=(x-t)v(x,t)$, and $u(0,0)v(0,0)\ne0$.
Here $t$ is a local equation of the resultant divisor.
Set $t=s^2$ and $x=s^2+s^3z$. Then
\[
 s^{-6}(f_0^3+f_1^2)
       =u(0,0)^3+v(0,0)^2z^2+O(s).
\]
The limiting polynomial has two distinct roots.
Thus the two roots near the common root differ by $s^3$ times
a unit. Their squared difference has order six in $s$,
or three in $t$. All other discriminant factors are units.
Consequently $\operatorname{ord}_{\mathcal R}(\mathcal P)=3$.
The coefficient ring is a UFD and the generic resultant is
irreducible, so $\mathcal R^3$ divides $\mathcal P$ as a polynomial.
The quotient $\mathcal T$ is an invariant of weight $60$.

At a general point of $D_{18}^M$, the polynomial $f_0^3+f_1^2$
has one double root, at which $f_1$ is nonzero.
Varying $f_1$ in a direction nonzero at that root gives a
transverse smoothing, since its first-order contribution is
$2f_1\delta f_1\ne0$. The local model is a unit times $x^2-t$,
whose discriminant vanishes to order one.
This proves \eqref{eq:classical-invariant-divisors}, and the
weighted degrees give the asserted divisor class relations.
\end{proof}

% \begin{rem}
% \label{rem:miranda-automorphic-polarizations}
% Let $\mathcal L_M$ be the arrangement polarization on $\M^M$.  Since $\eta\sim_{\Q}2\mathcal L_M$, the preceding formulas become $D_{15}^M\sim_{\Q}48\mathcal L_M$ and $D_{18}^M\sim_{\Q}120\mathcal L_M$.  These are the pushforwards of \eqref{eq:D15-corrected} and \eqref{eq:D18-corrected}, because the relative terms supported on $D_6$ and $D_9$ disappear under the Looijenga contraction.
% \end{rem}

\begin{cor}
\label{cor:Miranda-class-group-corrected}
One has $\mathrm{Cl}(\M^M)_{\Q}=\Pic(\M^M)_{\Q}=\Q[\eta]$ and $\rho(\M^M)=1$.  In particular, $\M^M$ is $\Q$-factorial and klt.
\end{cor}

\begin{proof}
Remove $D_{15}^M\cup D_{18}^M$ from $\M^M$.  Up to a subset of codimension at least two, the result is the reduced-discriminant open whose rational class group was shown to vanish in the proof of Theorem \ref{thm:picard-numbers-corrected}.  \cite[II, Proposition 6.5]{Ha} shows that $\mathrm{Cl}(\M^M)_{\Q}$ is generated by $D_{15}^M,D_{18}^M$, and Proposition \ref{prop:miranda-boundary-classes-corrected} identifies both with multiples of the ample $\Q$-Cartier class $\eta$.

The space $\M^M$ is of klt type by \cite{braun2024reductive,zhuang2024direct}, because $Y_{\mathrm{Mir}}$ has only quotient singularities.  Choose an effective $\Q$-divisor $\Delta_M$ such that $(\M^M,\Delta_M)$ is klt.  Since $\M^M$ is $\Q$-factorial, both $K_{\M^M}$ and $\Delta_M$ are $\Q$-Cartier, and for every prime divisor $E$ over $\M^M$ one has $A_{\M^M}(E)=A_{\M^M,\Delta_M}(E)+\operatorname{ord}_E(\Delta_M)>0$.  Hence $\M^M$ is klt.
\end{proof}

\begin{rem}
\label{rem:comparison-shen}
In a revised version of \cite{Shen2025MMP} communicated to us by the author, \cite[Corollary 6.2]{Shen2025MMP} states that the pair
$\bigl(\mathcal{M}^{M},\frac{2}{3}D_{18}^{M}\bigr)$ is not log canonical.
This is compatible with Corollary \ref{cor:Miranda-class-group-corrected}, which asserts that the underlying
variety $\mathcal{M}^{M}$ is klt.
On the Baily--Borel side, the relevant log canonical pair is
$\bigl(X_o^{\mathrm{bb}},B_{\mathrm{br}}^{\mathrm{bb}}\bigr)$.
Under the birational map $X_o^{\mathrm{bb}}\dashrightarrow\mathcal{M}^{M}$,
the divisors $\Delta(6)^{\mathrm{bb}}$ and $\Delta(9)^{\mathrm{bb}}$ are
contracted, while $\Delta(18)^{\mathrm{bb}}$ has strict transform
$D_{18}^{M}$. Thus the induced boundary on $\mathcal{M}^{M}$ is
$\frac{2}{3}D_{18}^{M}$.
By Corollary~6.10, $K_{X_o^{\mathrm{bb}}}+B_{\mathrm{br}}^{\mathrm{bb}}$
is $\mathbb{Q}$-Cartier, although $K_{X_o^{\mathrm{bb}}}$ is not.

\end{rem}

\begin{cor}
\label{cor:BB-not-Q-Gorenstein}
One has $K_{X_o^{\mathrm{bb}}}+B_{\rm br}^{\mathrm{bb}}\sim_{\Q}9\lambda_{\rm aut}$, while $u(B_{\rm br})=-17/2$ and $v(B_{\rm br})=-14/3$.  Thus neither $B_{\rm br}^{\mathrm{bb}}$ nor $K_{X_o^{\mathrm{bb}}}$ is $\Q$-Cartier, and $X_o^{\mathrm{bb}}$ is not $\Q$-Gorenstein.  On $X_o^{j_{\mathcal K}}$ one has
\begin{equation}
 K_{X_o^{j_{\mathcal K}}}\sim_{\Q}-31\Lambda_{\rm H}+\frac{17}{2}D_6+\frac{14}{3}D_9.
\label{eq:canonical-Y-corrected}
\end{equation}
\end{cor}

\begin{proof}
Pushing the toroidal relation of Proposition \ref{prop:toroidal-canonical-eight-ball} to the Baily-Borel compactification gives $K_{X_o^{\mathrm{bb}}}+B_{\rm br}^{\mathrm{bb}}\sim_{\Q}9\lambda_{\rm aut}$.  The values of $u(B_{\rm br})$ and $v(B_{\rm br})$ follow from \eqref{eq:uv-cartier-coordinates}, and Proposition \ref{prop:qcartier-at-cusps-corrected} gives the non-$\Q$-Cartier assertions.

Since $X_o^{j_{\mathcal K}}\to X_o^{\mathrm{bb}}$ is small, taking strict transforms of the log-canonical relation gives
\[
 K_{X_o^{j_{\mathcal K}}}+\frac56D_6+\frac23D_9+\frac23D_{18}\sim_{\Q}9\lambda_{\rm aut}.
\]
Substituting \eqref{eq:D18-corrected} gives \eqref{eq:canonical-Y-corrected}.
\end{proof}

\begin{defn}
\label{def:fano-type}
A normal projective variety $V$ is of \emph{Fano type} if there is an effective $\Q$-divisor $\Delta$ such that $(V,\Delta)$ is klt and $-(K_V+\Delta)$ is ample.  
\end{defn}

It is well-known by \cite[Corollary 1.3.1]{birkar2010existence} that if $V$ is $\Q$-factorial and of Fano type, then $V$ is a Mori dream space.

\begin{prop}
\label{prop:bb-arrangement-not-log-fano}
There is no effective divisor $\Delta=\sum_{d=6,9,15,18}a_d\Delta(d)$ with $a_d\in\Q_{\geq0}$ such that $K_{X_o^{\mathrm{bb}}}+\Delta$ is $\Q$-Cartier and $-(K_{X_o^{\mathrm{bb}}}+\Delta)$ is ample.  More precisely, every such $\Delta$ satisfies
\begin{equation}
 K_{X_o^{\mathrm{bb}}}+\Delta\sim_{\Q}(24a_{15}+60a_{18}-31)\lambda_{\rm aut},
\label{eq:arrangement-log-canonical-class}
\end{equation}
where
\begin{equation}
 24a_{15}+60a_{18}-31\geq\frac{38}{7}+\frac{288}{7}a_{15}\geq\frac{38}{7}>0.
\label{eq:arrangement-log-canonical-lower-bound}
\end{equation}
The minimum is attained at $\Delta_{\min}=\frac4{21}\Delta(9)+\frac{17}{28}\Delta(18)$, for which $K_{X_o^{\mathrm{bb}}}+\Delta_{\min}\sim_{\Q}\frac{38}{7}\lambda_{\rm aut}$.
\end{prop}

\begin{proof}
Since $K_{X_o^{\mathrm{bb}}}+B_{\rm br}^{\mathrm{bb}}$ is $\Q$-Cartier, Proposition \ref{prop:qcartier-at-cusps-corrected} shows that $K_{X_o^{\mathrm{bb}}}+\Delta$ is $\Q$-Cartier if and only if $u(\Delta)=u(B_{\rm br})=-17/2$ and $v(\Delta)=v(B_{\rm br})=-14/3$.  Using \eqref{eq:D15-corrected}, \eqref{eq:D18-corrected}, and \eqref{eq:canonical-Y-corrected} on the small modification then gives \eqref{eq:arrangement-log-canonical-class}.

The first Cartier equation gives $a_6=-17/2-4a_{15}+14a_{18}\geq0$, so $a_{18}\geq17/28+(2/7)a_{15}$.  Substitution proves \eqref{eq:arrangement-log-canonical-lower-bound}.  Equality forces $a_6=a_{15}=0$ and $a_{18}=17/28$; the second Cartier equation then gives $a_9=4/21$.  Since $\lambda_{\rm aut}$ is ample on $X_o^{\mathrm{bb}}$, the divisor $K_{X_o^{\mathrm{bb}}}+\Delta$ is ample rather than anti-ample.
\end{proof}

\begin{rem}
\label{rem:bb-fano-type-open}
Corollary \ref{cor:BB-not-Q-Gorenstein} shows that $X_o^{\mathrm{bb}}$ is not $\Q$-Fano, but does not by itself rule out Fano type in the sense of Definition \ref{def:fano-type}.  Proposition \ref{prop:bb-arrangement-not-log-fano} only treats classes in the cone generated by the four arrangement divisors.  Thus the unrestricted question whether $X_o^{\mathrm{bb}}$ is of Fano type remains open.  A negative answer would follow, for example, from $\overline{\operatorname{Eff}}(X_o^{j_{\mathcal K}})=\operatorname{Cone}(D_6,D_9,D_{15},D_{18})$.  Alternatively, one may determine the full Cox ring of the small $\Q$-factorialization and apply \cite[Corollary 5.3]{GOST15}.
\end{rem}
Finally, we prove the following for consistency. 
Let $\pi\colon \widetilde X_o^{\mathcal K}\to X_o^{j_{\mathcal K}}$ be the arrangement blow-up.
\begin{lem}\label{lem:no-cuspidal-exceptional-divisor}
Every nonempty arrangement stratum of $X_o^{j_{\mathcal K}}$ meeting a cusp fibre is the closure of an arrangement stratum in $X_o$. Consequently, the arrangement blow-up $\pi\colon \widetilde X_o^{\mathcal K}\to X_o^{j_{\mathcal K}}$
has no exceptional prime divisor whose centre is contained in the boundary
$X_o^{j_{\mathcal K}}\setminus X_o$.
\end{lem}

\begin{proof}
Pass to a normal neat finite-index subgroup of $\Gamma_o$, sufficiently small so that the arrangement hypersurfaces have no self-intersections, as in \cite[Lemma 5.1]{Looijenga2003Ball}. Fix a cusp represented by an isotropic line $I$, and put $J:=j_{\mathcal K}(I)$.
Looijenga's local model identifies a neighbourhood of the
boundary stratum with a locally trivial space over $X(J)$.
We work on the affine covering space $\mathbb P(V_{A_o}/J)
 \setminus \mathbb P(I^\perp/J)$
of $X(J)$.

For every $H\in\mathcal K$ containing $I$, the closure of the
corresponding arrangement divisor is locally the inverse image
of the affine hyperplane
\[
 A_H:=
 \mathbb P(H/J)
 \setminus
 \mathbb P\bigl((H\cap I^\perp)/J\bigr)
 \subset
 \mathbb P(V_{A_o}/J)
 \setminus \mathbb P(I^\perp/J);
\]
see \cite[Lemma 5.2]{Looijenga2003Ball} and
\cite[Sections 10.2-10.3]{HeckmanLooijenga2002}.

Consider a finite collection $\mathcal S$ of such hyperplanes and set
\[
 W_{\mathcal S}:=\bigcap_{H\in\mathcal S}H.
\]
After lifting an intersection from $X(J)$ to this affine
covering space, one has
\[
 \bigcap_{H\in\mathcal S}A_H
 =
 \mathbb{P}(W_{\mathcal S}/J)
 \setminus
 \mathbb{P}\bigl((W_{\mathcal S}\cap I^\perp)/J\bigr).
\]
This intersection is nonempty precisely when
$W_{\mathcal S}\not\subset I^\perp$. In that case $I\subset W_{\mathcal S}$, and one can choose
$v\in W_{\mathcal S}$ with $\psi(v,I)\neq0$. The restriction of the Hermitian form to
$\langle I,v\rangle$ is indefinite; hence
\[
 \mathbb{P}(W_{\mathcal S})\cap B_o\neq\emptyset.\]

Thus the same collection of mirror traces already has a nonempty intersection in the ball. On the other hand, if
$W_{\mathcal S}\subset I^\perp$, the displayed affine intersection is empty.
It follows that every nonempty boundary intersection is the closure of an interior arrangement intersection. Since Looijenga's arrangement blow-up is locally the corresponding arrangement blow-up in this locally trivial model, it introduces no exceptional prime divisor supported only over a cusp.
\end{proof}

\begin{prop}\label{prop:picard-arrangement-blowup}
There are precisely three
$\pi$-exceptional prime divisors $E_{6,9}$, $E_{9,9}$, and $E_{6,9,9}$,
corresponding respectively to the arrangement-intersection types
$(6,9)$, $(9,9)$, and $(6,9,9)$. Moreover,
$\widetilde X_o^{\mathcal K}$ is $\Q$-factorial and
\[
 N^1\bigl(\widetilde X_o^{\mathcal K}/X_o^{j_{\mathcal K}}\bigr)_{\Q}
 =
 \Q[E_{6,9}]
 \oplus\Q[E_{9,9}]
 \oplus\Q[E_{6,9,9}].\]

Consequently,
\[
 \rho\bigl(\widetilde X_o^{\mathcal K}/X_o^{j_{\mathcal K}}\bigr)=3,
 \qquad
 \rho\bigl(\widetilde X_o^{\mathcal K}\bigr)=6,
 \qquad
 \rho\bigl(\widetilde X_o^{\mathcal K}/\widehat X_o^{\mathcal K}\bigr)=5.
\]
\end{prop}

\begin{proof}
Choose a normal neat finite-index subgroup
$\Gamma'\triangleleft\Gamma_o$, sufficiently small so that the arrangement hypersurfaces have no self-intersections. By
\cite[Lemma 5.2 and Theorem 5.7]{Looijenga2003Ball}, their strict transforms on the intermediate compactification are $\mathbb{Q}$-Cartier divisors and
$\widetilde X_o^{\mathcal K}\to X_o^{j_{\mathcal K}}$ is the full arrangement blow-up described in
\cite[Section 2.2]{Looijenga2003Ball}. Its exceptional divisors are indexed by the connected components of arrangement intersections of codimension at least two.

By \cite[Remark 9.5 and Proposition A.11]{HeckmanLooijenga2002}, apart from the mirror traces themselves, the complete list of nonempty intersections of members of
$\mathcal H_o(6)\cup\mathcal H_o(9)$ consists of $(6,9)$, $(9,9)$, and $(6,9,9)$
of codimensions $2,2,3$, respectively. Furthermore, each type is a single
$\Gamma_o$-orbit. Thus, over the open ball quotient, Looijenga's construction first blows up the intersections of type $(6,9,9)$ and then the disjoint strict transforms of the intersections of types $(6,9)$ and $(9,9)$. By Lemma \ref{lem:no-cuspidal-exceptional-divisor}, the extension across the cusp fibres introduces no further exceptional prime.

On the neat cover, a given type may have several exceptional components. However, the finite group
$\Gamma_o/\Gamma'$ acts transitively on the components of each of the above three types, since each type is a single $\Gamma_o$-orbit. The quotient of such a transitive union is irreducible. Hence, on the $\Gamma_o$-quotient, there is exactly one exceptional prime divisor for each type, namely
$E_{6,9}$, $E_{9,9}$, and $E_{6,9,9}$.

Each exceptional component on the neat cover is Cartier: it is created as the exceptional divisor of one of the successive arrangement blow-ups and remains Cartier under the subsequent blow-ups. The finite norm therefore shows that each of
$E_{6,9}$, $E_{9,9}$, and $E_{6,9,9}$ is $\Q$-Cartier on
$\widetilde X_o^{\mathcal K}$.

By Theorem \ref{thm:picard-numbers-corrected}, $X_o^{j_{\mathcal K}}$ is $\Q$-factorial. If $F$ is a nonexceptional prime divisor on
$\widetilde X_o^{\mathcal K}$, then $\pi_*F$ is $\Q$-Cartier, and $F$ differs over $\Q$ from
$\pi^*(\pi_*F)$ by a linear combination of the three exceptional divisors. Hence $F$ is also $\Q$-Cartier. This proves that
$\widetilde X_o^{\mathcal K}$ is $\Q$-factorial and that its relative numerical divisor space is generated by the three exceptional classes.

These classes are linearly independent. Indeed, suppose that
\[
 E=aE_{6,9}+bE_{9,9}+cE_{6,9,9}
 \equiv_\pi0.
\]

Then both $E$ and $-E$ are $\pi$-nef. Applying \cite[Lemma 3.39]{komo} to each of them gives $E\leq0$ and $-E\leq0$, and therefore $E=0$. Thus $a=b=c=0$. It follows that $\rho\bigl(\widetilde X_o^{\mathcal K}/X_o^{j_{\mathcal K}}\bigr)=3$.

Finally, Theorem \ref{thm:picard-numbers-corrected} gives
$\rho(X_o^{j_{\mathcal K}})=3$, while Corollary \ref{cor:Miranda-class-group-corrected} gives
$\rho(\widehat X_o^{\mathcal K})=\rho(\M^M)=1$. The exact sequences of numerical divisor spaces for the two projective birational morphisms therefore yield $\rho(\widetilde X_o^{\mathcal K})=3+3=6$ and  $\rho\bigl(\widetilde X_o^{\mathcal K}/
              \widehat X_o^{\mathcal K}\bigr)=6-1=5$.
\end{proof}

\section{K-moduli wall crossing for quasimaps and the Hassett-Keel-Looijenga program}
\label{sec:K-moduli wall crossing for quasimaps and the HKL program}

In this section, we first compute the K-moduli wall crossing $\mathcal{M}^{\mathrm{Kpsqmaps}}_{12,\frac{1-t}{12},12,\frac{t}{12}}$ for any $t\in [0,1]$, where we regard $\mathcal{M}^{\mathrm{Kpsqmaps}}_{12,\frac{1}{12},12,0}=\mathcal{M}^{\mathrm{Kps}}_{12,\frac{1}{12}}$.
By Theorem \ref{thm:hm1-fromstable}, we see that $\mathcal{M}^{\mathrm{Kpsqmaps}}_{12,\frac{1-t}{12},12,\frac{t}{12}}$ for any $t\in [0,1]$ is dominated by $\mathcal{M}^{\mathrm{stqmaps}}_{12,1,12,1}$ and there exists a canonical morphism $\mathscr{M}\to \mathcal{M}^{\mathrm{Kpsqmaps}}_{12,\frac{1-t}{12},12,\frac{t}{12}}$.
% \Yota{To satisfy the assumption in Theorem \ref{thm:hm1-fromstable}, can we treat only $t=1/2$?} \Masa{I rewrote Thm and so it would be OK.}
Let $\mathscr{M}_t$ be the normalization of the closure of the image and $\mathcal{M}_t$ the good moduli space.
By the definition of $\mathcal{M}$, we see that the canonical morphism $\mathcal{M}\to \mathcal{M}_t$ is an open immersion.
Therefore, there exists a birational projective morphism $\rho_t\colon \mathcal{M}^K\to \mathcal{M}_t$.
Note that $\mathcal{M}_0\cong \mathcal{M}^*$ and $\mathcal{M}_1\cong \mathcal{M}^M$ by definition and Proposition \ref{prop:miranda-k-moduli}.
Set $\mathcal{M}_t(I_{4,4}^*):=\rho_t(\mathcal{M}^K(I^*_{4,4}))$ for example.

\begin{thm}\label{thm:wallcrossing-miranda-to-8points}
Set $ s_1\in(0,\frac{1}{7})$, $s_2\in (\frac{1}{7},\frac{1}{4})$, $s_3\in (\frac{1}{4},\frac{1}{3})$, and $s_4\in(\frac{1}{3},1)$ as arbitrary real numbers.
Then, there exist projective birational morphisms of the projective normal good moduli spaces.
\[
  \xymatrix{
  & \mathcal{M}_{s_1}\ar[ld]^{\phi_1}  \ar[rd]^{\psi_1}& & \mathcal{M}_{s_2}\ar[ld]^{\phi_2}  \ar[rd]^{\psi_2}& & \mathcal{M}_{s_3}\ar[ld]^{\phi_3}  \ar[rd]^{\psi_3}& & \mathcal{M}_{s_4}\ar[ld]^{\phi_4}  \ar[rd]^{\psi_4}& & \\
  \mathcal{M}_0 & & \mathcal{M}_{\frac{1}{7}} & & \mathcal{M}_{\frac{1}{4}}& & \mathcal{M}_{\frac{1}{3}}& & \mathcal{M}_{1}&
  }
\]
Furthermore, $\mathcal{M}_{s_1}$ (resp. $\mathcal{M}_{s_2}$, $\mathcal{M}_{s_3}$, $\mathcal{M}_{s_4}$) is independent of the choice of $s_1$ (resp. $s_2$, $s_3$, $s_4$), $\phi_4$ is isomorphic, $\psi_4$ is a divisorial contraction of $\mathcal{M}_{s_4}(I_{4,4}^*)$, $\psi_3$ is a divisorial contraction of $\mathcal{M}_{s_3}(I_9)$ and  all morphisms in the above diagram except $\psi_3,\psi_4$ are isomorphic in codimension one. 
\end{thm}

\begin{lem}\label{lem:first-step-wallcrossingdescription}
    For any $s\in [0,1]$, the canonical birational map $\mathcal{M}_0\dashrightarrow\mathcal{M}_s$ is a birational contraction.
    Furthermore, $\mathcal{M}^K(I_9)$ is contracted in $\mathcal{M}_s$ for any $s\ge\frac13$ and $\mathcal{M}^K(I_{4,4}^*)$ is contracted in $\mathcal{M}_1$.
\end{lem}

\begin{proof}
    We can use the description of each boundary divisor $\mathcal{M}^K$ given in \cite{HeckmanLooijenga2002}. 
    $\mathscr{M}$ is invariant in the wall-crossing phenomena along $\mathscr{M}_{s}$ for any $s\in [0,1]$.
    Therefore, it suffices to observe how each boundary divisor of $\mathcal{M}^K$ is mapped to $\mathcal{M}_s$.
    We note that the canonical morphism $\mathcal{M}^K\to\mathcal{M}_s$ constructed as Theorem \ref{thm:hm1-fromstable} cannot lift up to $\mathscr{M}^K\to\mathscr{M}_s$ in general.
However, by \cite[Remark 3.5 and Theorem 4.20]{HM1}, $\mathcal{M}_s$ is the normalization of the closure of the image of $\mathcal{M}$ in ${M}^{\mathrm{bpcyqmaps}}_{12,\frac{1-s}{6},12,\frac{s}{6}}$ and  there exists a canonical morphism $\xi_s\colon\mathscr{M}^K\to\mathcal{M}^{\mathrm{bpcyqmaps}}_{12,\frac{1-s}{6},12,\frac{s}{6}}$ by \cite[Theorems 5.10 and 5.11]{HM1}.
Furthermore, if an object $q\colon(\mathbb{P}^1,\frac{(1-s)}{12}D)\to [\mathbb{A}^2/\mathbb{G}_m]\in \mathscr{M}_s(\mathbb{K})$ is K-stable, then it is easy to check that $q\colon(\mathbb{P}^1,\frac{(1-s)}{6}D)\to [\mathbb{A}^2/\mathbb{G}_m]$ with weight $\frac{s}{6}$ is klt.
Note that the open locus $\mathscr{M}^{\mathrm{st}}_s$ of $\mathscr{M}_s$ parameterizing the K-stable quasimaps is saturated and the canonical morphism $\xi'_s\colon (\xi_s(\mathscr{M}^K))^\nu\to \mathscr{M}_s$ is isomorphic over this locus, where $(\xi_s(\mathscr{M}^K))^\nu$ is the normalization of the scheme-theoretic image of $\xi_s$.
Furthermore, $\xi'_s$ induces an isomorphism of the good moduli space of $(\xi_s(\mathscr{M}^K))^\nu$ and $\mathcal{M}_s$ by Theorem \ref{thm:wall-crossing-recall}.
Therefore, to calculate the dimension of each stratum, it suffices to see that the image of $\mathcal{M}^K(\bullet)$ in each $\mathcal{M}_s$.

To show Lemma \ref{lem:first-step-wallcrossingdescription}, it suffices to check that the image of $\mathcal{M}^K(\bullet)$ is not a divisor in each $\mathcal{M}_s$ unless $\bullet=I_2,II,I_9$ or $I_{4,4}^*$. 
Note that the case where $s=0$ or $1$ has been treated by \cite[Corollary 3.4 and Table]{HeckmanLooijenga2002}.
We only deal with $\mathcal{M}_{s}(I_9)$ in this proof and can also prove the assertion in the same argument.
%First, we deal with $\mathcal{M}_{s}(I_9)$.
By the description of \cite{HeckmanLooijenga2002}, we see that $\mathscr{M}^K(I_9)$ generically parameterizes $J\colon (C,D)\to \mathbb{P}^1$ such that $J$ is a stable map, $D$ is a reduced divisor of degree twelve, $D\subset J^{-1}(\infty)$, and $J$ is mapped to a closed point of $\mathcal{M}^M(I_9)$.
Furthermore, for any general $J$, we see that $C$ consists of exactly two irreducible components $C_1$ and $C_2$ that are rational curves intersecting at one point transversally.
$D$ consists of twelve points such that each of them is not contained in $\mathrm{Supp}(\mathrm{div}(f_0))$, where $J$ is defined by $[f_0:f_1]$ for some $f_0,f_1\in H^0(C,J^*\mathcal{O}(1))$, and we may assume that exactly nine of them lie in $C_1$.
On the other hand, $J|_{C_2}$ is of degree twelve but $J|_{C_1}$ is constant whose image is only $\infty=[1:0]$.
From now on, we will use the terminology of \cite[Definition 2.10]{HM1}, the contraction of irreducible components.
Note that $C_1$ and $C_2$ are rational tails of $C$.
Consider the image of the point defined by the above $J$ under $\rho_s\colon\mathcal{M}^K\to \mathcal{M}_s$ for any $0<s<1$.
It is not hard to check that the contraction of $C_1$ yields a log quasimap $q_2\colon (\mathbb{P}^1,D_2)\to [\mathbb{A}^2/\mathbb{G}_m]$ of degree twelve, where $D_2$ is the strict transform of $D$, and $q_2\colon (\mathbb{P}^1,\frac{1-s}{12}D_2)\to [\mathbb{A}^2/\mathbb{G}_m]$ with weight $\frac{s}{12}$ is K-stable for any $s>\frac13$. 
Indeed, we can see that $q_2$ represents the image in $\mathcal{M}_s$ of the point defined by $J$ above for any $s>\frac13$ chasing the wall crossings from $\mathcal{M}^{\mathrm{stqmaps}}_{12,1,12,1}$ to $\mathcal{M}^{\mathrm{bpcyqmaps}}_{12,\frac{1-s}{6},12,\frac{s}{6}}$ for any $s\in [\frac{1}{3},1]$ and using Lemma \ref{lem:useful-criterion}.
On the other hand, the contraction of $C_2$ yields a log quasimap $q_1\colon (\mathbb{P}^1,D_1)\to [\mathbb{A}^2/\mathbb{G}_m]$ of degree twelve, where $D_1$ is the strict transform of $D$, and $q_1\colon (\mathbb{P}^1,\frac{1-s}{12}D_1)\to [\mathbb{A}^2/\mathbb{G}_m]$ with weight $\frac{s}{12}$ is K-stable for $s<\frac13$.
Note that $q_1$ represents the image in $\mathcal{M}_s$ of the point defined by $J$ above for any $s<\frac13$.
We also note that  $D_1$ consists of reduced nine points and a point $p_1\in \mathbb{P}^1$ of multiplicity three and if $q_1$ is defined by $[g_0:g_1]$, where $g_0$ is the degeneration of the cube of quartics, then $\mathrm{div}(g_0)=12\{p_1\}$.
Furthermore, $D_2$ consists of reduced three points and a point of multiplicity nine and $q_2$ is defined by $[g_0'^3:g_0'^3+g_1'^2]$ for some $g_0'\in H^0(\mathbb{P}^1,\mathcal{O}(4))$ and $g_1'\in H^0(\mathbb{P}^1,\mathcal{O}(6))$ such that $\mathrm{Supp}(D_2)\cap\mathrm{Supp}(\mathrm{div}(g'_0))=\emptyset$.
Since $\mathcal{M}_1(I_9)$ consists of one point, there is only one isomorphic class of $q_2$ and hence $\mathcal{M}_s(I_9)$ consists of one closed point for any $s\in [\frac13,1]$.
It is not hard to check that the unique closed point of $\mathcal{M}_{\frac13}(I_9)$, which is represented by the K-polystable constant quasimap $ (\mathbb{P}^1,\frac{1}{6}\{0\}+\frac{1}{2}\{\infty\})\to [\{(x,0)|x\in\mathbb{C}\}/\mathbb{G}_{m}]$ with weight $\frac{1}{36}$ and the fixed part $12\{0\}$ by using the K-moduli wall crossing \cite[Proposition 2.40 (3) and Theorem 3.7]{HM1}.
Therefore, $\mathcal{M}^K(I_9)$ is contracted to one closed point of $\mathcal{M}_s$ for any $s\ge\frac{1}{3}$.
%Furthermore, it is not hard to see that $f_0^3$ does not vanish along $\overline{\mathcal{C}_0}$ but $f_0^3+f_1^2$ does, and $\mathrm{div}(f_0^3)$ consists of one closed point $p'$ of multiplicity twelve in $\overline{\mathcal{C}_0}$.
%Since $s<\frac{1}{3}$ and $Q_0$ is K-semistable, we see that there exists at most one point $p''\in \overline{\mathcal{C}_0}$ such that $\mathrm{mult}_{p''}(\overline{\mathcal{D}_0})>6$ and if such a point exists, then we see that $\mathrm{mult}_{p''}(\overline{\mathcal{D}_0})\le 8$ and $p''\ne p'$ by Lemma \ref{lem:useful-criterion}.
%First, suppose that $p''$ does not exist.
%Then $Q_0\colon (\overline{\mathcal{C}}_0,\frac{(1-s')}{12}\overline{\mathcal{D}}_0)\to[\mathbb{A}^2/\mathbb{G}_{m}]$ with weight $\frac{s'}{12}$ for any $0\le s'\le\frac{1}{3}$ is K-semistable by Lemma \ref{lem:useful-criterion}.
On the other hand, by the description of \cite[Section 3.3]{HeckmanLooijenga2002}, we see that $\mathcal{M}^K(I_9)$ is mapped to $\mathcal{M}_{s}(I_9)$ generically isomorphic for any $s<\frac13$.
This means that $\mathcal{M}_{s}$ possibly changes only around $s=0,\frac{1}{3}$ and $1$ along the generic point of $\mathcal{M}_{s}(I_9)$ and $\psi_3$ contracts $\mathcal{M}_{s_3}(I_9)$.

By the above argument, we see that $\mathcal{M}^K(I_9)$ is not contracted in $\mathcal{M}_s$ for any $s<\frac13$ but contracted to one closed point for any $s\ge\frac13$.
Moreover, $\mathcal{M}_{s}(I_7)\subset\mathcal{M}_{s}(I_9)$ for any $s\le\frac{1}{7}$ and $\mathcal{M}_{s}(I_8')\subset\mathcal{M}_{s}(I_9)$ for any $s\le\frac14$.
When $s=\frac17,\frac14$, these loci are contracted but $\mathcal{M}_s$ is not changed around the generic point of $\mathcal{M}_s(I_9)$ by a similar argument.
%For any other $s$, there is no birational change around the generic point of$\mathcal{M}_s(I_9)$.
We also obtain the following observation in the argument similar to above:
\begin{itemize}
    \item When $s=\frac{1}{7}$, $\mathcal{M}_{s}(I_7)$ is contracted to one closed point and for any sufficiently small $\epsilon>0$, $\mathcal{M}_{\frac{1}{7}-\epsilon}\dashrightarrow\mathcal{M}_{\frac{1}{7}+\epsilon}$ is a flip along $\mathcal{M}_{s}(I_7)$. For any other $0<s<1$, there is no birational change around the generic point of $\mathcal{M}_{s}(I_7)$.
    Note also that $\mathcal{M}_s(I_7)$ is birational to $\mathcal{M}_1(I_7)$ for any $s>\frac17$ and hence $\mathrm{dim}\mathcal{M}_s(I_7)=2$.
    We note that for any $0<s<\frac17$, $\mathcal{M}_s(I_7)$ generically parameterizes $q\colon(\mathbb{P}^1,\frac{1-s}{12}(\sum_{i=1}^{7}p_i+5p_8))\to [\mathbb{A}^2/\mathbb{G}_m]$ with weight $\frac{s}{12}$, where $q$ is defined by $[g:0]$ for some $g\in H^0(\mathbb{P}^1,\mathcal{O}(12))$, where $g$ is a degeneration of the cube of quartics, such that $\mathrm{div}(g)=12p_8$ and $p_1,\ldots,p_8$ are general.
    For any $\frac17<s<1$, $\mathcal{M}_s(I_7)$ generically parameterizes $q\colon(\mathbb{P}^1,\frac{1-s}{12}(\sum_{i=1}^{5}p'_i+7p'_6))\to [\mathbb{A}^2/\mathbb{G}_m]$ with weight $\frac{s}{12}$, where $q$ is defined by $[g_0^3:g_0^3+g_1^2]$ for some $g_0\in H^0(\mathbb{P}^1,\mathcal{O}(4))$ and $g_1\in H^0(\mathbb{P}^1,\mathcal{O}(6))$ such that $\mathrm{div}(g_0^3+g_1^2)=\sum_{i=1}^{5}p'_i+7p'_6$, $\mathrm{Supp}(\mathrm{div}(g_0))\cap\mathrm{Supp}(\mathrm{div}(g_0^3+g_1^2))=\emptyset$ and $p'_1,\ldots,p'_6$ are general.
    For $s=\frac{1}{7}$, $\mathcal{M}_{\frac17}(I_7)$ consists of one point, which is represented by the K-polystable constant quasimap $ (\mathbb{P}^1,\frac{5}{14}\{0\}+\frac{1}{2}\{\infty\})\to [\{(x,0)|x\in\mathbb{C}\}/\mathbb{G}_{m}]$ with weight $\frac{1}{84}$ and the fixed part $12\{0\}$ by using the K-moduli wall crossing \cite{HM1}.
    \item For any $0\le s<1$, there is no birational change along the generic points of $\mathcal{M}_s(I_{k,k'}^*)$ for any $1\le k,k'\le4$.
    When $s=1$, all $\mathcal{M}_s(I_{k,k'}^*)$ are contracted to the same closed point.
    If $k'<k''$, then $\mathcal{M}_s(I_{k,k'}^*)\subsetneq \mathcal{M}_s(I_{k,k''}^*)$ for any $0\le s<1$. Thus, $\mathrm{dim}\mathcal{M}_s(I_{k,k'}^*) = k+k'-1$ and in particular $\mathrm{dim}\mathcal{M}_s(I_{4,4}^*)=7$. Note that $\mathcal{M}_s(I_{4,4}^*)$ is mapped and generically isomorphic to $\mathcal{M}_0(I_{4,4}^*)$ for any $s<1$.
    As \cite{HeckmanLooijenga2002}, it is not hard to check that the general members $q\colon (\mathbb{P}^1,\frac{(1-s)}{12}D)\to [\mathbb{A}^2/\mathbb{G}_m]$ with weight $\frac{s}{12}$ of $\mathcal{M}_s(I_{k,k'}^*)$ can be exhibited as follows: $D$ has $k+k'$ points of multiplicity one, one point (say $0$) of multiplicity of $6-k$ and one point (say $\infty$) of multiplicity of $6-k'$.
    Furthermore, if $q$ is defined by $[g:0]$ for some $g_0\in H^0(\mathbb{P}^1,\mathcal{O}_{\mathbb{P}^1}(12))$, where $\mathrm{div}(g)$ is a degeneration of the cube of quartics, then $\mathrm{div}(g)=6\{0\}+6\{\infty\}$.
    \item When $s=\frac{1}{4}$, $\mathcal{M}_{s}(I_8')$ and $\mathcal{M}_{s}(I_8'')$ are contracted to closed points and for any sufficiently small $\epsilon>0$, $\mathcal{M}_{\frac{1}{4}-\epsilon}\dashrightarrow\mathcal{M}_{\frac{1}{4}+\epsilon}$ consists of two flips along the generic points of $\mathcal{M}_{s}(I_8')$ and $\mathcal{M}_{s}(I_8'')$.
    Thus, $\mathrm{dim}\mathcal{M}_{s}(I_7)=5$ and $\mathrm{dim}\mathcal{M}_s(I_8')=6$ for any $s<\frac{1}{7}$ and for any $s<\frac14$, respectively.
    Similarly, we can see that $\mathcal{M}_{s}(I_8'')\subsetneq\mathcal{M}_s(I_{4,4}^*)$ and hence $\mathrm{dim}\mathcal{M}_{s}(I_8'')=6$ for any $s<\frac14$.
    Note that $\mathcal{M}_s(I_8')$ and $\mathcal{M}_{s}(I_8'')$ are mapped and generically isomorphic to $\mathcal{M}_1(I_8')$ and $\mathcal{M}_{1}(I_8'')$ for any $s>\frac14$.
    For $\mathcal{M}_s(I_8')$ for any $0<s<\frac14$, we note that $\mathcal{M}_s(I_8')$ generically parameterizes $q\colon(\mathbb{P}^1,\frac{1-s}{12}(\sum_{i=1}^{8}p_i+4p_9))\to [\mathbb{A}^2/\mathbb{G}_m]$ with weight $\frac{s}{12}$, where $q$ is defined by $[g:0]$ for some $g\in H^0(\mathbb{P}^1,\mathcal{O}(12))$, where $g$ is a degeneration of the cube of quartics, such that $\mathrm{div}(g)=12p_9$ and $p_1,\ldots,p_9$ are general.
    For any $\frac14<s<1$, $\mathcal{M}_s(I_8')$ generically parameterizes $q\colon(\mathbb{P}^1,\frac{1-s}{12}(\sum_{i=1}^{4}p'_i+8p'_5))\to [\mathbb{A}^2/\mathbb{G}_m]$ with weight $\frac{s}{12}$, where $q$ is defined by $[g_0^3:g_0^3+g_1^2]$ for some $g_0\in H^0(\mathbb{P}^1,\mathcal{O}(4))$ and $g_1\in H^0(\mathbb{P}^1,\mathcal{O}(6))$ such that $\mathrm{div}(g_0^3+g_1^2)=\sum_{i=1}^{4}p'_i+8p'_5$, $\mathrm{Supp}(\mathrm{div}(g_0))\cap\mathrm{Supp}(\mathrm{div}(g_0^3+g_1^2))=\emptyset$ and $p'_1,\ldots,p'_5$ are general.
    For $s=\frac{1}{4}$, $\mathcal{M}_{\frac14}(I'_8)$ consists of one point, which is represented by the K-polystable constant quasimap $ (\mathbb{P}^1,\frac{1}{4}\{0\}+\frac{1}{2}\{\infty\})\to [\{(x,0)|x\in\mathbb{C}\}/\mathbb{G}_{m}]$ with weight $\frac{1}{48}$ and the fixed part $12\{0\}$ by using the K-moduli wall crossing \cite{HM1}.
    The same description we can give for $\mathcal{M}_s(I_8'')$.
    \item Around the generic points of $\mathcal{M}_s(II),\mathcal{M}_s(III)$ and $\mathcal{M}_s(IV)$, there is no birational change for any $s\in (0,1)$.
In particular, $\mathrm{dim}\mathcal{M}_s(II)=7$, $\mathrm{dim}\mathcal{M}_s(III)=6$ and $\mathrm{dim}\mathcal{M}_s(IV)=5$ because $\mathcal{M}_s(IV)\subsetneq \mathcal{M}_s(III)\subsetneq\mathcal{M}_s(II)$.
\item Around the generic points of $\mathcal{M}_s(I_{k})$ for any $2\le k\le 6$, there is no birational change for any $s\in (0,1)$.
In particular, $\mathrm{dim}\mathcal{M}_s(I_k)=9-k$ because $\mathcal{M}_s(I_k)\subsetneq \mathcal{M}_s(I_{k'})$ for any $2\le k'<k\le6$.
Note that $\mathcal{M}_s(I_k)$ is mapped and generically isomorphic to $\mathcal{M}_1(I_k)$ for any $s>0$.
\item Around the generic point of $\mathcal{M}_s(I_{0}^*)$, there is no birational change for any $s\in (0,1)$.
By a standard calculation, we see that $\mathrm{dim}\mathcal{M}_s(I_{0}^*)=1$ and $\mathcal{M}_s(I_{0}^*)$ is generically isomorphic to $\mathcal{M}_1(I_{0}^*)$.
\end{itemize}
Combining the above observations, we obtain the following table on the dimension of each stratum.
\newpage

\begin{table}[htbp]
\centering
\caption{Dimension of each boundary stratum for any $s\in[0,1]$}
\[
\begin{array}{c|cccccccc}
\mathrm{dim}\mathcal{M}_s(F)
&s=0&0<s<\frac{1}{7}&
s=\frac17
&\frac17<s<\frac14&
s=\frac14&\frac14<s<\frac13
&
\frac{1}{3}\le s<1&
s=1
\\
\hline
I_{2\le k\le5}
&
9-k
&
9-k
&
9-k
&
9-k
&
9-k
&
9-k
&9-k
&9-k
\\
I_6
&
0
&
3
&
3
&
3&
3&
3&
3&
3
\\
I_7
&
5
&
5
&
0
&
2&
2&
2&
2&
2
\\
I'_8,I''_8
&
6
&
6
&
6
&
6&0&1&1&1
\\
I_9
&
7
&
7
&
7
&
7&7&7&0&0
\\
I^*_{k,k'}
&
k+k'-1
&
k+k'-1
&
k+k'-1
&k+k'-1&
k+k'-1&
k+k'-1&
k+k'-1&
0
\\
II
&
7
&
7
&
7
&7&
7&
7&
7&
7
\\
III
&
6
&
6
&
6
&6&
6&
6&
6&
6
\\
IV
&
5
&
5
&
5
&5&
5&
5&
5&
5
\\
I_0^*
&
0
&
1
&
1
&1&
1&
1&
1&
1
\end{array}
\]
\end{table}

By the above table, we obtain the assertion combining with \cite[Table after Corollary 3.4]{HeckmanLooijenga2002}.
\end{proof}

By \cite[II, Proposition 6.5]{Ha}, we see that $\rho(\mathcal{M}_{s_1})\le 4$.
Therefore, we can show that $\Lambda_{\mathrm{Hodge},s_1}$ can be exhibited as a linear combination of $\mathcal{M}_{s_1}(I_9)$,  $\mathcal{M}_{s_1}(I_{4,4}^*)$, $\mathcal{M}_{s_1}(I_{2})$ and $\mathcal{M}_{s_1}(II)$. 

\begin{lem}\label{lem:putting-coeffs-Hodge}
  For any $s\in [0,1]$, $\rho(\mathcal{M}_{s})\le3$ and if $s\ge\frac{1}{3}$, then $\rho(\mathcal{M}_s)\le2$ and $\rho(\mathcal{M}_1)=1$.
  Furthermore, for any $s_1\in(0,\frac{1}{3})$, $\rho(\mathcal{M}_{s_1})=3$ if and only if $\mathcal{M}_{s_1}$ is $\Q$-factorial and for any $s'\in[\frac{1}{3},1)$, $\rho(\mathcal{M}_{s'})=2$ if and only if $\mathcal{M}_{s'}$ is $\Q$-factorial.
In particular, $\Lambda_{\mathrm{Hodge},s_1}$ can be exhibited as a $\R$-linear combination of $\mathcal{M}_{s_1}(I_9)$,  $\mathcal{M}_{s_1}(I_{4,4}^*)$, and $\mathcal{M}_{s_1}(I_{2})$ uniquely.
\end{lem}
\begin{proof}
Since $\mathcal{M}_{s_1}$ is isomorphic to $X_o^{j_{\mathcal K}}$ in codimension one and $\mathcal{M}_s$ is a birational contraction of $\mathcal{M}_{s_1}$ by Lemma \ref{lem:first-step-wallcrossingdescription}, $\mathrm{Pic}(\mathcal{M}_{s_1})_{\Q}\subset \mathrm{Cl}(\mathcal{M}_{s_1})_{\Q}\cong\mathrm{Cl}(X_o^{j_{\mathcal K}})_{\Q}$ and the right hand side has rank three. 
Therefore, the assertion for $\mathcal{M}_{s_1}$ follows from Corollary \ref{cor:picard-semitoric-exhibition} and Lemma \ref{lem:first-step-wallcrossingdescription}.
On the other hand, there exists a canonical surjection $\mathrm{Cl}(\mathcal{M}_{s_1})_{\Q}\to \mathrm{Cl}(\mathcal{M}_{s})_{\Q}$.  
Furthermore, the contraction of $\mathcal{M}_{s_1}(I_9)$ shows that the above surjection is not an injection.
Hence, $\mathrm{rank}\,\mathrm{Cl}(\mathcal{M}_{s})_{\Q}\le2$.
Observing that the canonical surjection $\mathrm{Cl}(\mathcal{M}_{s})_{\Q}\to \mathrm{Cl}(\mathcal{M}_{1})_{\Q}$ is not an injection and $\mathrm{rank}\,\mathrm{Cl}(\mathcal{M}_{1})_{\Q}=1$ by Corollary \ref{cor:Miranda-class-group-corrected}, we see that $\mathrm{rank}\,\mathrm{Cl}(\mathcal{M}_{s})_{\Q}=2$ for any $s\in[\frac13,1) $. 
Therefore, we obtain the assertion for $\mathcal{M}_{s}$ for any $s\in[\frac13,1) $.
We complete the proof.
\end{proof}

\begin{proof}[Proof of Theorem \ref{thm:wallcrossing-miranda-to-8points}]
By Lemma \ref{lem:first-step-wallcrossingdescription}, it suffices to show that only $s=0,\frac17,\frac14,\frac13,1$ 
can be walls, the wall $s=\frac{1}{3}$ (resp. $s=1$) corresponds to the divisorial contraction of $\mathcal{M}_s(I_{9})$ (resp. $\mathcal{M}_s(I_{4,4}^*)$) and $\phi_4$ is an isomorphism. Indeed, it suffices to consider $0<s<1$ for the first assertion.
For any sufficiently small $0<\epsilon<1$, let $q\colon (\mathbb{P}^1,\frac{1-\epsilon}{12}D)\to [\mathbb{A}^2/\mathbb{G}_m]$ with weight $\frac{\epsilon}{12}$ define a K-polystable object of $\mathcal{M}^{\mathrm{Kss}}_{12,\frac{1-\epsilon}{12},12,\frac{\epsilon}{12}}$ and a point of $\mathcal{M}_{\epsilon}$. 
Note that $(\mathbb{P}^1,\frac{1}{12}D)$ is K-semistable and $\mathrm{ord}_p(D)\le6$ for any $p\in\mathbb{P}^1$.
Since $q$ is a degeneration of objects of $\mathscr{M}$, there exist $g_0\in H^0(\mathbb{P}^1,\mathcal{O}(4))$ and $g_1\in H^0(\mathbb{P}^1,\mathcal{O}(6))$ such that $q$ is defined by $[g_0^3:g_0^3+g_1^2]$.
Here, we note that $g_0^3\not\equiv0$ or $g_0^3+g_1^2\not\equiv0$ and $g_0^3+g_1^2\not\equiv0$ implies $\mathrm{div}(g_0^3+g_1^2)=D$. 
By the definition of K-stability of log Fano quasimaps, we see that if $\mathrm{div}(g_0^3+g_1^2)=D$, changing the weight and the coefficients, $q\colon (\mathbb{P}^1,\frac{1-s}{12}D)\to [\mathbb{A}^2/\mathbb{G}_m]$ with weight $\frac{s}{12}$ is K-semistable for any $0<s<1$.
Thus, we see that $q\colon (\mathbb{P}^1,\frac{1-s}{12}D)\to [\mathbb{A}^2/\mathbb{G}_m]$ with weight $\frac{s}{12}$ is K-polystable for any $0<s<1$ in this case by Lemma \ref{lem:K-ps-remaining}.
Therefore, we may assume that $g_0^3+g_1^2\equiv0$.
In this case, we can apply Lemma \ref{lem:useful-criterion} and K-semistability of $q\colon (\mathbb{P}^1,\frac{1-s}{12}D)\to [\mathbb{A}^2/\mathbb{G}_m]$ with weight $\frac{s}{12}$ is defined as the condition that for any $p\in \mathbb{P}^1$, $\mathrm{ord}_p((1-s)D+3s\mathrm{div}(g_0))\le6$.
% \Yota{We cannot apply Lemma 3.12 here because we have to assume $D=\div(f_2)$. The inequality we want here can be lead by directly from Definition 3.10.}
% \Masa{I rewrote Lemma 3.12 and now it is applicable.}
Note that $3\mathrm{div}(g_0)=2\mathrm{div}(g_1)$ and hence $3\mathrm{ord}_p(\mathrm{div}(g_0))\in6\mathbb{Z}$.
Thus, above $q$ is always K-semistable for any $0<s<1$ unless $\mathrm{ord}_p(\mathrm{div}(g_0))=4$.
We may assume that $\mathrm{ord}_p(\mathrm{div}(g_0))=4$ by Lemma \ref{lem:K-ps-remaining} and if $\mathrm{ord}_p(D)=6$ then above $q$ is always K-unstable for any $0<s<1$.
Take $s_0$ as the unique number such that $(1-s_0)\mathrm{ord}_p(D)+12s_0=6$ when $\mathrm{ord}_p(D)<6$ and $\mathrm{ord}_p(\mathrm{div}(g_0))=4$.
It suffices to check that $s_0$ only takes $\frac17,\frac14,\frac13$ and there is no K-moduli wall-crossing around the point defined by $q$ for any $s>s_0$. 
By \cite[Theorem 3.6]{HM1} and \cite[Theorem A]{AHLH}, there exists another K-polystable quasimap $q'\colon (\mathbb{P}^1,\frac{(1-s')}{12}D')\to [\mathbb{A}^2/\mathbb{G}_m]$ with weight $\frac{s'}{12}$, where $s'=s_0+\delta$ for any sufficiently small $\delta>0$ but $q'\colon (\mathbb{P}^1,\frac{(1-s_0)}{12}D')\to [\mathbb{A}^2/\mathbb{G}_m]$ with weight $\frac{s_0}{12}$ is S-equivalent to $q$.
Suppose that $q'$ is defined by $[g_0'^3:g_0'^3+g'^2_1]$.
%Since $q$ and $q'$ are S-equivalent, it is not hard to check that $\max_{p'\in\mathbb{P}^1}\mathrm{ord}_{p'}((1-s_0)D'+3s_0\mathrm{div}(g'_0))\le6$ by the inversion of adjunction for log canonicity \cite{kawakita2007inversion} because the K-polystable degeneration of $q\colon (\mathbb{P}^1,\frac{1-s_0}{12}D)\to [\mathbb{A}^2/\mathbb{G}_m]$ with weight $\frac{s_0}{12}$ is induced by $p$ \Masa{We have to explain about it}.
If $g_0'^3+g'^2_1\equiv0$, then Lemma \ref{lem:useful-criterion} implies that the K-semistability for $q'\colon (\mathbb{P}^1,\frac{(1-s')}{12}D')\to [\mathbb{A}^2/\mathbb{G}_m]$ with weight $\frac{s'}{12}$ is equivalent to the condition that $\max_{p'\in\mathbb{P}^1}\mathrm{ord}_{p'}((1-s')D'+3s'\mathrm{div}(g'_0))\le6$ for any $s'\in (0,1)$.
In this case, we also note that $3\mathrm{div}(g'_0)=2\mathrm{div}(g'_1)$ and $\max_{p'\in\mathbb{P}^1}\mathrm{ord}_{p'}(\mathrm{div}(g'_0))\le2$ implies that $q'\colon (\mathbb{P}^1,\frac{(1-s')}{12}D')\to [\mathbb{A}^2/\mathbb{G}_m]$ with weight $\frac{s'}{12}$ is K-semistable for any $s_0<s'<1$. 
Moreover, if $s_0>\frac{1}{3}$, then $s_0=\frac12, \frac{5}{11},\frac25$. In these cases, there exists a point $p''\in\mathbb{P}^1$ such that $\mathrm{ord}_{p''}(D')\ge10$ and $\mathrm{ord}_{p''}(\mathrm{Bs}(q'))=0$. Because the point defined by $q'$ is mapped to the same image as $\mathcal{M}_{s}(I_{4,4}^*)$ in $\mathcal{M}_1$ and hence $q'$ lies in the closure of $\mathcal{M}_{s}(I_{4,4}^*)$, where $s=1-\delta$ for any sufficiently small $\delta>0$. 
Indeed, $\rho(\mathcal{M}_{s})\le2$ in the above case by Lemma \ref{lem:first-step-wallcrossingdescription} and $\mathcal{M}_s\to \mathcal{M}_1$ is the divisorial contraction of $\mathcal{M}_{s}(I_{4,4}^*)$ by \cite[Corollary 2.63]{komo} since $\mathcal{M}_1$ is $\Q$-factorial by Proposition \ref{prop:miranda-k-moduli} and Corollary \ref{cor:Miranda-class-group-corrected}.
This contradicts the description of general members of $\mathcal{M}_{0}(I_{4,4}^*)$ given in \cite{HeckmanLooijenga2002} (see the proof of Lemma \ref{lem:first-step-wallcrossingdescription} and the description of the general family given there).
As a result, we see that $s_0>\frac{1}{3}$ never happens in this case.
Thus, we may assume that $g_0'^3+g'^2_1\not\equiv0$.
Then, $q'\colon (\mathbb{P}^1,\frac{(1-s')}{12}D')\to [\mathbb{A}^2/\mathbb{G}_m]$ with weight $\frac{s'}{12}$ is K-semistable for any $s_0<s'<1$.
Assigning $s'=1$, $q'\colon \mathbb{P}^1\to [\mathbb{A}^2/\mathbb{G}_m]$ with weight $\frac{1}{12}$ is an object of the moduli space of Miranda $\mathcal{M}_1$.
We can estimate the form of $q'$ as follows.
Let $Q\colon (\mathcal{C},\frac{(1-\epsilon)}{12}\mathcal{D})\to [\mathbb{A}^2_S/\mathbb{G}_{m,S}]$ with weight $\frac{\epsilon}{12}$ for any sufficiently small $\epsilon>0$ be a family of quasimaps such that $S$ is a smooth affine curve such that $Q_0=q$ for some closed point $0\in S$.
Take $Q$ in an appropriate way, the morphism $h\colon S\to \mathcal{M}_{\epsilon}$ can be birationally equivalent to $h'\colon S\to \mathcal{M}_{1-\epsilon}$ such that $h'(0)$ corresponds to the point defined by $q'$.
Replacing $S$ with its finite cover, without changing the general fiber of $Q$ over $S$, we can replace $Q$ with $Q'\colon (\mathcal{C}',\frac{\epsilon}{12}\mathcal{D}')\to [\mathbb{A}^2_S/\mathbb{G}_{m,S}]$ with weight $\frac{(1-\epsilon)}{12}$ such that $Q'_0=q'$ (cf. \cite[Theorem A.8]{AHLH}).
Furthermore, $\mathcal{C}'\dashrightarrow\mathcal{C}$ is established by fundamental transforms finitely many times (cf. \cite[Proof of Theorem 3.31]{HH2}).
Taking this into the account, we see that there exists a closed point $p'''\in \mathbb{P}^1$ such that $\mathrm{ord}_{p'''}(D')\ge7,8,9,10,11,12$ when $s_0=\frac{1}{7},\frac14,\frac13,\frac25,\frac{5}{11},\frac12$, respectively.
Here, we note that $\mathrm{ord}_{p'''}(D')>7,8,9,10,11,12$ when $s_0=\frac{1}{7},\frac14,\frac13,\frac25,\frac{5}{11},\frac12$ never happens.
Indeed, assuming so, then we need at least two fundamental transforms in order to obtain $Q'$ from $Q$ and $\min\{\mathrm{ord}_{p'''}(g_0^3),\mathrm{ord}_{p'''}(g_1^2)\}\ge6$. It contradicts the K-semistability of $q'$ (cf. Lemma \ref{lem:useful-criterion}).
Thus, $\mathrm{ord}_{p'''}(D')=7,8,9,10,11,12$ when $s_0=\frac{1}{7},\frac14,\frac13,\frac25,\frac{5}{11},\frac12$, respectively.
When $\mathrm{ord}_{p'''}(D')=7,8,9$, it is not hard to check that the point defined by $q'$ lies in $\mathcal{M}_{1-\epsilon}(I_7)$, $\mathcal{M}_{1-\epsilon}(I_8')\cup \mathcal{M}_{1-\epsilon}(I_8'')$, $\mathcal{M}_{1-\epsilon}(I_9)$, respectively for any sufficiently small $\epsilon>0$ and $\mathrm{ord}_{p'''}(D')=10,11,12$ never happens.
Indeed, since $q'$ can be regarded an object of the moduli stack of Miranda $[\PGL(2)\!\setminus \!Y^{\mathrm{ss}}_{\mathrm{Mir}}]$, we can take a GIT semistable rational Weierstrass fibration $W$ over $\mathbb{P}^1$ corresponding to $q'$ (cf. \cite[Theorem 7.2]{Miranda1981}).
Since $W$ is GIT semistable, $\mathrm{ord}_{p'''}(D')\ge7$ means that the fiber of $W$ over $p'''$ is of type $I_{\mathrm{ord}_{p'''}(D')}$ or $I^*_{\mathrm{ord}_{p'''}(D')-6}$ by \cite[Table (6.1), Theorem 6.2 and Proposition 8.2]{Miranda1981}.
If the fiber is of type $I^*_{\mathrm{ord}_{p'''}(D')-6}$, then $\mathrm{ord}_{p'''}(g'_0)=2$ or $\mathrm{ord}_{p'''}(g'_1)=3$, and this contradicts the K-semistability of $q'\colon (\mathbb{P}^1,\frac{(1-s')}{12}D')\to [\mathbb{A}^2/\mathbb{G}_m]$ with weight $\frac{s'}{12}$ for any $s_0<s'<1$ by Lemma \ref{lem:useful-criterion}.
On the other hand, letting $\tilde{W}$ be the minimal resolution, we see that $\rho(\tilde{W})=10$ implies that $\mathrm{ord}_{p'''}(D')\le9$.

Finally, we show that the wall $s=\frac{1}{3}$ (resp. $s=1$) corresponds to the divisorial contraction of $\mathcal{M}_s(I_{9})$ (resp. $\mathcal{M}_s(I_{4,4}^*)$) and $\phi_4$ is an isomorphism.
For $s=1$ case, since $\mathcal{M}_1$ is $\Q$-factorial and $\rho(\mathcal{M}_1)=1$, $\rho(\mathcal{M}_{s_4})\ge2$ since $\psi_4$ at least contracts $\mathcal{M}_{s_4}(I_{4,4}^*)$, where $s_4\in(\frac{1}{3},1)$.
By Lemma \ref{lem:putting-coeffs-Hodge}, we see that $\rho(\mathcal{M}_{s_4})=2$ and $\mathcal{M}_{s_4}$ is $\Q$-factorial again.
By \cite[Corollary 2.63]{komo} and the fact that $\mathcal{M}_{s_4}(I_{4,4}^*)$ is only one $\psi_4$-exceptional divisor, we see that $\psi_4$ is a divisorial contraction of $\mathcal{M}_{s_4}(I_{4,4}^*)$.
By the description we obtained in the first paragraph of this proof, we see that $\phi_4$ is an isomorphism between the normalization of the good moduli spaces because the only K-polystable members are replaced at one point of $\mathcal{M}_{s_4}(I_9)$.
Moreover $\psi_3$ contracts $\mathcal{M}_{s_3}(I_9)$ and hence $\rho(\mathcal{M}_{s_3})\ge3$ since $\rho(\mathcal{M}_{\frac13})=2$ for any $s_3\in (\frac14,\frac13)$.
Since $\mathcal{M}_{\frac13}$ is $\Q$-factorial, $\mathcal{M}_{s_3}$ is $\Q$-factorial and $\rho(\mathcal{M}_{s_3})=3$ Lemma \ref{lem:putting-coeffs-Hodge} again. Hence, $\psi_3$ is the divisorial contraction of $\mathcal{M}_{s_3}(I_9)$ by \cite[Corollary 2.63]{komo}.
\end{proof}

\begin{rem}
\label{rem:flipping contraction}
 By the above proof of Theorem \ref{thm:wallcrossing-miranda-to-8points}, we see that $\psi_1$ and $\phi_2$ are the flipping contractions along $\mathcal{M}_{s}(I_7)$, $\psi_2$ and $\phi_3$ are the flipping contractions along $\mathcal{M}_{s}(I_8')$ and $\mathcal{M}_{s}(I_8'')$.  

 On the other hand, comparing \cite[Theorem 9.3]{AscherBejleri2018StablePairRES} and Theorem \ref{thm:wallcrossing-miranda-to-8points}, we see that for any $t\in [0,1]$, $\mathcal{M}_t$ is different from $\mathcal{M}^{M*}$, which is the graph of the birational map from $\mathcal{M}^*$ to $\mathcal{M}^M$ introduced in \cite{HeckmanLooijenga2002}, or $R(\frac{1}{6})$ constructed in \cite{AscherBejleri2018StablePairRES}.
\end{rem}

\begin{lem}\label{lem:general-curves}
Notations as in Theorem \ref{thm:wallcrossing-miranda-to-8points}.
    There exist three proper smooth curves $C_1,C_2$ and $C_3\subset \mathcal{M}_{s_1}$ that do not intersect with the image of any boundary divisor of $\mathcal{M}^K$ except $\mathcal{M}_{s_1}(I_9)$,  $\mathcal{M}_{s_1}(I_{4,4}^*)$, $\mathcal{M}_{s_1}(I_{2})$ and $\mathcal{M}_{s_1}(II)$ and satisfy the following.
    \begin{enumerate}
\item $C_1$ does not intersect with $\mathcal{M}_{s_1}(I_9)$ or $\mathcal{M}_{s_1}(I_{4,4}^*)$ but $\mathcal{M}_{s_1}(I_{2})$ and $\mathcal{M}_{s_1}(II)$ transversally at least one point,
\item $C_2$ does not intersect with $\mathcal{M}_{s_1}(I_9)$ but  $\mathcal{M}_{s_1}(I_{4,4}^*)$, $\mathcal{M}_{s_1}(I_{2})$ and $\mathcal{M}_{s_1}(II)$ transversally at least one point, and 
\item $C_3$ intersects with $\mathcal{M}_{s_1}(I_9)$,  $\mathcal{M}_{s_1}(I_{4,4}^*)$, $\mathcal{M}_{s_1}(I_{2})$ and $\mathcal{M}_{s_1}(II)$ transversally at least one point.
    \end{enumerate}
\end{lem}

\begin{proof}
 It follows immediately from Theorem \ref{thm:wallcrossing-miranda-to-8points}, Lemma \ref{lem:first-step-wallcrossingdescription}, and the Bertini theorem but we will give a proof of (2) for the reader's convenience.
 We know that $\mathcal{M}_{s_4}$ is projective by Theorem \ref{thm:HM1:construction} for any $s_4\in(\frac13,1)$.
 Let $A$ be a very ample line bundle on $\mathcal{M}_{s_4}$.
 Take general hypersurfaces $H_1,\ldots, H_7 \in|A|$ and set $C:=H_1\cap\ldots\cap H_7$.
 Set $C_2$ as the strict transform of $C$ in $\mathcal{M}_{s_1}$.
 Note that $C$ avoids the exceptional locus of the birational map $\mathcal{M}_{s_4}\dashrightarrow\mathcal{M}_{s_1}$.
 Then we can check that our $C_2$ is the desired one by the fact that $\mathcal{M}_{s_4}(I_9)$ is not a divisor.
\end{proof}

We can set $\Lambda_{\mathrm{Hodge},s_1}=a_{s_1}\mathcal{M}_{s_1}(I_9)+b_{s_1}\mathcal{M}_{s_1}(I_{4,4}^*)+c_{s_1}\mathcal{M}_{s_1}(I_{2})$ for some $a_{s_1},b_{s_1},c_{s_1}\in \mathbb{R}$ by Lemma \ref{lem:putting-coeffs-Hodge}.
By \cite[Theorem 1.1]{bejleri2025proper}, we see that there exists a finite cover $\tilde{C}_1\to C_1$ such that $\tilde{C}_1\to \mathcal{M}_{s_1}$ factors through $\mathscr{M}_{s_1}$, where $C_1$ is as in Lemma \ref{lem:general-curves}.
Let $q_1\colon (\mathcal{C}_1,(1-s_1)\Delta)\to [\mathbb{A}^2_{\tilde{C}_1}/\mathbb{G}_{m,\tilde{C}_1}]$ with weight $\frac{s_1}{12}$ be the family of quasimaps defined by $\tilde{C}_1\to \mathscr{M}_{s_1}$.
For any $t\in (0,1)$, $q\colon (\mathcal{C}_1,(1-t)\Delta)\to [\mathbb{A}^2_{\tilde{C}_1}/\mathbb{G}_{m,\tilde{C}_1}]$ with weight $\frac{t}{12}$ defines a morphism $\tilde{C}_1\to \mathscr{M}_{t}$.
Note that $12\Delta\sim\mathscr{L}_{q_1}$, where $\mathscr{L}_{q_1}$ is the associated line bundle to $q_1$.
Therefore, if we let $\lambda_{\mathrm{Hodge},q_1,t}$ be the Hodge line bundle with respect to $q_1$ with weight $t\in (0,1)$, then $\mathrm{deg}\lambda_{\mathrm{Hodge},q_1,t}=\mathrm{deg}\lambda_{\mathrm{Hodge},q_1,t'}$ for any $t,t'\in (0,1)$.
By this, we see that $c_{s_1}$ is constant on $s_1$.
We can also set  $\Lambda_{\mathrm{Hodge},s_2}=a_{s_2}\mathcal{M}_{s_2}(I_9)+b_{s_2}\mathcal{M}_{s_2}(I_{4,4}^*)+c_{s_2}\mathcal{M}_{s_2}(I_{2})$ for some $a_{s_2},b_{s_2},c_{s_2}\in \mathbb{R}$, $\Lambda_{\mathrm{Hodge},s_3}=a_{s_3}\mathcal{M}_{s_3}(I_9)+b_{s_3}\mathcal{M}_{s_3}(I_{4,4}^*)+c_{s_3}\mathcal{M}_{s_3}(I_{2})$ for some $a_{s_3},b_{s_3},c_{s_3}\in \mathbb{R}$, and $\Lambda_{\mathrm{Hodge},s_4}=b_{s_4}\mathcal{M}_{s_4}(I_{4,4}^*)+c_{s_4}\mathcal{M}_{s_4}(I_{2})$ for some $b_{s_4},c_{s_4}\in \mathbb{R}$ in the same way.
Moreover, by the same argument, we see that $c_{s_1}=c_{s_2}=c_{s_3}=c_{s_4}$ are constants.
Furthermore, we can extend $a_{s}$, $b_{s}$ and $c_s$ by the continuity and hence we can define $a_{\frac{1}{7}}$ for example.

From now on, we determine the variation of the coefficients $a_{s}$ and $b_{s}$ in $s$ using $C_2$ and $C_3$ as in Lemma \ref{lem:general-curves}.

\begin{defn}\label{de-derivation-CM}
Let $0\in C$ be a smooth curve germ and fix $s\in (0,1)$. 
Suppose that there exists a family of quasimaps $q^{(s)}\colon( \mathcal{C}^{(s)},(1-s)\Delta^{(s)})\to [\mathbb{A}^2_{C}/\mathbb{G}_{m,C}]\in \mathscr{M}_s(C)$.
Let $\mathscr{L}^{(s)}$ be the line bundle associated to $q^{(s)}$.
%Suppose further that there exists a family of quasimaps $q^{(0)}\colon \mathcal{C}^{(0)}\to [\mathbb{A}^2_{C}/\mathbb{G}_{m,C}]\in \mathscr{M}_0(C)$ such that $q^{(0)}$ and $q^{(s)}$ are isomorphic to each other over $C\setminus\{0\}$ as a family of quasimaps.
%Let $\mathscr{L}^{(0)}$ (resp. $\mathscr{L}^{(s)}$) be the line bundle associated to $q^{(0)}$ (resp. $q^{(s)}$).
Suppose further that an arbitrary general fiber of $q^{(s)}$ defines a point of $\mathscr{M}$.
We set $$\mathrm{deg}_0\mathrm{Der}:=2\mathrm{ord}_{\mathcal{C}^{(s)}_0}\left(\frac{1}{12}D^{(s)}-\Delta^{(s)}\right).$$
%Suppose that there exists a morphism $\varphi^{(s)}\colon C\to \mathcal{M}_{s}$.
%For any sufficiently small $\epsilon>0$, we can set $\varphi^{(t)}\colon C\to \mathcal{M}_t$ as the morphism induced by $q^{(s)}$ and the birational contractions in Theorem \ref{thm:wallcrossing-miranda-to-8points} for any $t\in (s-\epsilon,s+\epsilon)$.
%By \cite{bejleri2025proper}, taking a finite cover $\xi\colon\tilde{C}\to C$, we may assume that $\tilde{C}\to C$ factors through $\mathscr{M}_t$ for any $t\in (s-\epsilon,s+\epsilon)$.
%Let $q^{(t)}\colon (\mathcal{C}^{(t)},(1-t)\Delta^{(t)})\to [\mathbb{A}^2_{\tilde{C}}/\mathbb{G}_{m,\tilde{C}}]$ be the corresponding family of quasimaps to the morphism $\tilde{C}\to \mathscr{M}_t$ and $\mathscr{L}^{(t)}$ the associated line bundle.
\end{defn}
Note that $12\Delta^{(s)}$ is generically isomorphic to the divisor $D^{(s)}\sim \mathscr{L}^{(s)}$, which is defined by the section corresponding to $(0,1)\in \mathbb{A}^2$.

\begin{lem}\label{lem:Hodge-difference-over-curve}
Let $s\in (0,1)$ and $q^{(s)}\colon (\mathcal{C}^{(s)},(1-s)\Delta^{(s)})\to [\mathbb{A}^2_C/\mathbb{G}_{m,C}]\in \mathscr{M}_s(C)$ an object such that $C$ is a smooth proper curve. 
Suppose that for any $t\in (0,s)$, $q^{(s)}\colon (\mathcal{C}^{(s)},(1-t)\Delta^{(s)})\to [\mathbb{A}^2_C/\mathbb{G}_{m,C}]$ with weight $\frac{t}{12}$ defines an object of $\mathscr{M}_t(C)$. 
Suppose further that any general fiber of $q^{(s)}$ defines a point of $\mathscr{M}$.
Let $q^{(t)}$ denote the object.
Then $\mathrm{deg}_C(\lambda_{\mathrm{Hodge},q^{(s)}})-\mathrm{deg}_C(\lambda_{\mathrm{Hodge},q^{(t)}})=(s-t)\sum_{c\in C}\mathrm{deg}_c\mathrm{Der}$, where $c\in C$ runs over all closed points.
\end{lem}

\begin{proof}
It immediately follows from Definition \ref{defn:Hogde-HM2} and the standard calculation.
\end{proof}

Applying the lemma to $C_2$ and $C_3$, we see the following, which is useful to determine $\mathrm{deg}_0(\mathrm{Der})$ for a curve germ:

\begin{prop}\label{prop:determining}
Let $0\in C$ be a smooth curve germ and fix $s\in (0,1)$ (resp. $s\in (0,\frac{1}{3})$). 
Suppose that there exists a family of quasimaps $q^{(s)}\colon( \mathcal{C}^{(s)},(1-s)\Delta^{(s)})\to [\mathbb{A}^2_{C}/\mathbb{G}_{m,C}]\in \mathscr{M}_s(C)$. Let $\varphi\colon C\to \mathcal{M}_s$ be the induced morphism and suppose that $\varphi(0)$ is a general point of $\mathcal{M}_s(I^*_{4,4})$ (resp. $\mathcal{M}_s(I_{9})$) and $\varphi(C\setminus\{0\})\subset \mathcal{M}$.  Furthermore, suppose that $\varphi(C)$ and $\mathcal{M}_s(I^*_{4,4})$ (resp. $\mathcal{M}_s(I_{9})$) intersect transversally. 
Then $b_s-b_0=s\frac{\mathrm{deg}_0(\mathrm{Der})}{\mu}$ for any $s\in(0,1)$ (resp. $a_s-a_0=s\frac{\mathrm{deg}_0(\mathrm{Der})}{\mu}$ for any $s\in(0,\frac{1}{3})$), where $\mu$ is the local degree of $C\to \varphi(C)$ at $0$.
\end{prop}

\begin{proof}
First, we deal with $b_s$.    
By \cite[Theorem 1.1]{bejleri2025proper}, we see that there exists a finite morphism $\xi_2\colon\tilde{C}_2\to C_2$ such that $\xi_2$ factors through $\mathscr{M}_{s_1}$ for some $s_1\in (0,\frac{1}{7})$. 
By the choice of $C_2$, replacing $\tilde{C}_2$ if necessary, we may assume that there exists a family of quasimaps $q^{(s_1)}\colon (\mathcal{C}^{(s_1)},(1-s_1)\Delta^{(s_1)})\to [\mathbb{A}^2_{\tilde{C}_2}/\mathbb{G}_{m,\tilde{C}_2}]\in \mathscr{M}_{s_1}(\tilde{C}_2)$ such that letting $q^{(t)}$ be $q^{(s_1)}\colon (\mathcal{C}^{(s_1)},(1-t)\Delta^{(s_1)})\to [\mathbb{A}^2_{\tilde{C}_2}/\mathbb{G}_{m,\tilde{C}_2}]$ with weight $\frac{t}{12}$, $q^{(t)}$ defines an object of $\mathscr{M}_t(\tilde{C}_2)$ for any $t\in (0,1)$.
Let $\{c_1,\ldots,c_l\}=\xi_2^{-1}(\mathcal{M}_{s_1}(I^*_{4,4}))$ with the local degree $\mu_i$ of $\xi_2$ at $c_i$ and note that the induced morphism $\tilde{C}_2\to \mathcal{M}_{t}$ is isomorphic locally around $\{c_1,\ldots,c_l\}$.
Since $c_s$ is constant, we see that 
\[
\sum_{i=1}^l\mu_i(b_s-b_0)=s\sum_{i=1}^l\mathrm{deg}_{c_i}(\mathrm{Der}).
\]
Therefore, it suffices to show that $\frac{\mathrm{deg}_{c_i}(\mathrm{Der})}{\mu_i}$ is independent of the choice of $i$.  
To consider this, we set the following $\mathbb{Q}$-divisor $\mathcal{D}_s$ on $\mathcal{M}_s$ for any $0<s<1$. 
Set $\mathscr{M}_{s}'\subset\mathscr{M}_s$ as the complement of the inverse image of the union of all $\mathcal{M}_s(\bullet)$ except ones of codimension one in $\mathcal{M}_s$.
For $s\ne0,\frac{1}{7},\frac{1}{4},\frac{1}{3}$ or $1$, we see that $\lambda_{\mathrm{Hodge},s}$ is affine linear locally around on $s$.
Let $\mathscr{D}_s$ be $\frac{d}{ds'}|_{s'=s}\lambda_{\mathrm{Hodge},s'}$ for such $s$.
It is easy to show that $\mathscr{D}_s|_{\mathscr{M}'_s}$ is a $\mathbb{Q}$-Cartier $\Q$-divisor supported only on $\mathscr{M}_s(I_9)\cup\mathscr{M}_s(I_{4,4}^*)$ for $0<s<\frac{1}{3}$ and $\mathscr{M}_s(I_{4,4}^*)$ for $\frac{1}{3}<s<1$ and the coefficients of $\mathscr{D}_s$ along $\mathscr{M}_s(I_9)$ for $0<s<\frac{1}{3}$ and $\mathscr{M}_s(I_{4,4}^*)$ for $0<s<1$ are independent of the choice of $s$ if $s\ne0,\frac{1}{7},\frac{1}{4},\frac{1}{3}$ or $1$.
Indeed, let $C$ be as in the statement for $\mathcal{M}_s(I_9)$ and $\mathcal{M}_s(I_{4,4}^*)$.
We see that $\mathscr{D}_s$ can be descended to the unique $\mathbb{Q}$-Cartier $\mathbb{Q}$-divisor $\mathcal{D}_s$ on $\mathcal{M}_{s}$ for $s\ne0,\frac{1}{7},\frac{1}{4},\frac{1}{3}$ or $1$. 
Then $\mathrm{deg}_0(\mathrm{Der})$ is nothing but the local degree of $\varphi^*\mathcal{D}_s$ at $0$ by Lemma \ref{lem:Hodge-difference-over-curve}. 
By this fact, we can check that the coefficients of $\mathscr{D}_s$ along $\mathscr{M}_s(I_9)$ for $0<s<\frac{1}{3}$ and $\mathscr{M}_s(I_{4,4}^*)$ for $0<s<1$ are independent of the choice of $s$.
For $s=0,\frac{1}{7},\frac{1}{4},\frac{1}{3}$ or $1$, we set $\mathcal{D}_s$ as the strict transform of $\mathcal{D}_{s'}$ for any $0<s'<\frac{1}{7}$.
We are only interested in the coefficients of $\mathcal{D}_s$ along $\mathcal{M}_s(I_9)$ and $\mathcal{M}_s(I_{4,4}^*)$ and we will write $\mathcal{D}$ and $\mathscr{D}$ for simplicity.
 We saw that $\mathrm{deg}_{c_i}(\mathrm{Der})$ is nothing but the local degree of the pullback of $\mathscr{D}$ to $\tilde{C}_2$ at $c_i$.
Therefore, it is easy to see that $\frac{\mathrm{deg}_{c_i}(\mathrm{Der})}{\mu_i}$ is independent of the choice of $i$ and equals to $\frac{\mathrm{deg}_0(\mathrm{Der})}{\mu}$ for $C$ as the statement. Indeed, $\frac{\mathrm{deg}_{c_i}(\mathrm{Der})}{\mu_i}$ coincides with the coefficient of $\mathcal{D}$ along $\mathcal{M}_s(I_{4,4}^*)$. We complete the proof for $b_s$. 

%Note that $\lambda_{\mathrm{Hodge},s'}$ can be regarded as a $\mathbb{R}$-line bundle on $\mathscr{M}_{s}$ for any $s'\in (s-\epsilon,s+\epsilon)$ for any sufficiently small $\epsilon>0$ because if $\mathscr{M}_{s}\not\cong \mathscr{M}_{s'}$, we can pullback the Hodge line bundle $\lambda_{\mathrm{Hodge},s'}$ via the canonical morphism $\mathscr{M}_{s}\to\mathscr{M}_{s'}$.
%Note also that $\lambda_{\mathrm{Hodge},s'}$ depends locally on $s'$ as a rational affine function if $s'\ne0,\frac{1}{7},\frac{1}{4},\frac{1}{3}$ or $1$.
%Let $\mathscr{D}$ be the $\mathbb{Q}$-Cartier $\mathbb{Q}$-divisor on $\mathscr{M}_{s}$ which is defined as $\frac{d}{ds'}|_{s'=s}\lambda_{\mathrm{Hodge},s'}$.
%By the definition of $\mathscr{D}$, we see that $\mathscr{D}$ is only supported on $\mathscr{M}_{s}(I_{4,4}^*)$ and $\mathscr{M}_{s}(I_{9})$.

For $C_3$, we can proceed the argument as above and hence we obtain the assertion for $a_s$ as follows:
By \cite{bejleri2025proper}, we see that there exists a finite morphism $\xi_3\colon\tilde{C}_3\to C_3$ such that $\xi_3$ factors through $\mathscr{M}_{s_1}$ for some $s_1\in (0,\frac{1}{7})$. 
By the choice of $C_3$, replacing $\tilde{C}_3$ if necessary, we may assume that there exists a family of quasimaps $q^{(s_1)}\colon (\mathcal{C}^{(s_1)},(1-s_1)\Delta^{(s_1)})\to [\mathbb{A}^2_{\tilde{C}_3}/\mathbb{G}_{m,\tilde{C}_3}]\in \mathscr{M}_{s_1}(\tilde{C}_3)$ such that letting $q^{(t)}$ be $q^{(s_1)}\colon (\mathcal{C}^{(s_1)},(1-t)\Delta^{(s_1)})\to [\mathbb{A}^2_{\tilde{C}_3}/\mathbb{G}_{m,\tilde{C}_3}]$ with weight $\frac{t}{12}$, $q^{(t)}$ defines an object of $\mathscr{M}_t(\tilde{C}_3)$ for any $t\in (0,\frac{1}{3})$.
Let $\{c_1,\ldots,c_l\}=\xi_3^{-1}(\mathcal{M}_{s_1}(I^*_{4,4}))$ and $\{c'_1,\ldots,c'_{l'}\}=\xi_3^{-1}(\mathcal{M}_{s_1}(I_9))$ with the local degree $\mu_i$ (resp. $\mu'_i$) of $\xi_3$ at $c_i$ (resp. $c'_i$) and note that the induced morphism $\tilde{C}_3\to \mathcal{M}_{t}$ is isomorphic locally around $\{c_1,\ldots,c_l\}$ and $\{c'_1,\ldots,c'_{l'}\}$.
Since $c_s$ is constant, we see that 
\[
\sum_{i=1}^l\mu_i(b_s-b_0)+\sum_{i=1}^{l'}\mu'_i(a_s-a_0)=s\sum_{i=1}^l\mathrm{deg}_{c_i}(\mathrm{Der})+s\sum_{i=1}^{l'}\mathrm{deg}_{c'_i}(\mathrm{Der}).
\]
As above, we see that $\frac{\mathrm{deg}_{c_i}(\mathrm{Der})}{\mu_i}$ and $\frac{\mathrm{deg}_{c'_i}(\mathrm{Der})}{\mu'_i}$ are independent of $i$ because $\frac{\mathrm{deg}_{c'_i}(\mathrm{Der})}{\mu'_i}$ is nothing but the coefficient of $\mathcal{D}$ along $\mathcal{M}_s(I_9)$.
Combining this with the conclusion we obtained in the previous paragraph, we obtain the proof.
\end{proof}

Therefore, we can calculate the variations of $a_s$ and $b_s$ using a general curve germ that intersects with $\mathcal{M}_s(I^*_{4,4})$ and $\mathcal{M}_s(I_{9})$ at a general point.
We will construct a general curve germ. 

Before construction, we will collect the following terminologies.
A rational Weierstrass fibration is written as
\begin{equation}\label{eq:weierstrass}
 y^2z=x^3+A(u,v)xz^2+B(u,v)z^3,
 \qquad A\in H^0(\mathbb{P}^1,\mathcal{O}(4)),\quad B\in H^0(\mathbb{P}^1,\mathcal{O}(6)).
\end{equation}
Its Weierstrass discriminant is the binary duodecic
\[
 \DWei(A,B):=4A^3+27B^2\in H^0(\mathbb{P}^1,\mathcal{O}(12)).
\]
If a binary duodecic is written as
\[
 F(U,W)=a\prod_{i=1}^{12}(\alpha_iU+\beta_iW),
\]
its classical discriminant is
\[
 \Disc_{12}(F)
 =a^{22}\prod_{i<j}(\alpha_i\beta_j-\alpha_j\beta_i)^2.
\]
Thus $\Disc_{12}$ is homogeneous of degree $22=2(12-1)$ in the coefficients of $F$.  The number $132=2\binom{12}{2}$ is instead the affine root-coordinate degree, and $264=12\cdot 22$ is the total degree in twelve pairs of homogeneous point coordinates.  These three degrees must not be interchanged.

\begin{ex}\label{ex:I9}

Let
\[
 p(U,W)\in H^0(\mathbb{P}^1,\mathcal{O}(3)),\qquad
 c(U,W)\in H^0(\mathbb{P}^1,\mathcal{O}(6))
\]
be general, with $p(0,1)\neq0$. 
Define
\[
 \begin{aligned}
 A_\tau&=-3U^4+\tau Up,\\
 B_\tau &= 2U^6-\tau U^3p+\frac{\tau^2}{12}p^2+\tau^3c,
 \end{aligned}
\]
where $\tau$ is the parameter of the curve germ.
A direct expansion gives
\[
 4A_\tau ^3+27B_\tau^2=\tau^3F_\tau,
\]
where
\begin{equation}\label{eq:I9-Ft}
 \begin{aligned}
 F_\tau={}&108U^6c-\frac12U^3p^3\\
 &+\tau\left(\frac{3}{16}p^4-54U^3pc\right)
 +\frac92\tau^2cp^2+27\tau^3c^2.
 \end{aligned}
\end{equation}
Consequently,
\[
 \ord_{\tau=0}(4A_\tau^3+27B_\tau^2)=3.
\]
Here the order means the vanishing order along the central fiber.

At $\tau=0$,
\[
 F_0=\frac12U^3G_9,
 \qquad
 G_9:=216U^3c-p^3.
\]
For general $p,c$, $G_9$ is squarefree and $G_9(0,1)=-p(0,1)^3\neq0$.  The central twelve points consist of one point of multiplicity three and nine distinct points.
This is the generic configuration on $\mathcal M^*(I_9)=\Delta(6)$.

The construction covers the general point of that boundary.  Indeed, for a general binary nonic
\[
 G(U,W)=g_0W^9+g_1UW^8+g_2U^2W^7+\cdots,
 \qquad g_0\neq0,
\]
one may choose the first three coefficients of a cubic $p$ so that $G+p^3$ is divisible by $U^3$, and then set
\[
 c=\frac{G+p^3}{216U^3}.
\]
Thus the map $(p,c)\mapsto216U^3c-p^3$ dominates the space of general nonics with nonzero $W^9$-coefficient.

Set $W=1$, $x=U/W$, and $p_0=p(0,1)$.  From \eqref{eq:I9-Ft},
\begin{equation}\label{eq:I9-local}
 F_\tau(x,1)
 =-\frac12p_0^3x^3+\frac{3}{16}p_0^4\tau
 +O(x^4,\tau x,\tau^2).
\end{equation}
After multiplication by a unit, this has the local form $x^3-\kappa \tau+\cdots$, with $\kappa\neq0$.  The three colliding roots are of order $\tau^{1/3}$, and the product of their three squared pairwise differences has order $2$.  All other brackets are units.  Hence
\[
 \ord_{\tau=0}\Disc_{12}(F_\tau)=2,
 \qquad
 (C_{I_9}\cdot\Disc_{12})_0=2,
\]
where $C_{I_9}$ is the corresponding curve germ.
The coefficient of $t$ in \eqref{eq:I9-local} is nonzero, so the family is transverse to the reduced boundary and its local degree is $\mu=1$ of $\mathcal{M}_s(I_9)$.

As divisors on the total space,
\[
 \Div(4A_\tau^3+27B_\tau^2)=3\mathcal C_0+\Div(F_\tau),
 \qquad \mathcal C_0=\{\tau=0\}.
\]
Therefore, we obtain
\[
 \deg_0\operatorname{Der}=2\cdot\frac{3}{12}=\frac12.
\]

\end{ex}

\begin{ex}\label{ex:I44*}

Put $q=uv$, and take $a=u^4+u^3v+v^4$ and  $c=u^6+u^5v+v^6$.
Define $A_\tau=-3q^2+\tau a$ and $ B_\tau=2q^3-\tau qa+\tau^2c$.
Then $4A_\tau^3+27B_\tau^2=\tau^2F_\tau$ with
\[
 F_\tau=9q^2(12qc-a^2)
 +\tau(4a^3-54qac)+27\tau^2c^2.\]
Thus the vertical order is $2$.  At $\tau=0$, we have 
 $F_0=9u^2v^2G_8$ and $G_8=12uvc-a^2$.
For the chosen forms,
\begin{align*}
 G_8=&-u^8+10u^7v+11u^6v^2-2u^4v^4-2u^3v^5
 +12uv^7-v^8,
\end{align*}
and a direct gcd calculation shows that $G_8$ is squarefree and nonzero at $[1:0]$ and $[0:1]$.  Hence, the twelve points in the central fiber consist of two distinct points of multiplicity two and eight distinct points.
Near $u=0$, with $v=1$ and $x=u$,
\[
 F_\tau(x,1)=-9x^2+4\tau+\text{higher terms}.
\]
This double point contributes order one to $\Disc_{12}$, and the double point at $v=0$ contributes another order one.  Therefore $\ord_{\tau=0}\Disc_{12}(F_\tau)=2$ and $(C_{I_{4,4}^{*}}\cdot\Disc_{12})_0=2$,
where $C_{I_{4,4}^{*}}$ is the corresponding curve germ.
The asymmetric choice of $a,c$ removes a generic exchange symmetry, and $\tau$ is a transverse coarse parameter, so the local degree $\mu=1$ of $\mathcal{M}_s(I_{4,4}^{*})$.  Consequently,
\[
 \deg_0\operatorname{Der}=2\cdot\frac{2}{12}=\frac13\]
\end{ex}

Combining the two examples above, we obtain the following conclusion.
\begin{thm}\label{thm:Hodge-line-bundle-stricttrans}
 For any $t\in [0,1]$, $\Lambda_{\mathrm{Hodge},t}$ on $\mathcal{M}_t$ is the strict transform of $\Lambda_{\mathrm{Hodge},0}+\frac{t}{2}\Delta(6)+\frac{t}{3}\Delta(9)$ from $\mathcal{M}_0$.   
\end{thm}

\begin{proof}
Let $\xi\colon \mathcal{M}_0\dashrightarrow\mathcal{M}_t$ be the canonical birational contraction.
    Then, $\Lambda_{\mathrm{Hodge},t}=(a_{t}-a_0)\mathcal{M}_{t}(I_9)+(b_{t}-b_0)\mathcal{M}_{t}(I_{4,4}^*)+\xi_*\Lambda_{\mathrm{Hodge},0}$ as we set $a_t$ for $t\in [0,\frac13]$ and $b_t$ as we set above. Set $a_t:=a_0$ for $t\ge\frac13$.
    We note that $a_t-a_0$ and $b_t-b_0$ can be deduced as $\frac{t}2$ for $t\in [0,\frac13)$ and $\frac{t}{3}$, respectively. Note also that $\mathcal{M}_t(I_9)$ is not a divisor for $t\ge\frac13$.
    Therefore, the assertion follows from Examples \ref{ex:I9} and \ref{ex:I44*} and Proposition \ref{prop:determining}.
    % \Yota{Don't we need to use $a_t-a_0$ and $b_t-b_0$?} \Masa{I believe that I managed to observe your order.}
\end{proof}

\begin{cor}
For any $t\in (0,\frac{1}{7})$, $\mathcal{M}_t\cong X_o^{j_{\mathcal K}}$. \label{cor-first-model}   
\end{cor}

\begin{proof}
  Let $c\colon X_o^{j_{\mathcal K}}\to \mathcal{M}_0$ be the canonical morphism.
  By Corollary \ref{cor:cusp-intersection-matrix-corrected}, $3\Delta(6)^j+2\Delta(9)^j$ is $c$-ample.
  Therefore, for any sufficiently small $\epsilon>0$, $\Lambda_{\mathrm{Hodge},0}+\epsilon(3\Delta(6)^j+2\Delta(9)^j)$ is ample on $X_o^{j_{\mathcal K}}$ by Theorem \ref{thm:Hodge-line-bundle-stricttrans}.
  Since $X_o^{j_{\mathcal K}}$, $\mathcal{M}_0$ and $\mathcal{M}_{6\epsilon}$ are isomorphic in codimension one and $\Lambda_{\mathrm{Hodge},0}+\epsilon(3\Delta(6)^j+2\Delta(9)^j)$
  is the strict transform of $\Lambda_{\mathrm{Hodge},6\epsilon}$, we see that $X_o^{j_{\mathcal K}}\cong \mathcal{M}_{6\epsilon}$ by \cite[Theorem 11.39]{kollar-moduli}.
  Therefore, the conclusion follows from Theorem \ref{thm:wallcrossing-miranda-to-8points}.
\end{proof}

\begin{proof}[Proof of Theorem \ref{thm:main-HKL}]
Since $X_o^{j_{\mathcal K}}\setminus X_{o}$ is of codimension at least two and $X_o^{j_{\mathcal K}}$ is normal, we see that
\[
R\left(X_o,\mathcal{L}+\frac{t}{2}\Delta(6)+\frac{t}{3}\Delta(9)\right)=R\left(X^{j_{\mathcal K}}_o,\mathcal{L}^j+\frac{t}{2}\Delta(6)^j+\frac{t}{3}\Delta(9)^j\right),
\]
where $\mathcal{L}^j$ is the pullback of $\mathcal{L}$ on $X_o^{\mathrm{bb}}$ for any $t\in \mathbb{Q}\cap[0,1]$.
By Theorem \ref{thm:wallcrossing-miranda-to-8points} and Corollary \ref{cor-first-model}, $X^{j_{\mathcal K}}_o\cong\mathcal{M}_t$ for any $0<t<\frac17$ and note that $\mathcal{M}_{t'}$ is isomorphic to $X^{j_{\mathcal K}}_o$ in codimension one for any $0\le t'<\frac{1}{3}$.
By Theorem \ref{thm:Hodge-line-bundle-stricttrans}, we see that the ample $\mathbb{R}$-line bundle $\Lambda_{\mathrm{Hodge},t}$ is the strict transform of $\mathcal{L}^j+\frac{t}{2}\Delta(6)^j+\frac{t}{3}\Delta(9)^j$ for any $t\in [0,1]$.
Thus, we can conclude that $\mathrm{Proj}R\left(X^{j_{\mathcal K}}_o,\mathcal{L}^j+\frac{t}{2}\Delta(6)^j+\frac{t}{3}\Delta(9)^j\right)\cong \mathcal{M}_t$ for any $0\le t\le\frac13$.
Note that $\mathcal{M}_{t}$ for any $\frac13\le t<1$ is the contraction of the strict transform of $\Delta(6)^j$ in $\mathcal{M}_{\frac{7}{24}}$ (see Theorem \ref{thm:wallcrossing-miranda-to-8points}).
We also note that the strict transform $\mathcal{M}_{\frac{7}{24}}(I_{4,4}^*)$ of  $\Delta(9)^j$ is disjoint from the strict transform $\mathcal{M}_{\frac{7}{24}}(I_9)$ of $\Delta(6)^j$ in $\mathcal{M}_{\frac{7}{24}}$.
Let $\mathcal{L}_{\frac{7}{24}}$ be the strict transform of $\mathcal{L}^j$ in $\mathcal{M}_{\frac{7}{24}}$.
To show the assertion, it suffices to show that the ample model of $\mathcal{L}_{\frac{7}{24}}+\frac{t}{2}\mathcal{M}_{\frac{7}{24}}(I_9)+\frac{t}{3}\mathcal{M}_{\frac{7}{24}}(I_{4,4}^*)$ coincides with $\mathcal{M}_{\frac{1}{3}}$ for any $\frac13\le t<1$ and the ample model of $\mathcal{L}_{\frac{7}{24}}+\frac{t}{2}\mathcal{M}_{\frac{7}{24}}(I_9)+\frac{t}{3}\mathcal{M}_{\frac{7}{24}}(I_{4,4}^*)$ for any $t\ge1$ coincides with $\mathcal{M}_1$ in the sense of \cite[Definition 3.6.5]{birkar2010existence}.
By the ampleness of the Hodge line bundle on each $\mathcal{M}_s$, it is equivalent to show that the stable base locus $\mathbf{B}(\mathcal{L}_{\frac{7}{24}}+\frac{t}{2}\mathcal{M}_{\frac{7}{24}}(I_9)+\frac{t}{3}\mathcal{M}_{\frac{7}{24}}(I_{4,4}^*))=\mathcal{M}_{\frac{7}{24}}(I_9)$ for any $\frac13< t\le1$ and $\mathbf{B}(\mathcal{L}_{\frac{7}{24}}+\frac{t}{2}\mathcal{M}_{\frac{7}{24}}(I_9)+\frac{t}{3}\mathcal{M}_{\frac{7}{24}}(I_{4,4}^*))=\mathcal{M}_{\frac{7}{24}}(I_9)\cup\mathcal{M}_{\frac{7}{24}}(I_{4,4}^*)$ for any $t>1$. 
This follows from the fact that $\mathcal{L}_{\frac{7}{24}}+\frac{1}{6}\mathcal{M}_{\frac{7}{24}}(I_9)+\frac{1}{9}\mathcal{M}_{\frac{7}{24}}(I_{4,4}^*)$ is numerically trivial along $\mathcal{M}_{\frac{7}{24}}(I_9)$, $\mathcal{L}_{\frac{7}{24}}+\frac{1}{2}\mathcal{M}_{\frac{7}{24}}(I_9)+\frac{1}{3}\mathcal{M}_{\frac{7}{24}}(I_{4,4}^*)$ is numerically trivial along $\mathcal{M}_{\frac{7}{24}}(I_{4,4}^*)$, $-\mathcal{M}_{\frac{7}{24}}(I_9)$ is $\psi_3$-ample and $-\mathcal{M}_{\frac{1}{3}}(I_{4,4}^*)$ is $\psi_4$-ample.
Since
\[
R\left(X^{j_{\mathcal K}}_o,\mathcal{L}^j+\frac{t}{2}\Delta(6)^j+\frac{t}{3}\Delta(9)^j\right)\cong R\left(\mathcal{M}_{\frac{7}{24}},\mathcal{L}_{\frac{7}{24}}+\frac{t}{2}\mathcal{M}_{\frac{7}{24}}(I_9)+\frac{t}{3}\mathcal{M}_{\frac{7}{24}}(I_{4,4}^*)\right),
\]
 we see that
\begin{align*}
R\left(\mathcal{M}_{\frac{7}{24}},\mathcal{L}_{\frac{7}{24}}+\frac{t}{2}\mathcal{M}_{\frac{7}{24}}(I_9)+\frac{t}{3}\mathcal{M}_{\frac{7}{24}}(I_{4,4}^*)\right)&\cong R\left(\mathcal{M}_{\frac{7}{24}},\mathcal{L}_{\frac{7}{24}}+\frac{1}{6}\mathcal{M}_{\frac{7}{24}}(I_9)+\frac{t}{3}\mathcal{M}_{\frac{7}{24}}(I_{4,4}^*)\right)\\
&\cong R\left(\mathcal{M}_{t},\Lambda_{\mathrm{Hodge},t}\right)
\end{align*}
for any rational $\frac{1}{3}\le t\le 1$ and 
\begin{align*}
R\left(\mathcal{M}_{\frac{7}{24}},\mathcal{L}_{\frac{7}{24}}+\frac{t}{2}\mathcal{M}_{\frac{7}{24}}(I_9)+\frac{t}{3}\mathcal{M}_{\frac{7}{24}}(I_{4,4}^*)\right)&\cong R\left(\mathcal{M}_{\frac{7}{24}},\mathcal{L}_{\frac{7}{24}}+\frac{1}{6}\mathcal{M}_{\frac{7}{24}}(I_9)+\frac{1}{3}\mathcal{M}_{\frac{7}{24}}(I_{4,4}^*)\right)\\
&\cong R\left(\mathcal{M}_{1},\Lambda_{\mathrm{Hodge},1}\right)
\end{align*}
for any rational $t\ge1$.

Note that $R\left(X^{j_{\mathcal K}}_o;\mathcal{L}^j,\mathcal{L}^j+\frac{1}{14}\Delta(6)^j+\frac{1}{21}\Delta(9)^j\right)$ is finitely generated because $\mathcal{L}^j$ and $\mathcal{L}^j+\frac{1}{14}\Delta(6)^j+\frac{1}{21}\Delta(9)^j$ are semiample.
Note that 
\begin{align*}
    &R\left(X^{j_{\mathcal K}}_o;\mathcal{L}^j+\frac{1}{6}\Delta(6)^j+\frac{1}{9}\Delta(9)^j,\mathcal{L}^j+\frac{1}{2}\Delta(6)^j+\frac{1}{3}\Delta(9)^j\right) \\
    \cong& R\left(\mathcal{M}_{\frac{7}{24}};\mathcal{L}_{\frac{7}{24}}+\frac{1}{6}\mathcal{M}_{\frac{7}{24}}(I_9)+\frac{1}{9}\mathcal{M}_{\frac{7}{24}}(I_{4,4}^*),\mathcal{L}_{\frac{7}{24}}+\frac{1}{6}\mathcal{M}_{\frac{7}{24}}(I_9)+\frac{1}{3}\mathcal{M}_{\frac{7}{24}}(I_{4,4}^*)\right)
\end{align*} by what we have shown.
Since $\mathcal{L}_{\frac{7}{24}}+\frac{1}{6}\mathcal{M}_{\frac{7}{24}}(I_9)+\frac{1}{9}\mathcal{M}_{\frac{7}{24}}(I_{4,4}^*)$ and $\mathcal{L}_{\frac{7}{24}}+\frac{1}{6}\mathcal{M}_{\frac{7}{24}}(I_9)+\frac{1}{3}\mathcal{M}_{\frac{7}{24}}(I_{4,4}^*)$ are semiample, $R\left(\mathcal{M}_{\frac{7}{24}};\mathcal{L}_{\frac{7}{24}}+\frac{1}{6}\mathcal{M}_{\frac{7}{24}}(I_9)+\frac{1}{9}\mathcal{M}_{\frac{7}{24}}(I_{4,4}^*),\mathcal{L}_{\frac{7}{24}}+\frac{1}{6}\mathcal{M}_{\frac{7}{24}}(I_9)+\frac{1}{3}\mathcal{M}_{\frac{7}{24}}(I_{4,4}^*)\right)$ is finitely generated.
Moreover, $$R\left(X^{j_{\mathcal K}}_o;\mathcal{L}^j+\frac{1}{14}\Delta(6)^j+\frac{1}{21}\Delta(9)^j,\mathcal{L}^j+\frac{1}{8}\Delta(6)^j+\frac{1}{12}\Delta(9)^j\right)$$ and $$R\left(X^{j_{\mathcal K}}_o;\mathcal{L}^j+\frac{1}{8}\Delta(6)^j+\frac{1}{12}\Delta(9)^j,\mathcal{L}^j+\frac{1}{6}\Delta(6)^j+\frac{1}{9}\Delta(9)^j\right)$$ are finitely generated by the same reasoning.
Therefore, $R\left(X^{j_{\mathcal K}}_o;\mathcal{L}^j,\mathcal{L}^j+\frac{1}{2}\Delta(6)^j+\frac{1}{3}\Delta(9)^j\right)$ is finitely generated.
We complete the proof.
\end{proof}

\begin{rem}
The birational transformation
$\mathcal{M}^*\dashrightarrow\mathcal{M}^M$
is obtained by varying the ample model of
\[
 A_t:=(1-t)\mathcal{L}
 +t\left(
 \mathcal{L}+\frac12\Delta(6)+\frac13\Delta(9)
 \right)
\]
from $t=0$ to $t=1$.
In particular, the endpoint divisor is $A_1$.
By Lemma \ref{lem:putting-coeffs-Hodge}, it is not hard to check that $\mathcal{M}_t$ is $\mathbb{Q}$-factorial if and only if $t\ne0,\frac17$ or $\frac14$.
\end{rem}

\subsection*{Acknowledgments}
The authors thank Jan Hendrik Bruinier, Yuki Koto, Radu Laza, 
Yuji Odaka, Dali Shen, and Takuya Yamauchi for valuable discussions and useful comments. 
We thank Klaus Hulek for pointing out the relation with
the work of Odaka--Oshima \cite{OO21}.
The bulk of this work was carried
out during M.H.'s stay at Technische Universit\"at Darmstadt and Y.M.'s
stay at the University of Nottingham. 
The authors used ChatGPT (OpenAI) during the preparation of this manuscript for language and expository editing, and for exploratory assistance in finding the explicit families used in Examples \ref{ex:I9} and \ref{ex:I44*}. The final constructions, computations, and mathematical arguments were independently checked by the authors, who take full responsibility for the contents of the paper.
M.H. is supported by Royal
Society International Collaboration Award
ICA\textbackslash1\textbackslash231019.  Y.M. is partially supported by
the Alexander von Humboldt Foundation through a Humboldt Research
Fellowship, by the Deutsche Forschungsgemeinschaft (DFG, German
Research Foundation) through the Collaborative Research Centre TRR 326
\emph{Geometry and Arithmetic of Uniformized Structures}, project
number 444845124, and by Waseda University Grant for Special Research Projects (Project number: 2026R-022).

\bibliographystyle{alpha}
\bibliography{main}
\end{document}